\documentclass[11pt,reqno]{amsart}

\usepackage[T1]{fontenc}
\usepackage{lmodern}
\usepackage[margin=1.06in]{geometry}
\usepackage{microtype}
\usepackage{mathtools}
\usepackage{amssymb}
\usepackage{enumitem}
\usepackage{booktabs}
\usepackage{tabularx}
\usepackage{xcolor}
\usepackage[colorlinks=true,
 linkcolor=blue!55!black,citecolor=blue!55!black,
 urlcolor=blue!65!black]{hyperref}
\hypersetup{
 pdftitle={Pr\"ufer 2-group and Milnor's conjecture on fundamental groups},
 pdfauthor={Nan Wu and Zetian Yan},
 pdfsubject={Counterexamples to the Milnor conjecture in dimensions at least four},
 pdfkeywords={positive Ricci curvature, fundamental groups, Milnor conjecture,
 Pruefer group, complete noncompact manifolds}}

\newtheorem{theorem}{Theorem}[section]
\newtheorem{lemma}[theorem]{Lemma}
\newtheorem{proposition}[theorem]{Proposition}
\newtheorem{corollary}[theorem]{Corollary}
\theoremstyle{definition}
\newtheorem{definition}[theorem]{Definition}
\newtheorem{remark}[theorem]{Remark}
\numberwithin{equation}{section}

\DeclareMathOperator{\Ric}{Ric}
\DeclareMathOperator{\Rm}{Rm}
\DeclareMathOperator{\II}{II}
\DeclareMathOperator{\Hess}{Hess}
\DeclareMathOperator{\tr}{tr}
\DeclareMathOperator{\divsymb}{div}
\DeclareMathOperator{\sgn}{sgn}
\DeclareMathOperator{\Id}{Id}
\DeclareMathOperator{\Sym}{Sym}
\DeclareMathOperator{\vspan}{span}
\newcommand{\dd}{\mathop{}\!\mathrm d}
\newcommand{\rhoCab}{{\rho_{\mathrm{cab}}}}
\newcommand{\CPrufer}{C_{2^\infty}}

\title[Counterexamples to Milnor's Conjecture]
{Pr\"ufer $2$-group and Milnor's Conjecture on Fundamental Groups}

\author[N. Wu]{Nan Wu}
\address[N. Wu]{Department of Mathematics \& IMS \\ Nanjing University \\ Nanjing \\ 210093 \\ China}
\email{nwu2@nju.edu.cn}
\author[Z. Yan]{Zetian Yan}
\address[Z. Yan]{Department of Mathematics \\ The Chinese University of Hong Kong \\ Shatin, N.T. \\ Hong Kong}
\email{zetianyan@cuhk.edu.hk}


\begin{document}
\raggedbottom

\begin{abstract}
We construct complete, one-ended Riemannian manifolds in dimensions four
and five having strictly positive Ricci curvature, with fundamental group
isomorphic to the Pr\"ufer $2$-group $C_{2^\infty}$ and any subgroup\/ $\Gamma\leq \mathbb Q/\mathbb Z$, respectively. This gives counterexamples
to Milnor's conjecture in the two remaining dimensions. 
\end{abstract}

\maketitle
\tableofcontents
\enlargethispage{4pt}

\section{Introduction}
\label{sec:statement}
Milnor conjectured that the fundamental group of every complete
Riemannian manifold with nonnegative Ricci curvature is finitely
generated \cite{Milnor1968}. In his original argument,  Milnor showed that
every finitely generated subgroup of the fundamental group has
polynomial growth. Gromov's theorem then implies that every such
subgroup is virtually nilpotent \cite{Gromov1981}.  Wilking
subsequently reduced the search for counterexamples to the case of
abelian fundamental groups, while the generalized Margulis lemma of
Kapovitch--Wilking provided dimensionally uniform control of the
underlying nilpotent structure \cite{KapovitchWilking2011, Wilking2000}.

On the affirmative side, the conjecture holds in dimension two by the
work of Cohn--Vossen \cite{CohnVossen1935} and in dimension three by
the work of Schoen--Yau \cite{SchoenYau1982}, Liu \cite{Liu2013}, and
Pan \cite{Pan2020}.  A parallel line of results established finite
generation under additional geometric hypotheses: Li \cite{Li1986} and
Anderson \cite{AndersonTopology1990} proved that Euclidean volume
growth forces the fundamental group to be finite, Sormani
\cite{SormaniDiameter} treated small linear diameter growth, and
subsequent work of Pan \cite{PanStability,PanAlmostStability} and Huang
\cite{HuangPolar} imposed various stability, conicality, and polarity
assumptions on the universal cover.  More recently, Huang--Huang
\cite{HuangHuang2025} proved finite generation in dimension four when
the universal cover has Euclidean volume growth, while
Navarro--Pan--Zhu \cite{NavarroPanZhu2024} obtained a stronger
structural description in the linear-volume-growth setting.

The picture changed decisively with the constructions of
Bru\`e--Naber--Semola \cite{BNS7,BNS6}.  Their seven-dimensional
examples realize every subgroup of \(\mathbb Q/\mathbb Z\) as the
fundamental group of a complete manifold with nonnegative Ricci
curvature, and their six-dimensional example has fundamental group
\(\mathbb Q/\mathbb Z\).  Their construction proceeds on the universal
cover through an equivariant smooth snowflake.  This left dimensions
four and five open, these cases are listed as Problem~5.15 in the new
problem list in low-dimensional topology \cite{BaykurKirbyRuberman2026}.

In this paper, we settle the remaining cases by constructing
counterexamples in dimensions four and five directly on the quotient.
Our manifolds in fact have strictly positive Ricci curvature, and their
fundamental group is the Pr\"ufer $2$-group.
\begin{definition}
We define the Pr\"ufer $2$-group as
\begin{equation*}
 \CPrufer
 =\varinjlim_j
 \left(
 \mathbb Z/2^{j+1}\mathbb Z
 \longrightarrow
 \mathbb Z/2^{j+2}\mathbb Z
 \right),
 \qquad
 [k]\longmapsto[2k].
\end{equation*}
Equivalently,
\begin{equation*}
 \CPrufer
 =
 \bigcup_{j\ge0}
 \frac{1}{2^{j+1}}\mathbb Z\big/\mathbb Z
 \subset \mathbb Q/\mathbb Z.
\end{equation*}
\end{definition}
Every finitely generated subgroup of $\CPrufer$ is contained in a finite
cyclic subgroup of order $2^N$ for some $N$, whereas $\CPrufer$ contains
elements of arbitrarily large order.  Consequently, $\CPrufer$ is not
finitely generated.

Our main theorem is the following.
\begin{theorem}\label{thm:main}
For $m\in\{2,3\}$, there exists a connected, orientable, one-ended
smooth manifold $M^{m+2}$ admitting a complete Riemannian metric $g$, such that
\begin{equation*}
 \Ric_g>0,
 \qquad
 \pi_1(M^{m+2})\cong\CPrufer.
\end{equation*}
In particular, Milnor's conjecture fails in dimensions four and five,
even under the assumption of strictly positive Ricci curvature.
\end{theorem}

\begin{corollary}\label{cor:all-dimensions}
For every integer $n\ge4$, there exists a connected, orientable, one-ended
smooth $n$-manifold admitting a complete Riemannian metric of strictly
positive Ricci curvature and having fundamental group isomorphic to
$\CPrufer$.
\end{corollary}

\begin{proof}
For $n=4,5$, this is Theorem~\ref{thm:main}.  Let $n\ge6$, let
$(M^4,g_4)$ be the four-dimensional example of Theorem~\ref{thm:main},
and equip $N^n=M^4\times\mathbb S^{n-4}$ with the product metric
$g_4+g_{\mathbb S^{n-4}}$.  This metric is complete, $N^n$ is
connected, orientable, and one-ended, and $\pi_1(N^n)\cong\pi_1(M^4)
\cong\CPrufer$ because $\mathbb S^{n-4}$ is simply connected.  For
$(v,w)\in T_xM^4\oplus T_y\mathbb S^{n-4}$,
\begin{equation*}
 \Ric_{N^n}\bigl((v,w),(v,w)\bigr)
 =
 \Ric_{g_4}(v,v)
 +(n-5)\lvert w\rvert^2,
\end{equation*}
which is positive whenever $(v,w)\ne0$.
\end{proof}

The same method, with the involution replaced by cyclic actions of
arbitrary finite order, realizes every subgroup of
\(\mathbb Q/\mathbb Z\) in dimension five.
\begin{theorem}\label{thm:all-subgroups-QZ}
For every subgroup\/ $\Gamma\leq \mathbb Q/\mathbb Z$, there exists a
connected, orientable, one-ended, complete Riemannian $5$-manifold
$(M_\Gamma,g_\Gamma)$ such that
\begin{equation*}
 \Ric_{g_\Gamma}>0, \qquad \pi_1(M_\Gamma)\cong\Gamma.
\end{equation*}
\end{theorem}

\subsection{Sketch of the construction}
\label{subs:sketch}
The starting point is the following presentation of the Pr\"ufer
$2$-group:
\begin{equation*}
 \CPrufer
 =
 \left\langle
 a_0,a_1,a_2,\ldots
 \ \middle|\
 a_0^2=1,\quad a_j=a_{j+1}^2\ \text{for all }j\ge0
 \right\rangle.
\end{equation*}
Thus each $a_{j+1}$ is a square root of $a_j$, with the initial torsion
relation $a_0^2=1$.  This presentation suggests constructing a sequence
of compact cobordisms whose incoming and outgoing boundaries are both
diffeomorphic to \(\mathbb S^1\times\mathbb S^m\).  After choosing
product parametrizations of the two boundary components, we require
the outgoing product circle to represent a generator \(a\) of the
fundamental group of the cobordism, while the incoming product circle
represents \(a^2\).  Attaching the incoming boundary of one such
cobordism to the outgoing boundary of the preceding one then
introduces the relation $a_j=a_{j+1}^2$.  The incoming boundary of the
first cobordism is filled by the \emph{initial cap}
\(D^2\times\mathbb S^m\), in which each boundary circle
\(\mathbb S^1\times\{y\}\) bounds the disk \(D^2\times\{y\}\), its
attachment imposes $a_0^2=1$.

If the relation $a_0^2=1$ is omitted, the remaining presentation defines
$\mathbb Z[1/2]$.  This group has a classical three-dimensional
realization: in the dyadic solenoid construction one considers a nested
sequence of unknotted solid torus, each winding twice inside its
predecessor, and the fundamental groups of the successive complements
form the direct system $\mathbb Z\xrightarrow{2}\mathbb Z\xrightarrow{2}
\cdots$.  This model was discussed by Shen and Sormani in connection
with the Milnor conjecture \cite{ShenSormani2006}, and it also appears
in Semola's lecture notes as a warm-up to the construction of
Bru\`e, Naber, and Semola \cite{SemolaNotes2025}, see also
\cite{ConnerMeilstrupRepovs2015} for further generalizations.  Under the
identification with $\mathbb Z[1/2]$, the relation $a_0^2=1$ kills the
integral subgroup and leaves $\mathbb Z[1/2]/\mathbb Z\cong\CPrufer$.
We stress, however, that our manifolds are not obtained as quotients of
the three-dimensional solenoid complement.

The basic topological block is the exterior of a winding-two circle
$c_m\subset\mathbb S^1\times D^{m+1}$.  When $m=1$, removing a tubular
neighborhood of $c_1$ introduces an additional meridian.  For $m\ge2$,
by contrast, $c_m$ has codimension $m+1\ge3$, so the exterior has
fundamental group $\mathbb Z$, generated by the outer longitude
$\lambda_{\mathrm{out}}$, while the inner longitude satisfies
$[\lambda_{\mathrm{in}}]=[\lambda_{\mathrm{out}}]^2$.  Moreover, the
exterior is the mapping torus of an involution of
\begin{equation*}
 P_m\cong
 D^{m+1}\setminus
 \left(
 \operatorname{int}D^{m+1}_+
 \sqcup
 \operatorname{int}D^{m+1}_-
 \right)
\end{equation*}
that exchanges the two inner boundary spheres
(Proposition~\ref{prop:double-cable}).  We write
$\partial P_m=\partial_oP_m\sqcup I_+\sqcup I_-$, regard $\partial_oP_m$
as the outgoing boundary and $I_\pm$ as the incoming boundaries, and
call the tuple
\begin{equation}\label{def:equivariant-cobordism}
 (P_m,\sigma_m,\partial_oP_m,I_+,I_-),
\end{equation}
where $\sigma_m$ is an involution preserving $\partial_oP_m$ and
exchanging $I_+$ and $I_-$, a \emph{\(\mathbb Z_2\)-equivariant
cobordism}.

It remains to realize this block in dimension \(m+2\) with a
Ricci-positive metric and boundary geometry compatible with successive
attachments.  We carry this out in three steps.  Throughout the paper we
use the convention
\begin{equation*}
 \II(X,Y)=g(\nabla_X\nu,Y),
\end{equation*}
where \(\nu\) is the outward unit normal, thus \(\II>0\) means strict
convexity.

\emph{Step 1: the equivariant cobordism.}  Given $R_*>0$ and
$\varepsilon_*>0$, we construct a metric $h_m$ on $P_m$ with
$\Ric_{h_m}>0$ and an isometric involution $\sigma_m$ such that the
incoming boundary components are round $m$-spheres of radius $R_*$ with
\begin{equation*}
 \II_{I_\pm}>-\varepsilon_*\,h_m|_{I_\pm},
\end{equation*}
while the outer boundary is round and strictly convex, and $\sigma_m$
restricts to an axial half-turn when $m=2$ and to the antipodal map
when $m=3$ (Theorem~\ref{thm:ricci-positive-equivariant-cobordism}).
The construction uses Perelman's angular necks \cite{Perelman1997}: a
sectionally positively curved angular ball contains two smaller angular
balls exchanged by an involution, after removing them we attach
identical Ricci-positive necks, which normalize the induced metrics and
second fundamental forms on $I_\pm$.  The outer boundary metric is then
joined to a round metric by the equivariant path cylinder of
Theorem~\ref{thm:path-cylinder-unified}, using invariant Ricci flows:
Hamilton's convergence theorem for surfaces \cite{Hamilton1988Surfaces}
when $m=2$, and his convergence theorem for three-manifolds with
positive Ricci curvature \cite{Hamilton1982} when $m=3$.  This step
fails for $m=1$, see Remark~\ref{rem:m1-obstruction}.

\emph{Step 2: the mapping torus.}  We choose a positive
$\sigma_m$-invariant function $f$ such that the warped metric
$G=h_m+f^2\dd t^2$ on $\mathbb R\times P_m$ has positive Ricci curvature
and the required boundary second fundamental forms
(Lemma~\ref{lem:un-warp}).  For $T>0$, the map
$\mathcal T(t,x)=(t+T,\sigma_m x)$ is an isometry of $G$, and we set
\begin{equation}\label{def:mapping-torus}
 W_m=(\mathbb R\times P_m)/\langle\mathcal T\rangle.
\end{equation}
On the outer boundary, the distinguished loop closes after a single
application of \(\mathcal T\), whereas \(\mathcal T\) exchanges the two
lifts \(I_\pm\) of the inner boundary, so the first return to either
lift occurs only after applying \(\mathcal T^2\).  This discrepancy
between the two first-return times is the geometric origin of the
square relation $[\lambda_{\mathrm{in}}]=[\lambda_{\mathrm{out}}]^2$
(Lemma~\ref{lem:un-boundary-components}).

\emph{Step 3: the holonomy interpolation cylinder.}  The two boundary
components of $W_m$ have the desired relative periods but different
holonomies: the outer boundary carries a half-turn, whereas the inner
boundary is an untwisted product.  To interpolate between them, we
attach the \emph{holonomy interpolation cylinder}
\begin{equation}\label{def:holonomy-interpolation-cylinder}
 \mathcal C=[0,s_*]\times\mathbb S^1\times\mathbb S^m,
 \qquad
 g_{\mathcal C}=\dd s^2+g_{\ell(s),\rho(s),\omega(s)},
\end{equation}
where the twist $\omega$ decreases from $\omega(0)=1/2$ to
$\omega(s_*)=0$.  Once the circle and sphere radius $\ell$ and $\rho$
have been fixed, $\omega$ is chosen to minimize a weighted Dirichlet
energy, whose Euler--Lagrange equation
\begin{equation*}
 \left(\frac{\rho^{m+2}\omega'}{\ell}\right)'=0
\end{equation*}
makes the mixed Ricci components vanish identically
(Lemmas~\ref{lem:un-ricci} and~\ref{lem:least-energy-twist}).
Proposition~\ref{prop:holonomy-interpolation-cylinder} constructs the
Ricci-positive cylinder, and Lemmas~\ref{lem:un-left-boundary-gluing}
and~\ref{lem:relative-gluing-unified} glue it to $W_m$.

\emph{Assemble the blocks to the main construction.}  The resulting building blocks are compact
Ricci-positive cobordisms with untwisted product boundaries
(Proposition~\ref{prop:ricci-positive-building-block}).  Attaching the
initial cap (Lemma~\ref{lem:un-cap}) and then infinitely many blocks
produces a one-ended manifold \(M\) with positive Ricci curvature.  The
fundamental group of its \(N\)-th compact stage is
\(\mathbb Z/2^{N+1}\mathbb Z\), and each inclusion induces
multiplication by two, so that
\begin{equation*}
 \pi_1(M)
 \cong
 \varinjlim_N
 \left(
   \mathbb Z/2^{N+1}\mathbb Z,
   [k]\longmapsto[2k]
 \right)
 \cong
 \CPrufer.
\end{equation*}
Finally, each holonomy interpolation cylinder contains a central
subcylinder disjoint from every smoothing collar, and these
subcylinders are chosen so long that every curve traversing infinitely
many stages has infinite length, hence the metric is complete.

\begin{remark}[The case $m=1$]\label{rem:m1-obstruction}
Although the topological model also makes sense for $m=1$, the boundary
geometry of Step~1 does not.  In dimension two, positive Ricci curvature
means positive Gaussian curvature, and with our sign convention the
geodesic curvature of a boundary circle is $\II(T,T)$ for a unit
tangent vector $T$.  Suppose that $P_1$ carried such a metric $h_1$,
with inner boundary circles of radius $R_*$, strictly convex outer
boundary, and $\II_{I_\pm}>-\varepsilon_*h_1|_{I_\pm}$.  Since
$\chi(P_1)=-1$ and $\operatorname{Length}(I_\pm)=2\pi R_*$, the
Gauss--Bonnet theorem would give
\begin{equation*}
 -2\pi
 =
 \int_{P_1}K_{h_1}\,\dd A
 +
 \int_{\partial P_1}k_g\,\dd s
 >
 -4\pi R_*\varepsilon_*,
\end{equation*}
which is impossible whenever $\varepsilon_*\le(2R_*)^{-1}$.  Thus the
restriction $m\ge2$ is essential.
\end{remark}

\subsection{Comparison with the construction of Bru\`e--Naber--Semola}
The present quotient construction is related to the smooth snowflake
construction of Bru\`e, Naber, and Semola \cite{BNS7,BNS6}, but the two
constructions organize the branching in fundamentally different ways.
In their seven-dimensional construction \cite{BNS7}, one begins with a
subgroup $\Gamma\le\mathbb Q/\mathbb Z$ and an exhaustion by finite
cyclic subgroups whose generators satisfy
$\gamma_j^{k_j}=\gamma_{j-1}$ for a sequence of integers $k_j\ge2$.  The
universal cover is then constructed together with a compatible free
action of $\Gamma$, and the variable branching orders $k_j$ are encoded
by an equivariant smooth snowflake and a highly twisted gluing
procedure.  Their six-dimensional construction carries out a similar
argument for $\Gamma=\mathbb Q/\mathbb Z$ \cite{BNS6}.

Our construction instead works directly on the quotient and specializes
to $k_j\equiv2$.  The involution $\sigma_m$, which exchanges the two
incoming boundary components, encodes this order-two branching, and the
same block is used at every stage.  Consequently, we do not need to
construct compatible actions of increasingly large cyclic groups on the
universal cover or to untwist general elements of a higher-dimensional
mapping class group.  This reduction to a single involution and a single
block is what makes the direct four- and five-dimensional construction
possible.

To encode a general relation $\gamma_j^{k_j}=\gamma_{j-1}$, the
appropriate building block is a $\mathbb Z_{k_j}$-equivariant
cobordism
\begin{equation*}
 P_{m,k_j} \cong D^{m+1}\setminus \biggl( \bigsqcup_{\alpha=1}^{k_j}
 \operatorname{int}D^{m+1}_\alpha \biggr)
\end{equation*}
whose cyclic action permutes the $k_j$ incoming boundary components,
the corresponding mapping torus has first return time $k_jT$ along the
incoming boundary and therefore produces the required $k_j$-th-root
relation.  In base dimension four (that is, $m=3$),
Appendix~\ref{app:Zp-equivariant} constructs such a Ricci-positive
$\mathbb Z_p$-equivariant metric on $P_{3,p}$ for every integer
$p\ge2$, with the same boundary control as in
Theorem~\ref{thm:ricci-positive-equivariant-cobordism} and with
generator $\exp(2\pi K_3/p)$ on the round outer boundary.  This yields
Theorem~\ref{thm:all-subgroups-QZ}.

\begin{remark}\label{rem:m2-higher-order-limitation}
Theorem~\ref{thm:all-subgroups-QZ} is stated only in dimension five,
that is, for $m=3$.  The construction in
Appendix~\ref{app:Zp-equivariant} uses features specific to the round
$3$-sphere: its join description $\mathbb S^3\cong\mathbb S^1*\mathbb
S^1$, the resulting two circle directions, and the free diagonal action
$(z_1,z_2)\mapsto(e^{2\pi i/p}z_1,e^{2\pi i/p}z_2)$.  These ingredients
are used both in the torusc deformation of the $\mathbb Z_p$-equivariant
cobordism and in the positively curved interpolation of its outer
boundary metric.  There is no direct analogue on $\mathbb S^2$: the
second circle direction is absent, and every nontrivial
orientation-preserving cyclic action on $\mathbb S^2$ has fixed points.
Consequently, the method of Appendix~\ref{app:Zp-equivariant} does not
produce the $\mathbb Z_p$-equivariant building blocks needed for $p\ge3$
in base dimension three.  This is a limitation of the present method
rather than a topological obstruction. Hence, whether every subgroup of
$\mathbb Q/\mathbb Z$ arises as the fundamental group of a complete
four-manifold with positive Ricci curvature remains open.
\end{remark}

\subsection{Further properties}
Section~\ref{sec:further-properties} records further features of the
manifolds constructed in Theorem~\ref{thm:main}.  Their universal covers
$\widetilde M$ are diffeomorphic to
$\mathbb R^{m+2}\#_{\infty}(\mathbb S^2\times\mathbb S^m)$, and are
therefore one-ended and simply connected at infinity.  The free
parameters of the holonomy interpolation cylinders can be chosen so
that, simultaneously on $M$ and on $\widetilde M$,
\begin{equation*}
 \liminf_{r\to\infty}\frac{\log\operatorname{Vol}(B_r)}{\log r}=1,
 \qquad
 \limsup_{r\to\infty}\frac{\log\operatorname{Vol}(B_r)}{\log r}=n-1.
\end{equation*}
For such choices, both $M$ and $\widetilde M$ have zero asymptotic
volume ratio and superlinear volume growth.  For $M$ itself, the
vanishing of the asymptotic volume ratio holds for every choice of
parameters by \cite{Li1986,AndersonTopology1990}, since $\pi_1(M)$ is
infinite, when $m=2$, the same holds for $\widetilde M$ by
\cite{HuangHuang2025}, since $\pi_1(M)$ is not finitely generated.  These results describe the topology and large-scale geometry
of the examples and show how far they lie outside several known
finite-generation examples.

\subsection*{Organization of the paper}
Section~\ref{sec:double-cable} studies the topology of the winding-two
circle exterior and derives the square relation between its
distinguished boundary loops.  Section~\ref{sec:relative-tools}
establishes the relative Ricci-positive gluing and path-cylinder
results used throughout.  Section~\ref{sec:ricci-positive-equivariant-cobordism}
constructs the Ricci-positive \(\mathbb Z_2\)-equivariant cobordism,
with the angular-neck calculations deferred to
Appendix~\ref{app:perelman-neck}.  Section~\ref{sec:unified-cable}
constructs the Ricci-positive mapping torus and determines its boundary
geometry and loop classes.  Section~\ref{sec:holonomy-interpolation-cylinder}
constructs the holonomy interpolation cylinder.
Section~\ref{sec:complete-construction} assembles the complete manifold,
computes its fundamental group, and proves Theorems~\ref{thm:main}
and~\ref{thm:all-subgroups-QZ}.  Section~\ref{sec:further-properties}
studies the universal cover and the volume growth.
Appendix~\ref{app:Zp-equivariant} constructs the
$\mathbb Z_p$-equivariant cobordisms in base dimension four.

\subsection*{Acknowledgments}
The authors would like to thank Xingyu Zhu for many helpful discussions, Xianzhe Dai and Guofang Wei for bringing this problem to their attention, Xuezhang Chen, Qing Han, Aaron Naber, Jiayin Pan and Daniele Semola for valuable comments and suggestions. The authors would also like to thank Daniel Ruberman for pointing out the new problem list \cite{BaykurKirbyRuberman2026}, and in particular its Problem~5.15. Zetian Yan is substantially supported by research grants from the Research Grants Council of the Hong Kong Special Administrative Region, China [Project No.: CUHK 14307123, CUHK 14306624 and CUHK 14308325].

\subsection*{Disclosure on AI assistance}

The authors used AI-assisted tools, principally ChatGPT and Claude, for
brainstorming, literature searches, preliminary checks of calculations and parameter
estimates, notation and exposition review, and manuscript editing.
The authors independently worked through and verified all theorem statements, proofs,
calculations, and take full responsibility for the contents of the paper.

\section{Topology of the winding-two circle exterior}
\label{sec:double-cable}

Before constructing the Ricci-positive metric, we first isolate the
topological origin of the square-root relation. Let
\begin{equation*}
c\subset \mathbb S^{1}\times D^{m+1}
\end{equation*}
be an embedded circle with the following properties. The restriction to
\(c\) of the projection
\begin{equation*}
\pi_{\mathbb S^1}\colon
\mathbb S^{1}\times D^{m+1}\longrightarrow \mathbb S^{1}
\end{equation*}
is a two-fold covering, whereas the restriction of the projection onto
\(D^{m+1}\) maps \(c\) diffeomorphically onto an embedded circle. Thus,
over each point of \(\mathbb S^{1}\), the curve \(c\) meets the
corresponding fiber in exactly two points. It follows that the complement
of a sufficiently small tubular neighborhood of \(c\) fibers over
\(\mathbb S^{1}\), with fiber
\begin{equation*}
D^{m+1}\setminus
\left(
\operatorname{int}D^{m+1}_{+}
\sqcup
\operatorname{int}D^{m+1}_{-}
\right).
\end{equation*}
The monodromy of this fibration interchanges the two inner boundary
components. We next give an explicit description of the fibration. Its
fiber, together with this order-two monodromy, will provide the
topological model for the \(\mathbb Z_2\)-equivariant cobordism constructed later.

Let $m\ge2$.  For every $\rho>0$, denote by
\begin{equation*}
 D^{m+1}_\rho=\{x\in\mathbb R^{m+1}:|x|\le\rho\}
\end{equation*}
the closed Euclidean $(m+1)$-ball of radius $\rho$.  Abbreviate
$D^{m+1}=D^{m+1}_1$ and denote
\begin{equation*}
 V=\mathbb S^1\times D^{m+1}.
\end{equation*}
Let $e_1,e_{m+1}$ be the first and last standard basis vectors of
$\mathbb R^{m+1}$, write $\mathrm i=\sqrt{-1}$, and define the
$\mathbb R^{m+1}$-valued function
\begin{equation*}
 u:\mathbb R/2\pi\mathbb Z\longrightarrow\mathbb R^{m+1},
 \qquad
 u(\theta)=(\cos\theta,\sin\theta,0,\ldots,0).
\end{equation*}
Fix
\begin{equation*}
 0<\varepsilon<1,
 \qquad
 0<\rhoCab<\min\{\varepsilon,1-\varepsilon\},
\end{equation*}
and define
\begin{align*}
 c_{\varepsilon}(\theta)
   &=\bigl(e^{2\mathrm i\theta},\varepsilon u(\theta)\bigr),\\
 \mathcal F_{\rhoCab}(\theta,v)
   &=\bigl(e^{2\mathrm i\theta},\varepsilon u(\theta)+v\bigr),\\
 \mathcal N_{\rhoCab}(c_{\varepsilon})
   &=\mathcal F_{\rhoCab}
     \bigl((\mathbb R/2\pi\mathbb Z)\times
                     D^{m+1}_{\rhoCab}\bigr),
\end{align*}
\begin{align}
 C_m
   &=V\setminus\mathcal F_{\rhoCab}
     \bigl((\mathbb R/2\pi\mathbb Z)\times
              \operatorname{int}D^{m+1}_{\rhoCab}\bigr).\label{eq:explicit-exterior}
\end{align}
The first assertion of Proposition~\ref{prop:double-cable} below shows
that $\mathcal F_{\rhoCab}$ is a tubular embedding.  We introduce the two boundary components of $C_m$:
\begin{equation*}
 \partial_oC_m=\mathbb S^1\times\partial D^{m+1},
 \qquad
 \partial_iC_m=\mathcal F_{\rhoCab}
 \bigl((\mathbb R/2\pi\mathbb Z)\times
                         \partial D^{m+1}_{\rhoCab}\bigr).
\end{equation*}

We use the base points, longitudes, and connecting path
\begin{align*}
 x_O&=(1,e_{m+1}),
 &\lambda_O(\varphi)&=\bigl(e^{\mathrm i\varphi},e_{m+1}\bigr),\\
 x_I&=(1,\varepsilon e_1+\rhoCab e_{m+1}),
 &\lambda_I(\theta)&=
   \mathcal F_{\rhoCab}(\theta,\rhoCab e_{m+1}),\\
 x(t)&=\left(1,
 t\varepsilon e_1+(1-t+t\rhoCab)e_{m+1}\right),
 &\theta,\varphi&\in\mathbb R/2\pi\mathbb Z,
 &0&\le t\le1.
\end{align*}
Here $\lambda_O$ is the positively oriented outer longitude based at
$x_O$, while $\lambda_I$ is the positively oriented longitude of the
drilled tube based at $x_I$ and determined by the constant normal vector
$\rhoCab e_{m+1}$.

For $\alpha\in\mathbb R$, let $R_\alpha\in SO(m+1)$ be the rotation through
angle $\alpha$ in the $(e_1,e_2)$-plane and the identity on its
orthogonal complement.  Explicitly,
\begin{equation*}
 R_\alpha(x_1,x_2,x_3,\ldots,x_{m+1})
 =\bigl(x_1\cos\alpha-x_2\sin\alpha,
        x_1\sin\alpha+x_2\cos\alpha,
        x_3,\ldots,x_{m+1}\bigr).
\end{equation*}
Define
\begin{align*}
 D^{m+1}_+&=\{x\in D^{m+1}:|x-\varepsilon e_1|\le\rhoCab\},
 &D^{m+1}_-&=\{x\in D^{m+1}:|x+\varepsilon e_1|\le\rhoCab\},\\
 P_m&=D^{m+1}\setminus
 \bigl(\operatorname{int}D^{m+1}_+
       \sqcup\operatorname{int}D^{m+1}_-\bigr),
 &I_\pm&=\partial D^{m+1}_\pm.
\end{align*}
The inequalities imposed on $\rhoCab$ ensure that $D^{m+1}_+$ and
$D^{m+1}_-$ are disjoint and contained in the interior of $D^{m+1}$.
Thus $P_m$ is an $(m+1)$-disk with the interiors of two disjoint
$(m+1)$-disks removed.  The half-turn $R_\pi$ preserves $P_m$ and
exchanges $I_+$ and $I_-$.  Set
\begin{equation*}
 \phi=R_\pi|_{P_m},
 \qquad
 T_\phi=\bigl(P_m\times[0,2\pi]\bigr)/
       \bigl((x,2\pi)\sim(\phi(x),0)\bigr).
\end{equation*}
This cobordism $P_m$ is the fixed topological model used throughout the
paper.  In Section~\ref{sec:ricci-positive-equivariant-cobordism}, the same symbol
denotes a manifold equipped with a Ricci-positive metric and a specified
diffeomorphism to this model.

\begin{proposition}
\label{prop:double-cable}
The exterior $C_m$ defined above has the following properties.
\begin{enumerate}[label=\textup{(\roman*)}]
\item $\mathcal F_{\rhoCab}$ is a smooth embedding, and
\begin{equation*}
 \partial C_m=\partial_oC_m\sqcup\partial_iC_m,
 \qquad
 \partial_oC_m\cong\mathbb S^1\times\mathbb S^m,
 \qquad
 \partial_iC_m\cong\mathbb S^1\times\mathbb S^m.
\end{equation*}
\item If $j:C_m\hookrightarrow V$ is the inclusion, then
\begin{equation*}
 j_*:\pi_1(C_m,x_O)\xrightarrow{\ \cong\ }\pi_1(V,x_O),
 \qquad
 \pi_1(C_m,x_O)=\langle \gamma\rangle\cong\mathbb Z,
 \qquad
 \gamma=[\lambda_O].
\end{equation*}
\item The path $x(t)$ is contained in $C_m$, satisfies
\begin{equation*}
 x(0)=x_O,
 \qquad
 x(1)=x_I,
\end{equation*}
and, with $x^{-1}(t)=x(1-t)$,
\begin{equation*}
 \bigl[x*\lambda_I*x^{-1}\bigr]=\gamma^2
 \quad\text{in }\pi_1(C_m,x_O).
\end{equation*}
For every path $\widehat x:[0,1]\to C_m$ from $x_O$ to $x_I$,
\begin{equation*}
 \bigl[\widehat x*\lambda_I*\widehat x^{-1}\bigr]=\gamma^2.
\end{equation*}
\item The map
\begin{equation*}
 \Phi:T_\phi\longrightarrow C_m,
 \qquad
 \Phi([x,\varphi])=(e^{\mathrm i\varphi},R_{\varphi/2}x),
\end{equation*}
is a diffeomorphism, and
\begin{equation*}
 \phi(I_+)=I_-,
 \qquad
 \phi(I_-)=I_+.
\end{equation*}
\end{enumerate}
\end{proposition}

\begin{proof}
If $\mathcal F_{\rhoCab}(\theta,v)=\mathcal F_{\rhoCab}(\theta',v')$, then
$e^{2\mathrm i\theta}=e^{2\mathrm i\theta'}$, so $\theta'=\theta$ or
$\theta'=\theta+\pi$ modulo $2\pi$.  In the first case $v'=v$.  In the
second case $u(\theta')=-u(\theta)$, hence $v'-v=2\varepsilon u(\theta)$
and $2\varepsilon=|v'-v|\le2\rhoCab<2\varepsilon$, a contradiction.  Thus
$\mathcal F_{\rhoCab}$ is injective.  Its differential
\begin{equation*}
 d\mathcal F_{\rhoCab,(\theta,v)}(\tilde\theta,\tilde v)
 =\bigl(2\tilde\theta\mathrm i e^{2\mathrm i\theta},
        \varepsilon \tilde\theta u'(\theta)+\tilde v\bigr)
\end{equation*}
is injective, so $\mathcal F_{\rhoCab}$ is an injective immersion of a
compact manifold, hence a smooth embedding.  Since
$|\varepsilon u(\theta)+v|\le\varepsilon+\rhoCab<1$, the tube
$\mathcal N_{\rhoCab}(c_\varepsilon)$ lies in
$\mathbb S^1\times\operatorname{int}D^{m+1}$, and
\eqref{eq:explicit-exterior} gives the stated boundary decomposition. This proves  \textup{(i)}.

Radial deformation in the fibers of $\mathcal F_{\rhoCab}$, that is,
$\mathcal F_{\rhoCab}(\theta,v)\mapsto\mathcal F_{\rhoCab}\bigl(\theta,
((1-\tau)|v|+\tau\rhoCab)v/|v|\bigr)$ for $0<|v|\le\rhoCab$ and the
identity on $C_m$, is a strong deformation retraction of
$V\setminus\operatorname{im}c_\varepsilon$ onto $C_m$.  The circle
$\operatorname{im}c_\varepsilon$ has codimension $m+1\ge3$ in $V$, so by
transversality \cite[Chapter~3]{Hirsch1976} every loop based at $x_O$,
and every homotopy of such loops, can be moved off
$\operatorname{im}c_\varepsilon$ relative to its boundary.  Hence the
inclusion $V\setminus\operatorname{im}c_\varepsilon\hookrightarrow V$
induces an isomorphism on fundamental groups.  Since $V\simeq\mathbb S^1$,
this proves \textup{(ii)}.

The equation $e^{2\mathrm i\theta}=1$ has the solutions $\theta=0,\pi$,
so the tube meets the slice $\{1\}\times D^{m+1}$ in
$\{1\}\times(D^{m+1}_+\sqcup D^{m+1}_-)$.  Write $x(t)=(1,x_2(t))$ with
$x_2(t)=t\varepsilon e_1+(1-t+t\rhoCab)e_{m+1}$.  Then
$|x_2(t)|\le1-t(1-\varepsilon-\rhoCab)\le1$,
\begin{equation*}
 |x_2(t)-\varepsilon e_1|^2
 =\varepsilon^2(1-t)^2+(1-t+t\rhoCab)^2\ge\rhoCab^2,
 \qquad
 |x_2(t)+\varepsilon e_1|\ge(1+t)\varepsilon>\rhoCab,
\end{equation*}
so $x([0,1])\subset C_m$, with $x(0)=x_O$ and $x(1)=x_I$.  Let
$p\colon C_m\to\mathbb S^1$ be the restriction of the projection to the
first factor.  By \textup{(ii)}, $p_*\colon\pi_1(C_m,x_O)\to
\pi_1(\mathbb S^1,1)\cong\mathbb Z$ is an isomorphism.  Since
$p\circ\lambda_O(\varphi)=e^{\mathrm i\varphi}$, $p\circ x\equiv1$, and
$p\circ\lambda_I(\theta)=e^{2\mathrm i\theta}$, we obtain
$p_*[\lambda_O]=1$ and $p_*[x*\lambda_I*x^{-1}]=2$, whence
$[x*\lambda_I*x^{-1}]=\gamma^2$.  Another path $\widehat x$ from $x_O$ to
$x_I$ changes this class by conjugation by $[\widehat x*x^{-1}]$, which
is trivial because $\pi_1(C_m,x_O)\cong\mathbb Z$ is abelian. This proves  \textup{(iii)}.

The map $\widetilde\Phi(x,\varphi)=(e^{\mathrm i\varphi},R_{\varphi/2}x)$
on $P_m\times\mathbb R$ satisfies
$\widetilde\Phi(x,\varphi+2\pi)=\widetilde\Phi(\phi(x),\varphi)$, and so
induces $\Phi\colon T_\phi\to C_m$.  For fixed $\varphi$, the two
solutions of $e^{2\mathrm i\theta}=e^{\mathrm i\varphi}$ are
$\theta=\varphi/2$ and $\theta=\varphi/2+\pi$, and
$\varepsilon u(\varphi/2)=R_{\varphi/2}(\varepsilon e_1)$,
$\varepsilon u(\varphi/2+\pi)=R_{\varphi/2}(-\varepsilon e_1)$.  Hence
\begin{equation*}
 C_m\cap\bigl(\{e^{\mathrm i\varphi}\}\times D^{m+1}\bigr)
 =\{e^{\mathrm i\varphi}\}\times R_{\varphi/2}P_m,
\end{equation*}
so $\Phi$ is bijective.  Its differential is $R_{\varphi/2}$ in the
$P_m$-directions and has a nonzero component tangent to the first circle
factor in the $\varphi$-direction, so $\Phi$ is a bijective local
diffeomorphism, hence a diffeomorphism.  Finally,
$R_\pi(D^{m+1}_\pm)=D^{m+1}_\mp$ gives $\phi(I_\pm)=I_\mp$. This proves  \textup{(iv)}, which finishes the proof of the Proposition.
\end{proof}

Under the diffeomorphism of part~\textup{(iv)}, a loop traversing the
mapping-torus direction once represents the generator \(\gamma\),
whereas the distinguished inner longitude represents \(\gamma^2\).  It
remains to realize this model by a Ricci-positive metric whose boundary
metrics and second fundamental forms allow successive Ricci-positive
gluings.

\section{Relative Ricci-positive gluing and metric interpolation}
\label{sec:relative-tools}
The construction of Section~\ref{sec:double-cable} determines the topology of
the basic block and the relation between its distinguished boundary
loops, but
it does not yet provide the boundary geometry needed for Ricci-positive
gluing.  To attach the basic blocks successively, we must be able to
modify a metric inside prescribed collar neighborhoods while controlling
both the induced boundary metric and the corresponding second
fundamental form.  Two such operations will be used repeatedly below:
gluing Ricci-positive manifolds along isometric boundary components, and
transporting a boundary metric along a path of Ricci-positive metrics
with controlled convexity at the two ends.  Both constructions must also
respect the relevant symmetries.

We develop these tools in this section.  Their proofs are based on the
Ricci-curvature formulas for a one-parameter family of metrics on the
slices of a collar.  Accordingly, consider a metric of the form
\begin{equation*}
g=ds^2+h(s),
\end{equation*}
and let \(\nu=\partial_s\).  With the convention fixed above, the second
fundamental form of an \(s\)-slice with respect to \(\nu\) is
\begin{equation*}
L=\II_{\nu}=\frac12\dot h.
\end{equation*}
We begin by recording the corresponding Ricci-curvature formulas.

\begin{lemma}\label{lem:gaussian-ricci-unified}
For vectors $X,Y$ tangent to an $s$-slice,
\begin{align}
 \Ric_g(\partial_s,\partial_s)
 &=-\frac12\tr_h\ddot h
   +\frac14\tr(h^{-1}\dot h\,h^{-1}\dot h),\label{eq:gaussian-normal}\\
 \Ric_g(\partial_s,X)
 &=(\divsymb_hL)(X)-\dd(\tr_hL)(X),\label{eq:gaussian-mixed}\\
 \Ric_g(X,Y)
 &=\Ric_h(X,Y)-\frac12\ddot h(X,Y)
   +\frac12(\dot h\,h^{-1}\dot h)(X,Y)
   -\frac14\tr_h(\dot h)\dot h(X,Y).\label{eq:gaussian-tangential}
\end{align}
\end{lemma}

\begin{proof}
Extend the tangent vectors to be independent of $s$ in product
coordinates.  Then the Christoffel symbols of $g$ satisfy
\begin{equation*}
 \Gamma^s_{ij}=-L_{ij},\qquad
 \Gamma^k_{si}=\Gamma^k_{is}=(h^{-1}L)^k{}_i,
 \qquad \Gamma^k_{ij}=\Gamma^k_{ij}(h).
\end{equation*}
The Weingarten equation and the Gauss and Codazzi equations yield
\begin{align*}
 \Rm_g(\partial_s,X,Y,\partial_s)
   &=-\dot L(X,Y)+(Lh^{-1}L)(X,Y),\\
 \Ric_g(\partial_s,X)&=(\divsymb_hL)(X)-\dd(\tr_hL)(X),\\
 \Rm_g(X,Y,Z,W)&=\Rm_h(X,Y,Z,W)
 -L(X,W)L(Y,Z)+L(X,Z)L(Y,W).
\end{align*}
Taking the $h$-trace of the first identity gives
\eqref{eq:gaussian-normal}.  The second identity is
\eqref{eq:gaussian-mixed}, combining the first identity with the
$h$-trace of the Gauss equation gives
\begin{equation*}
 \Ric_g(X,Y)=\Ric_h(X,Y)-\dot L(X,Y)
 +2(Lh^{-1}L)(X,Y)-(\tr_hL)L(X,Y).
\end{equation*}
Substituting $L=\tfrac12\dot h$ yields
\eqref{eq:gaussian-tangential}.
\end{proof}
The following lemma is essentially Perelman's Ricci-positive gluing theorem
\cite[Section~4]{Perelman1997}, a more general version is proved in
\cite{ReiserWraithGluing}.  We need, in addition, control of the support
of the modification, equivariance, and $C^0$-closeness to the glued
piecewise metric, which also controls the volume.
\begin{lemma}
\label{lem:relative-gluing-unified}
Let $(M_i,g_i)$, $i=1,2$, be compact Riemannian manifolds with boundary
and $\Ric_{g_i}>0$.  For each $i$, let $N_i$ be a union of connected
components of $\partial M_i$, and let
\begin{equation*}
 \Psi:(N_1,g_1|_{TN_1})\longrightarrow(N_2,g_2|_{TN_2})
\end{equation*}
be an isometry.  Let $\II_i$ denote the second fundamental form of $N_i$
with respect to the outward unit normal of $M_i$, and suppose that
$\II_1+\Psi^*\II_2$ is positive definite on $TN_1$.  Then, for every
choice of collar neighborhoods $U_i\subset M_i$ of $N_i$, the glued
smooth manifold $M_1\cup_\Psi M_2$ admits a Riemannian metric
$\widetilde g$ such that
\begin{equation*}
 \Ric_{\widetilde g}>0,
 \qquad \widetilde g|_{M_i\setminus U_i}=g_i|_{M_i\setminus U_i}
 \quad(i=1,2).
\end{equation*}
Moreover:
\begin{enumerate}[label=\textup{(\alph*)}]
\item If a finite group $\Gamma$ acts isometrically on each
$(M_i,g_i)$, preserves $N_i$, and satisfies $\Psi(\gamma x)=\gamma\Psi(x)$
for all $\gamma\in\Gamma$ and $x\in N_1$, then $\widetilde g$ can be
chosen to be $\Gamma$-invariant.
\item The same conclusions hold for finitely many pairwise disjoint
boundary identifications, with the modifications contained in
prescribed pairwise disjoint collar neighborhoods.
\item For every $\eta\in(0,1)$ and $\varepsilon_{\mathrm{vol}}>0$, the
collars and the smoothing parameters can be chosen so that
\begin{equation}\label{eq:C0-close}
 (1-\eta)\,g_{\mathrm{pw}}\le\widetilde g\le(1+\eta)\,g_{\mathrm{pw}},
\end{equation}
where $g_{\mathrm{pw}}$ is the continuous metric equal to $g_i$ on
$M_i$, and
\begin{equation}\label{eq:volume pinching}
 \left|
 \operatorname{Vol}_{\widetilde g}(M_1\cup_\Psi M_2)
 -\operatorname{Vol}_{g_1}(M_1)
 -\operatorname{Vol}_{g_2}(M_2)
 \right|<\varepsilon_{\mathrm{vol}}.
\end{equation}
For finitely many disjoint identifications, the sum of the absolute
volume changes can be made less than $\varepsilon_{\mathrm{vol}}$.
\end{enumerate}
\end{lemma}

\begin{proof}
We recall Perelman's construction to keep track of the support of the metric modification and the equivariance of the construction.  Use $\Psi$ to identify $N_1$ and $N_2$ with one closed manifold
$N$, and choose a signed Fermi coordinate $s$, increasing from $M_1$ to
$M_2$, on collars contained in $U_1\cup U_2$.  After pulling back by
$\Psi$, $g=\dd s^2+h_-(s)$ for $s\le0$ and $g=\dd s^2+h_+(s)$ for
$s\ge0$, with $h_-(0)=h_+(0)$.  Because the outward normals are
$+\partial_s$ and $-\partial_s$,
\begin{equation}
 \dot h_-(0)-\dot h_+(0)
   =2(\II_1+\Psi^*\II_2)>0.                    \label{eq:normal-derivative-gap}
\end{equation}

\emph{Step 1: $C^1$ interpolation.}  For small $\varepsilon>0$, let
$h_\varepsilon(s)$ be the cubic polynomial in $s$, with coefficients
that are smooth sections of $\Sym^2T^*N$, whose values and first
derivatives at $s=\pm\varepsilon$ agree with those of $h_\pm$.  Taylor
expansion gives, uniformly for $|s|\le\varepsilon$,
\begin{equation}
 h_\varepsilon(s)=h_-(0)+O(\varepsilon),\qquad
 \dot h_\varepsilon(s)=O(1),\qquad
 \ddot h_\varepsilon(s)
   =-\frac{\dot h_-(0)-\dot h_+(0)}{2\varepsilon}+O(1),
 \label{eq:hermite-estimates}
\end{equation}
in the $C^2(N)$, $C^1(N)$, and $C^0(N)$ senses, respectively.  Let
$g_\varepsilon$ be the $C^1$ metric equal to $\dd s^2+h_\varepsilon(s)$
on $|s|\le\varepsilon$ and to $g$ elsewhere.  By
Lemma~\ref{lem:gaussian-ricci-unified}, \eqref{eq:normal-derivative-gap},
and \eqref{eq:hermite-estimates}, there are $c,C>0$ independent of small
$\varepsilon$ such that, on $|s|<\varepsilon$,
\begin{equation*}
 \Ric_{g_\varepsilon}(\partial_s,\partial_s)\ge\frac c\varepsilon-C,
 \qquad
 \Ric_{g_\varepsilon}|_{TN\times TN}\ge\Bigl(\frac c\varepsilon-C\Bigr)h_\varepsilon,
 \qquad
 \bigl|\Ric_{g_\varepsilon}(\partial_s,\mathord\cdot)|_{TN}\bigr|_{h_\varepsilon}\le C.
\end{equation*}
The Schur complement of the tangential block is therefore at least
$c/(2\varepsilon)-2C^2\varepsilon/c>0$ for small $\varepsilon$, so
$\Ric_{g_\varepsilon}>0$ wherever $g_\varepsilon$ is smooth.

\emph{Step 2: smoothing the corners.}  Near either hypersurface
$s=\pm\varepsilon$, write $g_\varepsilon=\dd t^2+k(t)$, where $k$ is $C^1$
and smooth on each side of $t=0$, and both one-sided Ricci tensors are
positive up to $t=0$.  Let $k_\eta$ be the convolution of $k$ in $t$
with an even mollifier supported in $(-\eta,\eta)$, and set
$\widetilde k_\eta=k+\chi_\eta(k_\eta-k)$, where $\chi_\eta$ is a cutoff
equal to one on $[-2\eta,2\eta]$, supported in $(-\sqrt\eta,\sqrt\eta)$,
with $|\dot\chi_\eta|\le C\eta^{-1/2}$ and $|\ddot\chi_\eta|\le
C\eta^{-1}$.  The formulas of Lemma~\ref{lem:gaussian-ricci-unified} are
affine in $\ddot k$, and $k_\eta\to k$ in $C^1$, hence the Ricci
tensor of $\dd t^2+k_\eta$ is, up to $o_\eta(1)$, an average of positive
one-sided Ricci tensors near $t=0$.  On $2\eta\le|t|\le\sqrt\eta$ the
convolution sees only one smooth side, and the evenness of the
mollifier gives $\widetilde k_\eta-k=O(\eta)$ in $C^2$.  Thus
$\dd t^2+\widetilde k_\eta$ has positive Ricci curvature for small
$\eta$, and performing this at both hypersurfaces yields a smooth metric
$\widetilde g$ with $\Ric_{\widetilde g}>0$ and $\widetilde g=g_i$ on
$M_i\setminus U_i$.

\emph{Step 3:}  Under the hypotheses of \textup{(a)}, the normal
exponential maps give $\Gamma$-equivariant Fermi collars, and Hermite
interpolation, convolution in the normal coordinate, and multiplication
by $\chi_\eta$ commute with the $\Gamma$-action, hence $\widetilde g$ is
$\Gamma$-invariant.  For \textup{(b)}, the construction is carried out
independently in the prescribed disjoint collars.  For \textup{(c)},
\eqref{eq:hermite-estimates} and the uniform convergence of the
mollified and cut-off tensors show that $\widetilde g-g_{\mathrm{pw}}\to0$
uniformly as $\varepsilon,\eta\to0$, while the supports of all
modifications shrink to $N$.  This gives \eqref{eq:C0-close}, and,
since the volume densities are uniformly bounded and the modified
collars have widths tending to zero, also \eqref{eq:volume pinching}.
\end{proof}

We next establish a Ricci-positive path-cylinder theorem, which may be
viewed as a scaled analogue of the ``isotopy implies concordance''
principle.  Given a path of Ricci-positive metrics, the construction
produces a Ricci-positive metric on a cylinder inducing \(g_0\) at one
end and a constant rescaling of \(g_1\) at the other, with the outward
second fundamental forms satisfying the one-sided bounds needed for
subsequent gluing.  The construction follows
\cite[Theorem~A.1]{Burdick} and \cite[Assertion]{Perelman1997}.

\begin{theorem}
\label{thm:path-cylinder-unified}
Let $N^n$ be a closed manifold, and let $(g_s)_{s\in[0,1]}$ be a smooth
path of Riemannian metrics satisfying $\Ric_{g_s}>0$.  Suppose that the
path is constant on a neighborhood of each endpoint.  For every
$\delta_{\mathrm{cyl}}>0$ there are $\Lambda_{\mathrm{cyl}}>0$ and a
smooth Riemannian metric $\widehat g$ on
$[0,1]\times N$ with $\Ric_{\widehat g}>0$ such that:
\begin{enumerate}[label=\textup{(\roman*)}]
\item the metric induced by $\widehat g$ on $\{0\}\times N$ is $g_0$,
and the second fundamental form with respect to the outward unit normal
is greater than $-\delta_{\mathrm{cyl}}g_0$,
\item the metric induced by $\widehat g$ on $\{1\}\times N$ is
$\Lambda_{\mathrm{cyl}}^2g_1$, and the second fundamental form with respect to the outward
unit normal is positive definite.
\end{enumerate}
Moreover, the construction is equivariant under every compact group of
diffeomorphisms of $N$ that preserves each metric $g_s$.
\end{theorem}

\begin{proof}
The idea is to spread the path along a long logarithmic interval of a
cone-like warped product $\dd t^2+f(t)^2g_{\lambda(t)}$, so slowly that the
positive curvature of the slices dominates every error term, and then to
rescale.  Compactness of $[0,1]\times N$ gives constants $\mu,C>0$ such that
\begin{equation}
 \Ric_{g_s}\ge\mu g_s,
 \qquad
 |\partial_sg_s|_{g_s}
 +|\nabla^{g_s}(\partial_sg_s)|_{g_s}
 +|\partial_s^2g_s|_{g_s}\le C
 \quad (0\le s\le1).
 \label{eq:path-uniform-bounds}
\end{equation}
Choose $a>0$ and $\sigma>0$ so that
\begin{equation}
 a<\delta_{\mathrm{cyl}},
 \qquad \frac{\mu}{a^2}>2n,
 \qquad
 n\sigma-\frac12\sup_{[0,1]\times N}
   |\tr_{g_s}(\partial_sg_s)|\ge1.
 \label{eq:path-parameter-choice}
\end{equation}

For a sufficiently large number $\ell$, consider
$[e^\ell,e^{2\ell}]\times N$ and define
\begin{equation}
 \lambda(t)=2\left(1-\frac{\ell}{\log t}\right),
 \qquad
 f(t)=at\,e^{-\sigma\lambda(t)},
 \qquad
 \overline g=\dd t^2+f(t)^2g_{\lambda(t)}.
 \label{eq:path-cylinder-metric}
\end{equation}
Then $\lambda(e^\ell)=0$, $\lambda(e^{2\ell})=1$, and
\begin{equation}
 \lambda'(t)=\frac{2\ell}{t(\log t)^2},
 \qquad
 \lambda''(t)=-\frac{\lambda'(t)}t
                   \left(1+\frac2{\log t}\right),
 \qquad
 \frac{f'(t)}{f(t)}=\frac1t-\sigma\lambda'(t).
 \label{eq:path-cylinder-derivatives}
\end{equation}
The shape endomorphism of a $t$-slice with respect to $\partial_t$ is
therefore
\begin{equation}
 \frac12\bigl(f^2g_{\lambda(t)}\bigr)^{-1}
       \partial_t\bigl(f^2g_{\lambda(t)}\bigr)
 =\left(\frac1t-\sigma\lambda'(t)\right)\Id
  +\frac{\lambda'(t)}2g_{\lambda(t)}^{-1}
       \left.\partial_sg_s\right|_{s=\lambda(t)}.
 \label{eq:path-slice-shape}
\end{equation}

Substituting \eqref{eq:path-cylinder-derivatives} and \eqref{eq:path-slice-shape} into the three formulas of Lemma~\ref{lem:gaussian-ricci-unified}, we find that the normal component is
\begin{align}
 \Ric_{\overline g}(\partial_t,\partial_t)
 ={}&\left(n\sigma-\frac12\tr_{g_{\lambda(t)}}
       \left(\left.\partial_sg_s\right|_{s=\lambda(t)}\right)\right)
       \frac{\lambda'(t)}t\left(1-\frac2{\log t}\right)
       +O\bigl(\lambda'(t)^2\bigr).
 \label{eq:path-normal-ricci}
\end{align}
The constant in the error term is independent of $\ell$.  Since
$\log t\in[\ell,2\ell]$, 
\eqref{eq:path-parameter-choice} and \eqref{eq:path-normal-ricci} give
\begin{equation*}
 \Ric_{\overline g}(\partial_t,\partial_t)
 \ge\frac1{8t^2\ell}
\end{equation*}
when $\ell$ is sufficiently large.

For the tangential component, use $f(t)^2g_{\lambda(t)}$ to raise one
index in \eqref{eq:gaussian-tangential}.  Since
$\lambda'(t)=O((t\ell)^{-1})$ and
$\lambda''(t)=O((t^2\ell)^{-1})$, 
\eqref{eq:path-uniform-bounds} and
\eqref{eq:path-cylinder-derivatives} give
\begin{align*}
 &\bigl(f(t)^2g_{\lambda(t)}\bigr)^{-1}
   \partial_t\bigl(f(t)^2g_{\lambda(t)}\bigr)
   =\frac2t\Id+O\left(\frac1{t\ell}\right),\\
 &\bigl(f(t)^2g_{\lambda(t)}\bigr)^{-1}
   \partial_t^2\bigl(f(t)^2g_{\lambda(t)}\bigr)
   =\frac2{t^2}\Id+O\left(\frac1{t^2\ell}\right),\\
 &\tr_{f(t)^2g_{\lambda(t)}}
   \partial_t\bigl(f(t)^2g_{\lambda(t)}\bigr)
   =\frac{2n}{t}+O\left(\frac1{t\ell}\right).
\end{align*}
The endomorphism-valued error terms are measured with respect to
$g_{\lambda(t)}$, and all constants are independent of $t$ and $\ell$.
Thus the last three terms of \eqref{eq:gaussian-tangential}, after one
index is raised, equal
\begin{equation*}
 -\frac1{t^2}\Id+\frac2{t^2}\Id-\frac n{t^2}\Id
 +O\left(\frac1{t^2\ell}\right)
 =-\frac{n-1}{t^2}\Id+O\left(\frac1{t^2\ell}\right).
\end{equation*}
Moreover, \eqref{eq:path-uniform-bounds} implies
\begin{equation*}
 \bigl(f(t)^2g_{\lambda(t)}\bigr)^{-1}
 \Ric_{g_{\lambda(t)}}\ge\frac{\mu}{f(t)^2}\Id.
\end{equation*}
After increasing $C$, we therefore obtain
\begin{align}
 \left.\Ric_{\overline g}\right|_{TN\times TN}
 &\ge\left(\frac{\mu}{f(t)^2}
       -\frac{n-1+C/\ell}{t^2}\right)f(t)^2g_{\lambda(t)}
   \ge\frac n{t^2}f(t)^2g_{\lambda(t)}.
 \label{eq:path-tangential-lower}
\end{align}
Indeed, $f(t)\le at$ gives
\begin{equation*}
 \frac{\mu}{f(t)^2}-\frac{n-1+C/\ell}{t^2}
 \ge\frac1{t^2}
 \left(\frac{\mu}{a^2}-n+1-\frac C\ell\right)
 \ge\frac n{t^2},
\end{equation*}
where the last inequality follows from
$\mu/a^2>2n$ once $\ell$ is large enough that $C/\ell<1$.
The mixed component has the exact expression
\begin{align*}
 \Ric_{\overline g}(\partial_t,\mathord\cdot)
 =\frac{\lambda'(t)}2\Bigl(&
 \divsymb_{g_{\lambda(t)}}
   \left(\left.\partial_sg_s\right|_{s=\lambda(t)}\right)-\dd\tr_{g_{\lambda(t)}}\left(\left.\partial_sg_s\right|_{s=\lambda(t)}\right)\Bigr),
\end{align*}
because multiplying a slice metric by the spatially constant factor
$f(t)^2$ does not change its Levi-Civita connection.  Consequently,
\begin{equation*}
 |\Ric_{\overline g}(\partial_t,\mathord\cdot)|_{f(t)^2g_{\lambda(t)}}
 \le\frac{C}{t^2\ell}.
\end{equation*}

By \eqref{eq:path-tangential-lower}, the inverse of the tangential Ricci
block, viewed as a quadratic form on $T^*N$, is bounded above by
 $\frac{t^2}{n}(f(t)^2g_{\lambda(t)})^{-1}$.  Thus
\begin{align*}
 &\Ric_{\overline g}(\partial_t,\partial_t)
 -\left(\left.\Ric_{\overline g}\right|_{TN\times TN}\right)^{-1}
  \bigl(\Ric_{\overline g}(\partial_t,\mathord\cdot),
        \Ric_{\overline g}(\partial_t,\mathord\cdot)\bigr)\ge
 \frac1{8t^2\ell}-\frac{C^2}{nt^2\ell^2}>0
\end{align*}
for sufficiently large $\ell$.  Together with \eqref{eq:path-tangential-lower}, this implies $\Ric_{\overline g}>0$.

Now set $\widehat g=(ae^\ell)^{-2}\overline g$.  Because
$\partial_sg_s=0$ at $s=0,1$, the induced metrics and outward second
fundamental forms at the two boundary components are
\begin{align*}
 \left.\widehat g\right|_{\{e^\ell\}\times N}
 &=g_0,&
 \II^{\widehat g}_{\{e^\ell\}\times N}
 &=-a\left(1-\frac{2\sigma}{\ell}\right)g_0,\\
 \left.\widehat g\right|_{\{e^{2\ell}\}\times N}
 &=\Lambda_{\mathrm{cyl}}^2g_1,&
 \II^{\widehat g}_{\{e^{2\ell}\}\times N}
 &=ae^{-\ell}\left(1-\frac{\sigma}{2\ell}\right)
       \Lambda_{\mathrm{cyl}}^2g_1,
 \qquad \Lambda_{\mathrm{cyl}}=e^{\ell-\sigma}.
\end{align*}
For $\ell>2\sigma$, these satisfy the two required inequalities because
$a<\delta_{\mathrm{cyl}}$.  Pulling back $\widehat g$ by an
orientation-preserving diffeomorphism from $[0,1]$ to
$[e^\ell,e^{2\ell}]$ gives \textup{(i)} and \textup{(ii)}.

Finally, if a compact group of diffeomorphisms preserves every $g_s$,
then it preserves \eqref{eq:path-cylinder-metric} and hence also the
rescaled and pulled-back metric.
\end{proof}

We call the Riemannian cylinder $([0,1]\times N,\widehat g)$ obtained
in Theorem~\ref{thm:path-cylinder-unified} a \emph{Ricci-positive path
cylinder} associated with the path $(g_s)_{s\in[0,1]}$.

\section{The Ricci-positive \(\mathbb Z_2\)-equivariant cobordism }
\label{sec:ricci-positive-equivariant-cobordism}

Recall that the \(\mathbb Z_2\)-equivariant cobordism is the tuple
$(P_m,\sigma_m,\partial_oP_m,I_+,I_-)$, where $\partial_oP_m$ is the
outgoing boundary and the involution $\sigma_m$ exchanges the two incoming
boundaries $I_+$ and $I_-$.  The purpose of this section is to equip this
fixed topological object with the Ricci-positive metric and boundary
geometry required by the mapping torus defined in (\ref{def:mapping-torus}).
We first specify the model involutions on the round boundary spheres.
We then introduce an equivariant angular-neck construction, whose proof
is deferred to Appendix~\ref{app:perelman-neck}, and use it to normalize
the two incoming boundary components.  Finally, an equivariant
Ricci-positive path cylinder rounds the outgoing boundary.
These ingredients are assembled in
Theorem~\ref{thm:ricci-positive-equivariant-cobordism}.  Throughout the section, fix
\begin{equation*}
m\in\{2,3\},
\qquad
k=m-1,
\end{equation*}
and let \(g_{\mathbb S^k}\) denote the unit-round metric on
\(\mathbb S^k\).  The cases \(m=2\) and \(m=3\) produce bases of
dimensions three and four, respectively.  In either case, the
distinguished boundary hypersurfaces have dimension \(m\), while the
auxiliary round factor has dimension \(k\).  The principal ingredients
are the relative equivariant Ricci-positive gluing result,
Lemma~\ref{lem:relative-gluing-unified}, and the path-cylinder
construction, Theorem~\ref{thm:path-cylinder-unified}, established in
Section~\ref{sec:relative-tools}.

Identify $\mathbb R^3=\mathbb C\oplus\mathbb R$ and
$\mathbb R^4=\mathbb C^2$, and define skew-symmetric endomorphisms by
\begin{equation*}
 K_2(z,u)=(\mathrm i z,0),\qquad
 K_3(z_1,z_2)=(\mathrm i z_1,\mathrm i z_2).
\end{equation*}
For $x\in\mathbb S^m$, the vector $K_mx$ is tangent to $\mathbb S^m$, so $K_m$ restricts to a Killing vector field on $\mathbb S^m$.  Its
time-$\tau$ flow is the restriction to $\mathbb S^m$ of the orthogonal
linear map $\exp(\tau K_m)$, where $\exp$ denotes the matrix
exponential.  In particular,
\begin{equation}
\exp(\pi K_2)(z,u)=(-z,u),\qquad
 \exp(\pi K_3)(z_1,z_2)=(-z_1,-z_2).                            \label{eq:unified-base-3-1}
\end{equation}
Thus $\exp(\pi K_2)$ is rotation through angle $\pi$ about the axis
$\{0\}\oplus\mathbb R\subset\mathbb C\oplus\mathbb R$ and fixes the two points
$(0,\pm1)\in\mathbb S^2$, whereas $\exp(\pi K_3)$ is the antipodal map
of $\mathbb S^3$.

These order-two isometries will be the prescribed actions on the round
outer boundary.  We next construct a local Ricci-positive model carrying
the same symmetry.  The local model is expressed in terms of angular
distance from the collapsed axis.  Its level sets provide the inner
hypersurfaces to which the equivariant necks will be attached.

The auxiliary function \(\mathcal R\) is constructed in Lemma~\ref{lem:concave-warping} in Appendix~\ref{app:perelman-neck}.  Its curvature inequalities produce
a translation-invariant family of angular \(m\)-spheres with positively
curved induced metrics and controlled outward second fundamental forms.
The following definition specifies these spheres, and
Proposition~\ref{prop:unified-neck} connects each of them to a round
sphere by an equivariant Ricci-positive neck.

\begin{definition}
On
\begin{equation*}\mathcal Q=\bigl([0,\pi/2)\times\mathbb R\times\mathbb S^k\bigr)/\!\sim,
\end{equation*}
where each set $\{0\}\times\{z\}\times\mathbb S^k$ is collapsed to one
point,
denote by $o_{z_0}$ the point represented by
$\{0\}\times\{z_0\}\times\mathbb S^k$.  On the strip
$|z-z_0|<\pi$, define
\begin{equation*}
 \varrho_{z_0}(t,z)
 =\arccos\bigl(\cos t\cos(z-z_0)\bigr).
\end{equation*}
For $0<r<\pi/2$, the sets
\begin{equation*}
 B_r^{\angle}(o_{z_0})=\{\varrho_{z_0}\le r\},
 \qquad
 S_r^{\angle}(o_{z_0})=\{\varrho_{z_0}=r\}
\end{equation*}
are called, respectively, the angular ball and the angular sphere of
radius $r$ centered at $o_{z_0}$.
\end{definition}
The following proposition is essentially due to Perelman \cite[Section~3 and Assertion]{Perelman1997}, for completeness, we prove it in Appendix~\ref{app:perelman-neck}.
\begin{proposition}
\label{prop:unified-neck}
There is $r_*>0$ with the following property.  Fix
$m\in\{2,3\}$, let $k=m-1$, and choose
\begin{equation*}
             0<r<r_*,\qquad 0<\rho<1,\qquad \pi\rho<r.
\end{equation*}
For every sufficiently small $w>0$ satisfying
\begin{equation}
w^{m-1}<\rho^m,                            \label{eq:unified-base-3-10}
\end{equation}
there is a function $\mathcal R$ satisfying
Lemma~\ref{lem:concave-warping} for which the following assertions hold
on the subset $\{0\le t\le5r\}$ of $\mathcal Q$.

\begin{enumerate}[label=\textup{(\roman*)}]
\item On a neighborhood of $B_{4r}^{\angle}(o_0)$, the metric
\begin{equation}
g_{\mathrm{amb},m}=\cot^2r\left(
 \dd t^2+\cos^2t\,\dd z^2+\mathcal R(t)^2g_{\mathbb S^k}\right)         \label{eq:unified-base-3-11}
\end{equation}
is smooth and satisfies $\sec_{g_{\mathrm{amb},m}}>0$.

\item For every $z_0\in\mathbb R$, the angular sphere
$S_r^{\angle}(o_{z_0})$ is diffeomorphic to $\mathbb S^m$.  Define its
latitude coordinate $\varphi\in[-\pi/2,\pi/2]$ by
\begin{equation*}
 \cos\varphi=\frac{\mathcal R(t)}{\mathcal R(r)},\qquad
 \sgn\varphi=\sgn(z-z_0).
\end{equation*}
In this coordinate its induced metric is independent of $z_0$ and has
the form
\begin{equation}
g_{\mathrm{ang},m}=A_{\mathrm{ang}}(\varphi)^2\dd\varphi^2
           +w^2\cos^2\varphi\,g_{\mathbb S^k}.                   \label{eq:unified-base-3-12}
\end{equation}
 Here $A_{\mathrm{ang}}$ is smooth and positive, with
 $A_{\mathrm{ang}}(\pm\pi/2)=w$, and $\sec_{g_{\mathrm{ang},m}}>1$.  With respect to the normal
pointing out of
$B_r^{\angle}(o_{z_0})$, every principal curvature of
$S_r^{\angle}(o_{z_0})$ lies in $(0,1]$.

\item There are $\ell,\lambda>0$ and a Ricci-positive metric
\begin{equation}
g_{\mathrm{neck},m}=\dd s^2+A(s,\varphi)^2\dd\varphi^2
          +B(s)^2\cos^2\varphi\,g_{\mathbb S^k}                 \label{eq:unified-base-3-13}
\end{equation}
on $[0,\ell]\times\mathbb S^m$.  At $s=0$, its induced metric is
$(\rho/\lambda)^2g_{\mathbb S^m}$ and all principal curvatures with
respect to the outward normal $-\partial_s$ equal $-\lambda$.  At
$s=\ell$, its induced metric is $g_{\mathrm{ang},m}$ and all principal curvatures
with respect to $\partial_s$ are greater than one.  Moreover,
$B'(s)>0$ and $B''(s)<0$ for $0\le s\le\ell$.
\end{enumerate}
The construction commutes with translations of the ambient $z$-coordinate and with orthogonal transformations of $\mathbb S^k$.
\end{proposition}

The neck construction normalizes the two incoming boundary components,
but the metric induced on $\partial_oP_m$ is merely positively curved.
To replace it by an exactly round metric without losing either Ricci
positivity or the involutive symmetry, we use the following equivariant
rounding lemma.

\begin{lemma}
\label{lem:unified-round-path}
Let $m\in\{2,3\}$, let $g$ be a metric of positive sectional
curvature on $\mathbb S^m$, and let a finite group $\Gamma$ act isometrically
on $(\mathbb S^m,g)$.  There is a smooth path
$(g_s)_{s\in[0,1]}$ of $\Gamma$-invariant metrics such that
\begin{equation*}
 g_0=g,\qquad \Ric_{g_s}>0\quad(0\le s\le1),
 \qquad g_1\text{ has constant positive sectional curvature}.
\end{equation*}
The path can be chosen constant for $s$ near $0$ and near $1$.
\end{lemma}

\begin{proof}
Suppose first that $m=2$.  The main convergence theorem in Hamilton's
paper on the Ricci flow of surfaces \cite{Hamilton1988Surfaces} states
that the normalized Ricci flow starting at $g$ converges smoothly
to a metric $g_*$ of constant positive Gaussian curvature.  Because the
two-dimensional Ricci flow is conformal, $g_*$ lies in the conformal
class of $g$.  The normalized flow is
natural under pullback and has a unique solution, since every element of
$\Gamma$ preserves its initial metric, it preserves the entire flow and its
limit $g_*$.  Write $g=e^{2u}g_*$.  The invariance of $g$ and
$g_*$ implies that $u$ is $\Gamma$-invariant.

For $0\le s\le1$, set $h_s=e^{2su}g_*$.  The conformal curvature
formula and the identity
$e^{2u}\sec_g=\sec_{g_*}-\Delta_{g_*}u$ give
\begin{equation*}
\sec_{h_s}=e^{-2su}\bigl((1-s)\sec_{g_*}
                         +s e^{2u}\sec_g\bigr)>0.           
\end{equation*}
Thus $h_s$ joins $g_*$ to $g$ through metrics of positive Gaussian
curvature.  Reversing the parameter and composing it with a smooth
nondecreasing function $[0,1]\to[0,1]$ which is constant near both
endpoints gives the required path from $g$ to $g_*$.

Now suppose that $m=3$.  Positive sectional curvature implies positive
Ricci curvature, so the main theorem of Hamilton's work
\cite{Hamilton1982} applies to the volume-normalized Ricci flow $g(t)$ with $g(0)=g$.  The flow exists for all $t\ge0$, preserves positive Ricci curvature, and converges smoothly to a metric
$g_\infty$ of constant positive sectional curvature.  As in the
two-dimensional case, naturality and uniqueness imply that every
$g(t)$ and $g_\infty$ are $\Gamma$-invariant.

Choose $t_1$ so large that $g(t_1)$ lies in a $\Gamma$-invariant
$C^2$-neighborhood of $g_\infty$ on which Ricci curvature is positive.
The affine segment
\begin{equation*}
 (1-s)g(t_1)+s g_\infty,\qquad 0\le s\le1,
\end{equation*}
then consists of $\Gamma$-invariant Ricci-positive metrics.  Concatenate the
flow segment $(g(t))_{0\le t\le t_1}$ with this affine segment, and
reparametrize both pieces so that they are constant near their
endpoints.  This gives the asserted smooth path.
\end{proof}

We now assemble the preceding ingredients.  The angular region supplies
a ball with two disjoint balls removed, the involution exchanges the two
incoming boundary components, the necks give them the prescribed round
geometry, and the path cylinder rounds the outgoing boundary.

\begin{theorem}
\label{thm:ricci-positive-equivariant-cobordism}
Let $m\in\{2,3\}$, and let
$(P_m,\partial_oP_m,I_+,I_-)$ be the fixed topological model defined in
Section~\ref{sec:double-cable}.  For every $R_*>0$ and
$\varepsilon_*>0$, $P_m$ admits a metric $h_m$ with
$\Ric_{h_m}>0$ and an isometric
involution $\sigma_m:P_m\to P_m$ such that:
\begin{enumerate}[label=\textup{(\roman*)}]
\item $\sigma_m(I_+)=I_-$ and $\sigma_m(I_-)=I_+$,
\item $(I_+,h_m|_{I_+})$ and $(I_-,h_m|_{I_-})$ are round
$m$-spheres of radius $R_*$, if $\II_{I_\pm}$ denotes their second
fundamental form with respect to the outward unit normal of $P_m$, then
\begin{equation}
\II_{I_\pm}>-\varepsilon_*h_m|_{I_\pm},     \label{eq:unified-base-3-73}
\end{equation}
\item $(\partial_oP_m,h_m|_{\partial_oP_m})$ is round, and its second
fundamental form with respect to the outward unit normal is positive
definite,
\item after identifying the round outer boundary with a standard round
$\mathbb S^m$,
\begin{equation}
\sigma_m|_{\partial_oP_m}=\exp(\pi K_m).        \label{eq:unified-base-3-74}
\end{equation}
\end{enumerate}
Here the maps $\exp(\pi K_m)$ are defined in
\eqref{eq:unified-base-3-1}, they are the half-turn of $\mathbb S^2$
about an axis and the antipodal map of $\mathbb S^3$, respectively.
\end{theorem}

\begin{proof}
We realize $P_m$ as an angular region in the ambient model of
Proposition~\ref{prop:unified-neck}, from which two angular balls
exchanged by an involution have been removed.  We then attach
equivariant necks along the two inner boundaries, rescale so that the
inner spheres have radius $R_*$, round the outer boundary with an
equivariant path cylinder, and finally identify the involution on the
round outer boundary.

Choose $0<r<r_*$, where $r_*$ is the constant in
Proposition~\ref{prop:unified-neck}, and then choose
\begin{equation}
0<\rho<\min\{1,r/\pi,R_*\varepsilon_*\}.            \label{eq:unified-base-3-75}
\end{equation}
In particular, $\pi\rho<r$.  Choose $w>0$ so small that $w^{m-1}<\rho^m$, and apply Proposition~\ref{prop:unified-neck}.
In the ambient manifold of Proposition~\ref{prop:unified-neck}, write
\begin{equation*}
 y=o_0,\qquad y_+=o_{2r},\qquad y_-=o_{-2r}.
\end{equation*}
Thus $y$ and $y_\pm$ are the points obtained by collapsing, respectively,
$\{0\}\times\{0\}\times\mathbb S^k$ and
$\{0\}\times\{\pm2r\}\times\mathbb S^k$.  Define
\begin{equation*}
 \Omega_r=B_{4r}^{\angle}(y)\setminus
   \left(\operatorname{int}B_r^{\angle}(y_+)
          \sqcup\operatorname{int}B_r^{\angle}(y_-)\right).
\end{equation*}
Its three boundary components are
\begin{equation*}
 \Sigma_o=S_{4r}^{\angle}(y),\qquad
 \Sigma_+=S_r^{\angle}(y_+),\qquad
 \Sigma_-=S_r^{\angle}(y_-).
\end{equation*}

We first verify the topology of $\Omega_r$.  For the unscaled warped
metric $ \dd t^2+\cos^2t\,\dd z^2+\mathcal R(t)^2g_{\mathbb S^k}$ in \eqref{eq:unified-base-3-11},
projection onto the $(t,z)$-factor with metric
$\dd t^2+\cos^2t\,\dd z^2$ is length nonincreasing, while every radial
geodesic in that factor has a horizontal lift.  Hence the distance from $y=o_0$ is the function $\varrho_0$ in the definition of an angular
ball, where
\begin{equation*}
\cos\varrho_0=\cos t\cos z.              
\end{equation*}
Replacing $z$ by $z\mp2r$ gives the corresponding distance formula
from $y_\pm$.  The factor $\cot^2r$ in
\eqref{eq:unified-base-3-11} multiplies all these distances by
$\cot r$.  Thus the angular balls of angular radius $r$ and $4r$ are
normal balls for $g_{\mathrm{amb},m}$ of Riemannian radius
$r\cot r$ and $4r\cot r$, respectively.  Since the center distances
are $2r\cot r$, $2r\cot r$, and $4r\cot r$, the two closed inner
balls are disjoint, and each lies in the interior of the outer ball.
Consequently, the boundary-decomposed manifold
$(\Omega_r,\Sigma_o,\Sigma_+,\Sigma_-)$ is diffeomorphic to the fixed
model $(P_m,\partial_oP_m,I_+,I_-)$ from
Section~\ref{sec:double-cable}.

We next record the geometry of the outer boundary before any gluing.
With respect to the outward unit normal of $B_{4r}^{\angle}(y)$, the
meridional principal curvature of $\Sigma_o$ is
\begin{equation}
\frac{\cot(4r)}{\cot r}>0,                \label{eq:unified-base-3-77}
\end{equation}
and every vertical principal curvature is
\begin{equation}
\frac1{\cot r}\frac{\mathcal R'(t)}{\mathcal R(t)}
       \frac{\sin t\cos z}{\sin(4r)}>0.                           \label{eq:unified-base-3-78}
\end{equation}
At the two points where the $\mathbb S^k$-factor collapses, the
expression in \eqref{eq:unified-base-3-78} has limit
\eqref{eq:unified-base-3-77}.  Thus $\Sigma_o$ is smooth and strictly
convex.

The ambient metric is invariant under the involution
\begin{equation*}
\iota_2(t,z,(p_1,p_2))=(t,-z,(p_1,-p_2)),\qquad
 \iota_3(t,z,p)=(t,-z,-p),                                      
\end{equation*}
where $(p_1,p_2)\in\mathbb S^1$ in the first formula and
$p\in\mathbb S^2$ in the second.  Each $\iota_m$ is an isometry that fixes $y$, preserves $\Sigma_o$, and exchanges $\Sigma_+$ with $\Sigma_-$.

Take two isometric copies of the neck
$([0,\ell]\times\mathbb S^m,g_{\mathrm{neck},m})$ from
Proposition~\ref{prop:unified-neck}.  Identify their $s=\ell$
boundaries with $\Sigma_+$ and $\Sigma_-$ using the latitude coordinate
in \eqref{eq:unified-base-3-12}, and choose the second identification to
be the image of the first under $\iota_m$.  Along each $\Sigma_\pm$, the
outward second fundamental form of $\Omega_r$ is at least
$-g_{\mathrm{ang},m}$: its outward normal is the negative of the normal
pointing out of the removed angular ball.  At the $s=\ell$ end of each
neck, the outward second fundamental form is strictly greater than
$g_{\mathrm{ang},m}$.  Hence the sum of the two forms is positive definite.

Apply Lemma~\ref{lem:relative-gluing-unified} simultaneously at the two
boundary components.  Attaching the two neck cylinders does not change the diffeomorphism type, so we transport the resulting
Ricci-positive metric to the fixed manifold $P_m$.  We denote this
metric by $h_m$, identify $\Sigma_o$ with $\partial_oP_m$, and identify
the two exposed $s=0$ ends with $I_+$ and $I_-$.  The metric remains
unchanged near all three exposed boundary components.  Equivariance in
the gluing lemma transports $\iota_m$ to an isometric involution
$\sigma_m$ of $P_m$ which exchanges the two necks and hence exchanges
$I_+$ and $I_-$.

At each exposed inner end the radius is $\rho/\lambda$ and the outward
principal curvatures are $-\lambda$.  Multiply the entire metric by
\begin{equation*}
c^2,\qquad
                         c=\frac{R_*\lambda}{\rho}.                
\end{equation*}
Under this homothety, lengths are multiplied by $c$ and principal
curvatures by $c^{-1}$.  Hence both inner radius become $R_*$, while
every inner principal curvature becomes
\begin{equation*}
 -\frac{\lambda}{c}=-\frac{\rho}{R_*}>-\varepsilon_*
\end{equation*}
by \eqref{eq:unified-base-3-75}.  This proves \eqref{eq:unified-base-3-73}, the outer boundary remains strictly convex.

It remains to make the outer metric exactly round.  The sectional
curvatures of the ambient metric near $\Sigma_o$ and the principal
curvatures of $\Sigma_o$ are positive.  The Gauss equation shows that
the metric induced on $\partial_oP_m=\Sigma_o$ has positive sectional
curvature.  Apply Lemma~\ref{lem:unified-round-path} to this metric and
the order-two group generated by $\sigma_m|_{\partial_oP_m}$.  This
gives a $\sigma_m$-invariant Ricci-positive path from the induced metric
to a round metric.

Apply Theorem~\ref{thm:path-cylinder-unified} to this path, choosing
$\delta_{\mathrm{cyl}}>0$ smaller than the least principal curvature of
$\partial_oP_m$.  The second fundamental form at the initial boundary
of the resulting cylinder is greater than
$-\delta_{\mathrm{cyl}}g_0$.  Its sum with the outward second fundamental form of
$\partial_oP_m$ is therefore positive definite.  Lemma
\ref{lem:relative-gluing-unified} attaches this cylinder to
$\partial_oP_m$.  Both constructions are equivariant, so $\sigma_m$
extends isometrically across the cylinder.  The gluing leaves the metric
unchanged near $I_+\sqcup I_-$, and its new outer boundary is round and
strictly convex.  Attaching this cylinder does not change the diffeomorphism type of $P_m$.

Finally, we identify the involution on the round outer boundary.  At
the fixed point $y=o_0$, the tangent space splits naturally as
\begin{equation*}
                         T_y\cong\mathbb R_z\oplus\mathbb R^m,
\end{equation*}
and
\begin{equation*}
d\iota_2|_y=\operatorname{diag}(-1,1,-1),\qquad
 d\iota_3|_y=-I_4.                                                
\end{equation*}
The first map is conjugate to $(-I_2)\oplus(1)$, and the second is
$-I_4$, by \eqref{eq:unified-base-3-1}, these are the linear models of
$\exp(\pi K_2)$ and $\exp(\pi K_3)$.  Since $\Sigma_o$ is the
$g_{\mathrm{amb},m}$-distance sphere of radius $4r\cot r$ centered at
$y$, naturality of the exponential map gives
\begin{equation*}
 \iota_m(\exp_y v)=\exp_y(d\iota_m|_y v).
\end{equation*}
Thus the involution on the original outer boundary is conjugate to
$\exp(\pi K_m)$.

The outer cylinder does not change the underlying action on
$\mathbb S^m$, while its terminal metric is round.  The terminal action
is therefore an orthogonal involution with the same fixed-point set:
two points for $m=2$ and the empty set for $m=3$.  Its $+1$-eigenspace
has dimension one in the first case and dimension zero in the second.
After choosing an orthogonal identification of the terminal round sphere
with the unit round sphere, the
action is $(-I_2)\oplus(1)$ for $m=2$ and $-I_4$ for $m=3$.  This proves
\eqref{eq:unified-base-3-74} and completes the construction.
\end{proof}

\begin{remark}
\label{rem:unified-differences}
In this construction, the dependence on
the dimension $m$ is confined to the following four points.
\begin{enumerate}[label=\textup{(\arabic*)}]
\item A two-plane tangent to the $\mathbb S^k$-factor exists only when
$m=3$.  Its sectional curvature in the ambient metric, the angular
sphere, the slice metric, and the neck metric is given, respectively,
by \eqref{eq:unified-base-3-18}, \eqref{eq:unified-base-3-21},
\eqref{eq:unified-base-3-30}, and \eqref{eq:unified-base-3-42}.
\item The waist hypothesis is $w<\rho^2$ for $m=2$ and
$w^2<\rho^3$ for $m=3$.  In both cases it gives exactly
$\alpha_0<m$.
\item The mixed Ricci component carries the factor $k$, while the
leading coefficient in \eqref{eq:unified-base-3-61} is
$m-\alpha u\eta/f_u$.  The inequality $\alpha<m$ makes this coefficient
positive in both dimensions and yields the determinant estimate
\eqref{eq:unified-base-3-63}.
\item The involution on the collapsing factor is a reflection of
$\mathbb S^1$ for $m=2$ and the antipodal map of $\mathbb S^2$ for
$m=3$.  The induced maps on the outer sphere are, respectively,
$\exp(\pi K_2)$ and $\exp(\pi K_3)$.  Lemma
\ref{lem:unified-round-path} uses Hamilton's convergence theorem for surfaces \cite{Hamilton1988Surfaces} in dimension two and his convergence theorem for three-manifolds \cite{Hamilton1982} in dimension three.
\end{enumerate}
No other step of the local ambient construction, neck interpolation,
smoothing argument, rescaling, or equivariant-cobordism argument depends on the
dimension.
\end{remark}

\section{Ricci-positive geometry of the mapping torus}
\label{sec:unified-cable}

This section constructs the mapping torus $W_m$ defined in (\ref{def:mapping-torus}) from the \(\mathbb Z_2\)-equivariant cobordism of
Theorem~\ref{thm:ricci-positive-equivariant-cobordism}. A circle warping first produces the prescribed second fundamental forms along the boundary. Passing to the quotient preserves the outgoing boundary while merging the two incoming boundary components into a single component. We then identify both boundary components of the quotient with $\mathbb S^1\times\mathbb S^m$ and compute explicitly their induced metrics and corresponding circle classes.

Fix \(m\in\{2,3\}\), \(R_*>0\), and \(\varepsilon_*>0\).  Let
\((P_m,h_m,\sigma_m)\) be supplied by
Theorem~\ref{thm:ricci-positive-equivariant-cobordism}.  Thus
\begin{equation*}
             \partial P_m=\partial_oP_m\sqcup I_+\sqcup I_-,
\end{equation*}
the involution $\sigma_m$ exchanges $I_+$ and $I_-$, and the two incoming
boundary components are round $m$-spheres of radius $R_*$ satisfying
\begin{equation*}
 \II_{I_\pm}>-\varepsilon_*h_m|_{I_\pm}
\end{equation*}
with respect to the outward unit normal.  Denote the radius of the
strictly convex round boundary $\partial_oP_m$ by $R_o$.  Under the
round identification in Theorem~\ref{thm:ricci-positive-equivariant-cobordism},
\begin{equation*}
 \sigma_m|_{\partial_oP_m}=\exp(\pi K_m).
\end{equation*}

\subsection{An invariant warped-product metric}

The boundary values of the following warping function determine the
circle lengths in the quotient, while its outward logarithmic
derivatives determine the circle-direction components of the second
fundamental forms.

\begin{lemma}\label{lem:un-warp}
Let \((P_m,h_m,\sigma_m)\) be furnished by
Theorem~\ref{thm:ricci-positive-equivariant-cobordism}.  Given \(\rho_*>0\) and
\(\delta_*>0\), there are a positive \(\sigma_m\)-invariant function \(f\) on \(P_m\) and a constant \(\alpha>0\)
such that
\begin{equation*}
                         G=h_m+f^2\dd t^2
\end{equation*}
has \(\Ric_G>0\), and
\begin{align*}
 f|_{\partial_oP_m}&=\rho_*(1+\alpha),&
 \partial_\nu\log f|_{\partial_oP_m}&>0,
 \\
 f|_{I_\pm}&=\rho_*,&
 \partial_\nu\log f|_{I_\pm}&>-\delta_*.
\end{align*}
Moreover, if \(X,Y\) are tangent to a boundary component of \(P_m\),
then the second fundamental form of the
corresponding boundary of \((\mathbb R\times P_m,G)\), with respect to
the lift of the outward unit normal \(\nu\), satisfies
\begin{equation}\label{eq:un-warp-second-fundamental-form}
 \II_G(X,Y)=\II_{h_m}(X,Y),\qquad
 \II_G(f^{-1}\partial_t,f^{-1}\partial_t)=\partial_\nu\log f,
 \qquad \II_G(X,f^{-1}\partial_t)=0.
\end{equation}
\end{lemma}

\begin{proof}
Let \(\psi\) be the solution of
\begin{equation*}
 \Delta_{h_m}\psi=0,\quad \psi|_{\partial_oP_m}=1,
 \quad \psi|_{I_+\cup I_-}=0,
\end{equation*}
and let \(\varphi\) be the solution of
\begin{equation*}
 \Delta_{h_m}\varphi=-1,\quad \varphi|_{\partial P_m}=0.
\end{equation*}
Since \(\sigma_m\) is an isometry, preserves \(\partial_oP_m\), and
interchanges \(I_+\) and \(I_-\), the functions
\(\psi\circ\sigma_m\) and \(\varphi\circ\sigma_m\) solve the same Dirichlet
problems as \(\psi\) and \(\varphi\), respectively.  Uniqueness therefore gives
\begin{equation*}
                       \psi\circ\sigma_m=\psi,
                 \qquad \varphi\circ\sigma_m=\varphi.
\end{equation*}
The strong maximum principle gives \(0<\psi<1\) and \(\varphi>0\) in
\(\mathring{P_m}\).  Applying the Hopf boundary lemma at the maximum and
minimum boundary components of \(\psi\), and at the boundary minimum of
\(\varphi\), gives
\begin{equation*}
 \psi_\nu>0\ \hbox{on }\partial_oP_m,\qquad
 \psi_\nu<0\ \hbox{on }I_\pm,\qquad
 \varphi_\nu<0\ \hbox{on }\partial P_m.
\end{equation*}

For \(\alpha>0\), define
\begin{equation*}
f=\rho_*(1+\alpha\psi+\alpha^2\varphi).
\end{equation*}
This function is positive and \(\sigma_m\)-invariant.  Since \(\varphi=0\) on
\(\partial P_m\), its boundary values are
\begin{equation*}
 f|_{\partial_oP_m}=\rho_*(1+\alpha),
 \qquad f|_{I_\pm}=\rho_*.
\end{equation*}
On the outer boundary,
\begin{equation*}
 \partial_\nu\log f
 =\frac{\alpha(\psi_\nu+\alpha\varphi_\nu)}{1+\alpha}.
\end{equation*}
Because \(\psi_\nu\) has a positive minimum and \(\varphi_\nu\) is bounded on
the compact boundary \(\partial_oP_m\), this expression is positive for
all sufficiently small \(\alpha>0\).
On either inner boundary,
\begin{equation*}
 \partial_\nu\log f
 =\alpha\psi_\nu+\alpha^2\varphi_\nu.
\end{equation*}
The last expression converges uniformly to zero as \(\alpha\to0\).
Thus, after decreasing \(\alpha\), it is greater than \(-\delta_*\) on
both \(I_+\) and \(I_-\).

It remains to verify the Ricci inequality.  For vectors \(X,Y\) tangent
to \(P_m\), the Ricci tensor of the one-dimensional warped product is
\begin{align*}
 \Ric_G(X,Y)
 &=\Ric_{h_m}(X,Y)-f^{-1}\Hess_{h_m}f(X,Y),\\
 \Ric_G(X,f^{-1}\partial_t)&=0,\\
 \Ric_G(f^{-1}\partial_t,f^{-1}\partial_t)
 &=-f^{-1}\Delta_{h_m}f.
\end{align*}
Here
\begin{equation*}
 \Hess_{h_m}f=\rho_*\alpha\bigl(\Hess_{h_m}\psi+\alpha\Hess_{h_m}\varphi\bigr),
 \qquad
 -\Delta_{h_m}f=\rho_*\alpha^2>0.
\end{equation*}
Since $P_m$ is compact and $\operatorname{Ric}_{h_m}>0$, there exists
$\kappa>0$ such that $\operatorname{Ric}_{h_m}\geq \kappa h_m$.
Moreover, for all sufficiently small $\alpha>0$,
\begin{equation*}
\left\|f^{-1}\operatorname{Hess}_{h_m}f\right\|_{h_m}
\leq C\alpha.
\end{equation*}
Thus
\begin{equation*}
\operatorname{Ric}_G(X,X)
\geq (\kappa-C\alpha)|X|_{h_m}^2>0
\end{equation*}
for every nonzero $X\in TP_m$. On the other hand,
\begin{equation*}
\operatorname{Ric}_G(f^{-1}\partial_t,f^{-1}\partial_t)
=-f^{-1}\Delta_{h_m}f
=\frac{\rho_*\alpha^2}{f}>0.
\end{equation*}
Since $\operatorname{Ric}_G(X,f^{-1}\partial_t)=0$, the Ricci tensor is block diagonal
with respect to $TP_m\oplus\mathbb R\partial_t$, and hence
$\operatorname{Ric}_G>0$.

Finally, the standard Levi--Civita formulas for a one-dimensional
warped product give
\begin{equation*}
 \nabla_X\nu=\nabla^{h_m}_X\nu,
 \qquad
 \nabla_{f^{-1}\partial_t}\nu
   =(\partial_\nu\log f)f^{-1}\partial_t.
\end{equation*}
Their tangential components are exactly the three identities in
\eqref{eq:un-warp-second-fundamental-form}.
\end{proof}

\subsection{Boundary geometry and distinguished loop classes}

Fix $T>0$, and let $\mathcal T$ and $W_m$ be the isometry and mapping
torus defined in (\ref{def:mapping-torus}).  Since $h_m$ and $f$ are $\sigma_m$-invariant, $\mathcal T$ is an isometry of $G$, and we equip $W_m$ with the induced metric.  On $\mathbb R\times\partial_oP_m$ its monodromy
is $\exp(\pi K_m)$.  It exchanges
$\mathbb R\times I_+$ and $\mathbb R\times I_-$, and its first return to
$\mathbb R\times I_+$ is
\begin{equation*}
 \mathcal T^2(t,x)=(t+2T,x).
\end{equation*}
Consequently the quotient has one outer and one inner boundary
component, the outer circle has period $T$, whereas the inner circle
has period $2T$.  The following product metrics make this distinction
precise.

For \(\ell,\rho>0\) and \(\omega\in\mathbb R\), define a metric
\(g_{\ell,\rho,\omega}\) on
\(\mathbb S^1\times\mathbb S^m\), where \(\theta\) has period \(2\pi\),
by prescribing the squared norm of each tangent vector:
\begin{equation}\label{eq:un-boundary-metric-definition}
 g_{\ell,\rho,\omega}(b\partial_\theta+Y,b\partial_\theta+Y)
 =\ell^2b^2+\rho^2
   \lvert Y+b\omega K_m(y)\rvert_{g_{\mathbb S^m}}^2,
 \qquad Y\in T_y\mathbb S^m.
\end{equation}
This definition is global even when \(K_m\) has zeros.

\begin{lemma}
\label{lem:un-boundary-components}
The boundary of \(W_m\) is the disjoint union
\begin{equation*}
                         \partial W_m=\partial_oW_m\sqcup\partial_iW_m,
\end{equation*}
where
\begin{align*}
 \partial_oW_m
 &=(\mathbb R\times\partial_oP_m)/
   \bigl((t,x)\sim(t+T,\sigma_mx)\bigr),\\
 \partial_iW_m
 &=(\mathbb R\times(I_+\sqcup I_-))/
   \bigl((t,x)\sim(t+T,\sigma_mx)\bigr).
\end{align*}
Use the round identifications
\(\partial_oP_m\cong\mathbb S^m_{R_o}\) and
\(I_+\cong\mathbb S^m_{R_*}\) fixed above.  Then the maps
\begin{align}
 \Phi_o:\mathbb S^1\times\mathbb S^m&\longrightarrow\partial_oW_m,
 &\Phi_o(\theta,y)
   &=\left[\frac{T\theta}{2\pi},
       \exp\left(\frac\theta2K_m\right)y\right],
 \label{eq:un-outer-product-map}\\
 \Phi_i:\mathbb S^1\times\mathbb S^m&\longrightarrow\partial_iW_m,
 &\Phi_i(\theta,y)&=\left[\frac{T\theta}{\pi},y\right]
 \label{eq:un-inner-product-map}
\end{align}
are diffeomorphisms.  We call them the boundary product parametrizations of $W_m$.
Moreover,
\begin{align}
 \Phi_o^*(G|_{\partial_oW_m})&=g_{\ell_o,R_o,1/2},&
 \ell_o&=\frac{T\rho_*(1+\alpha)}{2\pi},
 \label{eq:un-outer-boundary-metric}\\
 \Phi_i^*(G|_{\partial_iW_m})&=g_{\ell_{\mathrm{in}},R_*,0},&
 \ell_{\mathrm{in}}&=\frac{T\rho_*}{\pi}.
 \label{eq:un-inner-boundary-metric}
\end{align}
In particular,
\begin{equation}\label{eq:un-axial-ratio}
                         \ell_o=\frac{1+\alpha}{2}\ell_{\mathrm{in}}.
\end{equation}
Finally, \(\pi_1(W_m)\cong\mathbb Z\).  For any \(y\in\mathbb S^m\),
the loop \(\theta\mapsto\Phi_o(\theta,y)\) represents a generator, while
\(\theta\mapsto\Phi_i(\theta,y)\) represents the square of that
generator.
\end{lemma}

\begin{proof}
The action of \(\mathbb Z\) on \(\mathbb R\times P_m\) generated by
\((t,x)\mapsto(t+T,\sigma_m x)\) is free and properly discontinuous.
A fundamental domain is \([0,T]\times P_m\), with
\((T,x)\) identified with \((0,\sigma_m x)\).  Thus \(W_m\) is compact,
and the Ricci-positive metric \(G\) descends to it.

The action preserves \(\mathbb R\times\partial_oP_m\) and \(\mathbb R\times(I_+\sqcup I_-)\), and it interchanges the two components of the latter.  Taking their quotients gives precisely the two
boundary components stated in the lemma.

We first show that \(\Phi_o\) is a well-defined diffeomorphism.  Replacing \(\theta\) by
\(\theta+2\pi\)
changes the displayed representative to
\begin{equation*}
 \left[\frac{T\theta}{2\pi}+T,
   \sigma_m\exp\left(\frac\theta2K_m\right)y\right],
\end{equation*}
which represents the same point in the quotient.  Conversely, every
class in \(\partial_oW_m\) has a representative with \(0\le t<T\),
setting \(\theta=2\pi t/T\) and applying
\(\exp(-\theta K_m/2)\) to its sphere coordinate gives the inverse of
\(\Phi_o\).  Hence \(\Phi_o\) is a diffeomorphism.

For tangent vectors \(b\partial_\theta+Y\) and
\(c\partial_\theta+Z\), the differential of
\(\Phi_o\), after applying the inverse of the
sphere isometry \(\exp(\theta K_m/2)\), gives
\begin{equation*}
 \frac{Tb}{2\pi}\partial_t+Y+\frac b2K_m,
 \qquad
 \frac{Tc}{2\pi}\partial_t+Z+\frac c2K_m.
\end{equation*}
On \(\partial_oP_m\), Lemma~\ref{lem:un-warp} gives
\(f=\rho_*(1+\alpha)\), while the sphere metric is
\(R_o^2g_{\mathbb S^m}\).  The pullback of \(G\) is therefore
\begin{equation*}
 \left(\frac{T\rho_*(1+\alpha)}{2\pi}\right)^2bc
 +R_o^2g_{\mathbb S^m}
   \left(Y+\frac b2K_m,Z+\frac c2K_m\right),
\end{equation*}
which proves \eqref{eq:un-outer-boundary-metric}.

For the inner boundary, fix the round identification
\(I_+\cong\mathbb S^m_{R_*}\), and use \(\sigma_m:I_+\to I_-\) for the
corresponding identification of \(I_-\).  Every equivalence class in
the quotient of \(\mathbb R\times(I_+\sqcup I_-)\) has a representative
in \(\mathbb R\times I_+\).  Two representatives in
\(\mathbb R\times I_+\) determine the same class precisely when their
first coordinates differ by an even multiple of \(T\), because one
application of the defining identification moves \(I_+\) to \(I_-\),
whereas two applications return to \(I_+\).  Thus \(\Phi_i\) in
\eqref{eq:un-inner-product-map} is a well-defined
diffeomorphism.  Since \(f=\rho_*\) and the sphere metric is
\(R_*^2g_{\mathbb S^m}\) on \(I_+\), its pullback metric is
\begin{equation*}
 \left(\frac{T\rho_*}{\pi}\right)^2\dd\theta^2
   +R_*^2g_{\mathbb S^m}=g_{\ell_{\mathrm{in}},R_*,0},
 \qquad \ell_{\mathrm{in}}=\frac{T\rho_*}{\pi}.
\end{equation*}
This proves \eqref{eq:un-inner-boundary-metric}, and
\eqref{eq:un-axial-ratio} follows immediately.

It remains to identify the two circle classes.  By the definition of
the fixed model $P_m$ in Section~\ref{sec:double-cable}, attaching an
\((m+1)\)-disk to each of \(I_+\) and \(I_-\) produces $D^{m+1}$.
Since \(m\ge2\), the attaching spheres \(I_+\) and
\(I_-\) are simply connected.  Applying the Seifert--van Kampen theorem
at the two attachments therefore shows that
\(\pi_1(P_m)=0\).

The quotient map makes \(W_m\) a fiber bundle over the circle, with
projection
\begin{equation*}
 W_m\longrightarrow\mathbb R/T\mathbb Z,
 \qquad [t,x]\longmapsto[t],
\end{equation*}
and fiber \(P_m\).  Its homotopy exact sequence gives an isomorphism
\(\pi_1(W_m)\cong\pi_1(\mathbb R/T\mathbb Z)\cong\mathbb Z\).  Under
\(\Phi_o\), the loop obtained by holding \(y\) fixed projects once around
\(\mathbb R/T\mathbb Z\), and hence represents a generator.  Under
\(\Phi_i\), the corresponding loop changes \(t\) by \(2T\), so it projects
twice around the base and represents the square of the first loop.
\end{proof}

\section{A Ricci-positive holonomy interpolation cylinder}
\label{sec:holonomy-interpolation-cylinder}

Fix \(m\in\{2,3\}\).  We now construct the holonomy interpolation cylinder $\mathcal C$ defined in (\ref{def:holonomy-interpolation-cylinder}) and prove that its metric has positive Ricci
curvature.  It joins the twisted outer boundary of the mapping torus to
the untwisted product boundary required by the next block.  By
Lemma~\ref{lem:un-boundary-components}, the outer boundary of
\(W_m\) has metric \(g_{\ell_o,R_o,1/2}\).  The product identification
\(\Phi_o\) reduces the geometric problem to a path in the family
\(g_{\ell,\rho,\omega}\) on
\(\mathbb S^1\times\mathbb S^m\): the functions \(\ell\) and \(\rho\)
specify the circle and sphere scales, while \(\omega\) specifies the
holonomy.

We require the twist to change from \(\omega(0)=1/2\) to
\(\omega(s_*)=0\), while the sphere radius returns to its initial value
at \(s=s_*\).  We first compute the Ricci tensor and the
second fundamental forms of the two boundaries exactly.  We then choose \(\omega\) by a
weighted least-energy principle, which eliminates the axial--sphere
Ricci components, and choose \(\ell\) and \(\rho\) so that the positive
axial and spherical terms dominate the twisting errors.  Finally, we
verify the second-fundamental-form inequality needed to attach the
left boundary of $\mathcal C$ to \(\partial_oW_m\).

Let \(s_*>0\), let \(\ell,\rho:[0,s_*]\to(0,\infty)\), and let
\(\omega:[0,s_*]\to\mathbb R\) be smooth.  Set
\begin{equation*}
 \mathcal C=[0,s_*]\times\mathbb S^1\times\mathbb S^m
\end{equation*}
and define on $\mathcal C$
\begin{equation}\label{eq:interpolation-cylinder-metric}
 g_{\mathcal C}=\dd s^2+g_{\ell(s),\rho(s),\omega(s)}.
\end{equation}
When
\begin{equation*}
 \omega(0)=\frac12,\qquad \omega(s_*)=0,
 \qquad \rho(s_*)=\rho(0),
\end{equation*}
we call $(\mathcal C,g_{\mathcal C})$ a \emph{holonomy interpolation cylinder}: its
left boundary has half-turn holonomy, its right boundary is an untwisted
product, and the sphere radius of the two boundary metrics agree.
Explicitly, for \(Y\in T_y\mathbb S^m\),
\begin{equation*}
 g_{\mathcal C}(b\partial_\theta+Y,b\partial_\theta+Y)
 =\ell^2b^2+\rho^2
   \lvert Y+b\omega K_m(y)\rvert_{g_{\mathbb S^m}}^2.
\end{equation*}
Define the logarithmic derivatives and the normalized twist rate by
\begin{equation}\label{eq:un-log-derivatives}
 \mathfrak p=\frac{\ell'}{\ell},
 \qquad \mathfrak q=\frac{\rho'}{\rho},
 \qquad \vartheta=\frac{\rho\omega'}{\ell}.
\end{equation}
Let \(U_1,\ldots,U_m\) be a local
\(g_{\mathbb S^m}\)-orthonormal frame.  The coordinate column of the
Killing field \(K_m\) in this frame is
\begin{equation*}
 K_{m,i}=g_{\mathbb S^m}(K_m,U_i),
 \qquad
 \mathbf K_m=(K_{m,1},\ldots,K_{m,m})^{\mathsf T}.
\end{equation*}
Then
\(\lvert\mathbf K_m\rvert^2
=\lvert K_m\rvert_{g_{\mathbb S^m}}^2\le1\), and the following frame is
\(g_{\mathcal C}\)-orthonormal:
\begin{equation}\label{eq:un-adapted-frame}
 E_0=\partial_s,
 \qquad
 E_1=\ell^{-1}(\partial_\theta-\omega K_m),
 \qquad
 E_{i+1}=\rho^{-1}U_i\quad(1\le i\le m).
\end{equation}

\begin{lemma}
\label{lem:un-ricci}
Write
\(\Ric_{ij}=\Ric_{g_{\mathcal C}}(E_i,E_j)\).  In the frame
\eqref{eq:un-adapted-frame}, the potentially nonzero components not
determined by symmetry are
\begin{align}
 \Ric_{00}
  &=-(\mathfrak p'+\mathfrak p^2)-m(\mathfrak q'+\mathfrak q^2)
    -\frac{\vartheta^2}{2}\lvert\mathbf K_m\rvert^2,
  \label{eq:un-ric00}\\
 \Ric_{11}
  &=-\bigl[\mathfrak p'+\mathfrak p(\mathfrak p+m\mathfrak q)\bigr]
    -\frac{\vartheta^2}{2}\lvert\mathbf K_m\rvert^2,
  \label{eq:un-ric11}\\
 \Ric_{i+1,j+1}
  &=\left\{\frac{m-1}{\rho^2}
    -\bigl[\mathfrak q'+\mathfrak q(\mathfrak p+m\mathfrak q)\bigr]\right\}\delta_{ij}
    +\frac{\vartheta^2}{2}K_{m,i}K_{m,j},
  \label{eq:un-ricij}\\
 \Ric_{1,i+1}
  &=-\frac12\bigl[\vartheta'+(m+1)\mathfrak q\vartheta\bigr]K_{m,i}
   =-\frac{K_{m,i}}{2\rho^{m+1}}
       \left(\frac{\rho^{m+2}\omega'}{\ell}\right)'.
  \label{eq:un-ric1i}
\end{align}
Here \(1\le i,j\le m\), and
\begin{equation*}
 \Ric_{0i}=0\qquad(1\le i\le m+1).
\end{equation*}
With respect to the unit normal \(\partial_s\),
the second fundamental form of the hypersurface
\(\{s\}\times\mathbb S^1\times\mathbb S^m\) satisfies
\begin{align}
 \II_s(E_1,E_1)&=\mathfrak p,\nonumber\\
 \II_s(E_1,E_{i+1})&=\frac{\vartheta}{2}K_{m,i},
 \label{eq:un-shape}\\
 \II_s(E_{i+1},E_{j+1})&=\mathfrak q\delta_{ij}.\nonumber
\end{align}
These identities hold for every local orthonormal frame on
\(\mathbb S^m\), including at the zeros of \(K_m\), which occur when
\(m=2\).
\end{lemma}

\begin{proof}
Fix \(s\).  Since the flow of \(K_m\) consists of isometries of the
unit round sphere, on every coordinate neighborhood the local
diffeomorphism
\begin{equation*}
 (\theta,y)\longmapsto
 \bigl(\theta,\exp(\omega(s)\theta K_m)y\bigr)
\end{equation*}
pulls \(\ell^2\dd\theta^2+\rho^2g_{\mathbb S^m}\) back to
\(g_{\ell,\rho,\omega}\).  Therefore the intrinsic Ricci tensor of the
\(s\)-slice
satisfies
\begin{equation}\label{eq:un-slice-ricci}
 \Ric_{g_{\ell,\rho,\omega}}(E_1,E_1)=0,
 \qquad
 \Ric_{g_{\ell,\rho,\omega}}(E_{i+1},E_{j+1})
 =\frac{m-1}{\rho^2}\delta_{ij},
\end{equation}
and its axial--sphere components vanish.

In the Gaussian coordinate \(s\), the second fundamental form with respect to the normal \(\partial_s\) is one half of the \(s\)-derivative of the slice
metric.  Evaluating this derivative on the frame
\eqref{eq:un-adapted-frame} gives
\begin{align*}
 \frac12(\partial_s g_{\ell,\rho,\omega})(E_1,E_1)
 &=\frac{\ell'}{\ell}=\mathfrak p,\\
 \frac12(\partial_s g_{\ell,\rho,\omega})(E_1,E_{i+1})
 =\frac{\rho\omega'}{2\ell}K_{m,i}
 &=\frac{\vartheta}{2}K_{m,i},\\
 \frac12(\partial_s g_{\ell,\rho,\omega})(E_{i+1},E_{j+1})
 =\frac{\rho'}{\rho}\delta_{ij}
 &=\mathfrak q\delta_{ij},
\end{align*}
which proves \eqref{eq:un-shape}.  We now apply the Gaussian-coordinate identity \(\nabla_{\partial_s}X=\partial_sX+\II_s^\sharp X\) for tangential vector fields to the frame \eqref{eq:un-adapted-frame}.  Extend \(U_i\) independently of \(s\).
Then
\begin{equation*}
 \partial_sE_1=-\mathfrak p E_1-\vartheta\sum_i K_{m,i}E_{i+1},
 \qquad
 \partial_sE_{i+1}=-\mathfrak q E_{i+1},
\end{equation*}
whereas \eqref{eq:un-shape} gives
\begin{equation*}
 \II_s^\sharp E_1=\mathfrak p E_1+\frac{\vartheta}{2}
   \sum_i K_{m,i}E_{i+1},
 \qquad
 \II_s^\sharp E_{i+1}
 =\frac{\vartheta}{2}K_{m,i}E_1+\mathfrak q E_{i+1}.
\end{equation*}
Adding the corresponding terms yields
\begin{equation}\label{eq:un-normal-rotation}
 \nabla_{E_0}E_1=-\frac{\vartheta}{2}\sum_i K_{m,i}E_{i+1},
 \qquad
 \nabla_{E_0}E_{i+1}=\frac{\vartheta}{2}K_{m,i}E_1.
\end{equation}
Consequently,
\begin{align*}
 (\nabla_{E_0}\II_s)(E_1,E_1)
 &=\mathfrak p'+\frac{\vartheta^2}{2}\lvert\mathbf K_m\rvert^2,\\
 (\nabla_{E_0}\II_s)(E_{i+1},E_{j+1})
 &=\mathfrak q'\delta_{ij}-\frac{\vartheta^2}{2}K_{m,i}K_{m,j},\\
 (\nabla_{E_0}\II_s)(E_1,E_{i+1})
 &=\frac12\bigl[\vartheta'+(\mathfrak q-\mathfrak p)\vartheta\bigr]K_{m,i}.
\end{align*}
Indeed, the definition of the covariant derivative of a symmetric
two-tensor and \eqref{eq:un-normal-rotation} give
\begin{align*}
 (\nabla_{E_0}\II_s)(E_1,E_1)
 &=\mathfrak p'
   -2\II_s\left(-\frac{\vartheta}{2}
       \sum_r K_{m,r}E_{r+1},E_1\right)
 =\mathfrak p'+\frac{\vartheta^2}{2}\lvert\mathbf K_m\rvert^2,\\
 (\nabla_{E_0}\II_s)(E_{i+1},E_{j+1})
 &=\mathfrak q'\delta_{ij}
   -\II_s\left(\frac{\vartheta}{2}K_{m,i}E_1,E_{j+1}\right)
   -\II_s\left(E_{i+1},\frac{\vartheta}{2}K_{m,j}E_1\right)\\
 &=\mathfrak q'\delta_{ij}-\frac{\vartheta^2}{2}K_{m,i}K_{m,j},\\
 (\nabla_{E_0}\II_s)(E_1,E_{i+1})
 &=\frac{\vartheta'}2 K_{m,i}
   -\II_s\left(-\frac{\vartheta}{2}
       \sum_r K_{m,r}E_{r+1},E_{i+1}\right)
   -\II_s\left(E_1,\frac{\vartheta}{2}K_{m,i}E_1\right)\\
 &=\frac12\bigl[\vartheta'+(\mathfrak q-\mathfrak p)\vartheta\bigr]K_{m,i}.
\end{align*}
Moreover,
\begin{equation*}
 \operatorname{tr}_{g_{\ell,\rho,\omega}}\II_s=\mathfrak p+m\mathfrak q,
 \qquad
 |\II_s|^2=\mathfrak p^2+m\mathfrak q^2
   +\frac{\vartheta^2}{2}\lvert\mathbf K_m\rvert^2.
\end{equation*}
The coefficient \(\vartheta^2/2\) is the sum of the two equal contributions
from the symmetric axial--sphere entries of \(\II_s\).

For vectors \(X,Y\) tangent to the \(s\)-slice, the
identities in Lemma~\ref{lem:gaussian-ricci-unified} can be written as
\begin{align*}
 \Ric_{g_{\mathcal C}}(E_0,E_0)
 &=-\partial_s
   (\operatorname{tr}_{g_{\ell,\rho,\omega}}\II_s)-|\II_s|^2,\\
 \Ric_{g_{\mathcal C}}(X,Y)
 &=\Ric_{g_{\ell,\rho,\omega}}(X,Y)
  -(\nabla_{E_0}\II_s)(X,Y)
  -(\operatorname{tr}_{g_{\ell,\rho,\omega}}\II_s)\II_s(X,Y).
\end{align*}
For the normal component this gives
\begin{equation*}
 \Ric_{00}
 =-(\mathfrak p'+m\mathfrak q')
  -\left(\mathfrak p^2+m\mathfrak q^2
          +\frac{\vartheta^2}{2}\lvert\mathbf K_m\rvert^2\right),
\end{equation*}
which is \eqref{eq:un-ric00}.  For the tangential components,
\eqref{eq:un-slice-ricci} and
the three preceding identities for \(\nabla_{E_0}\II_s\)
give
\begin{align*}
 \Ric_{11}
 &=0-\left(\mathfrak p'
       +\frac{\vartheta^2}{2}\lvert\mathbf K_m\rvert^2\right)
   -(\mathfrak p+m\mathfrak q)\mathfrak p,\\
 \Ric_{i+1,j+1}
 &=\frac{m-1}{\rho^2}\delta_{ij}
   -\left(\mathfrak q'\delta_{ij}
          -\frac{\vartheta^2}{2}K_{m,i}K_{m,j}\right)
   -(\mathfrak p+m\mathfrak q)\mathfrak q\delta_{ij},\\
 \Ric_{1,i+1}
 &=0-\frac12\bigl[\vartheta'+(\mathfrak q-\mathfrak p)\vartheta\bigr]K_{m,i}
   -\frac12(\mathfrak p+m\mathfrak q)\vartheta K_{m,i}\\
 &=-\frac12\bigl[\vartheta'+(m+1)\mathfrak q\vartheta\bigr]K_{m,i}.
\end{align*}
These are \eqref{eq:un-ric11}--\eqref{eq:un-ric1i}.

It remains to compute the components containing exactly one \(E_0\).  Under the
local diffeomorphism used above, the slice metric becomes the product
metric on \(\mathbb S^1_\ell\times\mathbb S^m_\rho\), and the second
fundamental form becomes
\begin{equation*}
 \mathfrak p\ell^2\dd\theta^2+\mathfrak q\rho^2g_{\mathbb S^m}
 +\frac{\rho^2\omega'}2\left(
   \dd\theta\otimes g_{\mathbb S^m}(K_m,\mathord\cdot)
   +g_{\mathbb S^m}(K_m,\mathord\cdot)\otimes\dd\theta\right).
\end{equation*}
The first two summands are parallel on the product.  The divergence of
the last summand is zero: \(\dd\theta\) is parallel, \(K_m\) is
independent of \(\theta\), and
\(\operatorname{div}_{g_{\mathbb S^m}}K_m=0\) because \(K_m\) is a
Killing field.  Since
\(\operatorname{tr}_{g_{\ell,\rho,\omega}}\II_s
=\mathfrak p+m\mathfrak q\) is constant on each slice,
the contracted Codazzi identity
\begin{equation*}
 \Ric_{g_{\mathcal C}}(E_0,\mathord\cdot)
 =\operatorname{div}_{g_{\ell,\rho,\omega}}\II_s
  -\dd(\operatorname{tr}_{g_{\ell,\rho,\omega}}\II_s)
\end{equation*}
implies \(\Ric_{g_{\mathcal C}}(E_0,E_\alpha)=0\) for every \(\alpha\ge1\).
Finally,
\begin{equation*}
 \frac1{\rho^{m+1}}
 \left(\frac{\rho^{m+2}\omega'}{\ell}\right)'
 =\frac1{\rho^{m+1}}(\rho^{m+1}\vartheta)'
 =\vartheta'+(m+1)\mathfrak q\vartheta,
\end{equation*}
which proves the second expression in \eqref{eq:un-ric1i}.
\end{proof}

Equation~\eqref{eq:un-ric1i} isolates the only Ricci components coupling
the circle and sphere directions.  They vanish identically precisely
when \(\rho^{m+2}\omega'/\ell\) is constant.  This condition is the
Euler--Lagrange equation of the natural weighted twisting energy below.
The variational characterization determines \(\omega\) uniquely once
\(\ell,\rho\) and the two endpoint values of \(\omega\) have been
prescribed.

\begin{lemma}\label{lem:least-energy-twist}
Let $\ell,\rho\in C^\infty([0,s_*],(0,\infty))$ and let
$\omega_-,\omega_+\in\mathbb R$.  Among all
$\omega\in H^1([0,s_*])$ with $\omega(0)=\omega_-$ and
$\omega(s_*)=\omega_+$, the weighted Dirichlet energy functional
\begin{equation}\label{eq:weighted Dirichlet energy}
\mathcal E(\omega)=\int_0^{s_*}
 \frac{\rho(s)^{m+2}}{\ell(s)}\bigl(\omega'(s)\bigr)^2\,\dd s
\end{equation}
has the unique minimizer
\begin{equation*}
 \omega_*(s)=\omega_-+(\omega_+-\omega_-)
 \frac{\displaystyle\int_0^s
        \frac{\ell(t)}{\rho(t)^{m+2}}\,\dd t}
      {\displaystyle\int_0^{s_*}
        \frac{\ell(t)}{\rho(t)^{m+2}}\,\dd t}.
\end{equation*}
Equivalently, it is the unique solution of
\begin{equation*}
 \left(\frac{\rho^{m+2}\omega'}{\ell}\right)'=0.
\end{equation*}
\end{lemma}

\begin{proof}
For every $\eta\in H_0^1([0,s_*])$, the first variation at $\omega$ in
the direction $\eta$ is
\begin{equation*}
 D\mathcal E_\omega[\eta]
 =2\int_0^{s_*}\frac{\rho^{m+2}}\ell\omega'\eta'\,\dd s
 =-2\int_0^{s_*}
   \left(\frac{\rho^{m+2}\omega'}\ell\right)'\eta\,\dd s.
\end{equation*}
Thus the Euler--Lagrange equation is the displayed differential
equation in the statement.  Integrating it and imposing the two
endpoint values gives the stated formula for $\omega_*$.  For every
$\eta\in H_0^1([0,s_*])$,
\begin{equation*}
 \mathcal E(\omega_*+\eta)
 =\mathcal E(\omega_*)+
  \int_0^{s_*}\frac{\rho^{m+2}}\ell(\eta')^2\,\dd s,
\end{equation*}
because the cross term vanishes.  Hence $\omega_*$ is the unique
minimizer.  By \eqref{eq:un-ric1i}, the same equation makes every
axial--sphere Ricci component vanish.
\end{proof}

With the mixed Ricci components eliminated, it remains to choose the two
scale functions.  The next proposition chooses \(\ell\) so that the
axial Ricci terms have a uniform positive lower bound, and chooses
\(\rho\) to return to its initial value while having small logarithmic
derivatives of opposite signs at the two boundary components.  The
least-energy condition then determines \(\omega\) uniquely.

\begin{proposition}
\label{prop:holonomy-interpolation-cylinder}
Let \(\ell_0,\rho_0>0\).  Choose
\begin{equation*}
 0<\mathfrak p_0<\frac1{\rho_0},
 \qquad
 \Lambda_{\mathrm{con}}\ge
 \max\left\{2,\frac{40\rho_0}{\ell_0}\right\},
\end{equation*}
and set
\begin{equation}\label{eq:interpolation-cylinder-parameters}
 \mathfrak p_1=\frac{\mathfrak p_0}{\Lambda_{\mathrm{con}}^2},
 \qquad
 s_*=\frac{\Lambda_{\mathrm{con}}^2-1}{2\mathfrak p_0}.
\end{equation}
Define, for \(0\le s\le s_*\),
\begin{align}
 \ell(s)&=\ell_0\sqrt{1+2\mathfrak p_0s},
 \label{eq:un-ell-profile}\\
 \rho(s)&=\rho_0\exp\left(
 -\frac{\mathfrak p_1}{100s_*}s(s_*-s)\right),
 \label{eq:un-rho-profile}\\
\omega(s)&=\frac12\left(
  1-\frac{\displaystyle\int_0^s
               \frac{\ell(t)}{\rho(t)^{m+2}}\dd t}
             {\displaystyle\int_0^{s_*}
               \frac{\ell(t)}{\rho(t)^{m+2}}\dd t}
 \right).
 \label{eq:un-omega-profile}
\end{align}
Then the metric \eqref{eq:interpolation-cylinder-metric} has \(\Ric>0\), and its
boundary values are
\begin{equation*}
\begin{array}{c|c|c|c}
 &\ell&\rho&\omega\\ \hline
 s=0&\ell_0&\rho_0&1/2\\
 s=s_*&\Lambda_{\mathrm{con}}\ell_0&\rho_0&0.
\end{array}
\end{equation*}
With respect to the outward normals of $\mathcal C$ and the frame
\eqref{eq:un-adapted-frame}, the matrices of the boundary second
fundamental forms are
\begin{align}
 \II^{\mathrm{out}}_0&=
 \begin{pmatrix}
  -\mathfrak p_0&-\frac{\vartheta(0)}2\mathbf K_m^{\mathsf T}\\[2pt]
  -\frac{\vartheta(0)}2\mathbf K_m&
     \frac{\mathfrak p_1}{100}\operatorname{Id}_m
 \end{pmatrix},
 \label{eq:un-II-left}\\
 \II^{\mathrm{out}}_{s_*}&=
 \begin{pmatrix}
  \mathfrak p_1&\frac{\vartheta(s_*)}2\mathbf K_m^{\mathsf T}\\[2pt]
  \frac{\vartheta(s_*)}2\mathbf K_m&
     \frac{\mathfrak p_1}{100}\operatorname{Id}_m
 \end{pmatrix}.
 \label{eq:un-II-right}
\end{align}
In particular, the right boundary is strictly convex:
\begin{equation*}
 \II^{\mathrm{out}}_{s_*}>0.
\end{equation*}
\end{proposition}

\begin{proof}
We first record the contribution of the axial profile.  From
\eqref{eq:un-ell-profile},
\begin{equation*}
 \mathfrak p(s)=\frac{\mathfrak p_0}{1+2\mathfrak p_0s},
 \qquad
 \mathfrak p'=-2\mathfrak p^2,
 \qquad
 \mathfrak p_1\le\mathfrak p(s)\le\mathfrak p_0.
\end{equation*}
Indeed, \(1+2\mathfrak p_0s_*=\Lambda_{\mathrm{con}}^2\), so
\begin{equation*}
 \ell(s_*)=\Lambda_{\mathrm{con}}\ell_0,
 \qquad
 \mathfrak p(s_*)=
 \frac{\mathfrak p_0}{\Lambda_{\mathrm{con}}^2}=\mathfrak p_1.
\end{equation*}
In particular,
\begin{equation*}
 -\bigl(\mathfrak p'+\mathfrak p^2\bigr)=\mathfrak p^2
 \ge\mathfrak p_1^2,
\end{equation*}
which will dominate the error terms below.

We next consider the spherical radius.  Equation
\eqref{eq:un-rho-profile} gives
\begin{equation*}
 \mathfrak q(s)=\frac{\mathfrak p_1}{100}
 \left(\frac{2s}{s_*}-1\right).
\end{equation*}
Its integral over \([0,s_*]\) is zero.  Thus \(\rho(s_*)=\rho_0\),
while \(s(s_*-s)\le s_*^2/4\) and
\eqref{eq:interpolation-cylinder-parameters} give
\begin{equation}\label{eq:un-rho-bound}
 \rho(s_*)=\rho_0,
 \qquad
 e^{-(\Lambda_{\mathrm{con}}^2-1)/
       (800\Lambda_{\mathrm{con}}^2)}\rho_0
 \le\rho(s)\le\rho_0.
\end{equation}
Moreover,
\begin{equation*}
 |\mathfrak q|\le\frac{\mathfrak p_1}{100},
 \qquad
 \mathfrak q'=
 \frac{\Lambda_{\mathrm{con}}^2}
 {25(\Lambda_{\mathrm{con}}^2-1)}\mathfrak p_1^2
\le\frac{4}{75}\mathfrak p_1^2
 <\frac{3}{50}\mathfrak p_1^2.
\end{equation*}
When \(\mathfrak q\ge0\), necessarily
\(s\in[s_*/2,s_*]\), and direct substitution yields
\begin{equation*}
 \frac{\mathfrak p\mathfrak q}{\mathfrak p_1^2/100}
 =\frac{\Lambda_{\mathrm{con}}^2(2s/s_*-1)}
 {1+(\Lambda_{\mathrm{con}}^2-1)s/s_*}\le1.
\end{equation*}
When \(\mathfrak q\le0\), we instead use
\begin{equation*}
 \mathfrak p+m\mathfrak q
 \ge\left(1-\frac3{100}\right)\mathfrak p_1>0,
 \qquad
 \mathfrak q(\mathfrak p+m\mathfrak q)\le0.
\end{equation*}

After \(\ell\) and \(\rho\) have been fixed, the twist is determined by
\eqref{eq:un-omega-profile}.  Differentiating that formula gives
\begin{equation*}
 \frac{\rho^{m+2}\omega'}{\ell}
 =-\frac12\left(\int_0^{s_*}
       \frac{\ell(t)}{\rho(t)^{m+2}}\dd t\right)^{-1}.
\end{equation*}
The right-hand side is independent of \(s\), hence
\begin{equation}\label{eq:un-mixed-cancellation}
 \left(\frac{\rho^{m+2}\omega'}{\ell}\right)'=0.
\end{equation}
It follows from \eqref{eq:un-ric1i} that
\(\Ric_{1,i+1}=0\) for every \(1\le i\le m\).  The normalization in
\eqref{eq:un-omega-profile} also gives
\(\omega(0)=1/2\) and \(\omega(s_*)=0\).

Since \(\rho\le\rho_0\),
\begin{equation*}
 \int_0^{s_*}\frac{\ell(s)}{\rho(s)^{m+2}}\dd s
 \ge\frac1{\rho_0^{m+2}}\int_0^{s_*}\ell(s)\dd s
 =\frac{\ell_0(\Lambda_{\mathrm{con}}^3-1)}
 {3\mathfrak p_0\rho_0^{m+2}}.
\end{equation*}
By \eqref{eq:un-log-derivatives} and
\eqref{eq:un-omega-profile},
\begin{equation*}
 \vartheta(s)=-\frac1{2\rho(s)^{m+1}}
 \left(\int_0^{s_*}\frac{\ell(t)}{\rho(t)^{m+2}}\dd t\right)^{-1}.
\end{equation*}
Since \(m+1\le4\), the lower bound for \(\rho\) in
\eqref{eq:un-rho-bound} and the choice
\(\Lambda_{\mathrm{con}}\ge40\rho_0/\ell_0\) imply
\begin{align}
 \frac{|\vartheta|}{\mathfrak p_1}
 &\le\frac{3e^{1/200}}2
       \frac{\rho_0}{\ell_0}
       \frac{\Lambda_{\mathrm{con}}^2}
            {\Lambda_{\mathrm{con}}^3-1}\le\frac{3e^{1/200}}{80}
       \frac{\Lambda_{\mathrm{con}}^3}
            {\Lambda_{\mathrm{con}}^3-1}\le\frac{3e^{1/200}}{70}<\frac{1}{10}.
 \label{eq:un-twist-rate-bound}
\end{align}
Here we used \(\Lambda_{\mathrm{con}}\ge2\), for which
\(\Lambda_{\mathrm{con}}^3/
(\Lambda_{\mathrm{con}}^3-1)\le8/7\).

We now substitute these estimates into the formulas of
Lemma~\ref{lem:un-ricci}.  The remaining terms are measured relative to
the positive scale \(\mathfrak p_1^2\).  Substituting
\(\mathfrak p'=-2\mathfrak p^2\) into
\eqref{eq:un-ric11} gives
\begin{equation*}
 \Ric_{11}=\mathfrak p^2-m\mathfrak p\mathfrak q
            -\frac{\vartheta^2}{2}|\mathbf K_m|^2.
\end{equation*}
If \(\mathfrak q\le0\), its middle term is nonnegative.  If
\(\mathfrak q\ge0\), we use the estimate
\(\mathfrak p\mathfrak q\le\mathfrak p_1^2/100\) proved above.
Since \(m\le3\) and \(|K_m|\le1\),
\begin{align*}
 \Ric_{11}
 &\ge\left(
  1-\frac{3}{100}-\frac{1}{200}
 \right)\mathfrak p_1^2
 >\frac{19}{20}\mathfrak p_1^2.
\end{align*}
Here and below the displayed numerical constants are uniform for
\(m=2,3\).  Similarly, \eqref{eq:un-ric00} becomes
\begin{align*}
 \Ric_{00}
 &=\mathfrak p^2-m\mathfrak q'-m\mathfrak q^2
   -\frac{\vartheta^2\lvert\mathbf K_m\rvert^2}{2}\ge
 \left(
  1-\frac{9}{50}-\frac{3}{10^4}-\frac{1}{200}
 \right)\mathfrak p_1^2
 >
 \frac45\mathfrak p_1^2.
\end{align*}
For the restriction of \(\Ric\) to
\(\operatorname{span}\{E_2,\ldots,E_{m+1}\}\), first suppose that
\(\mathfrak q\le0\).  The final inequality in the spherical-profile
calculation shows that
\begin{equation*}
 \mathfrak q'+\mathfrak q(\mathfrak p+m\mathfrak q)
 \le\mathfrak q'.
\end{equation*}
If \(\mathfrak q\ge0\), then the preceding bounds on
\(\mathfrak p\mathfrak q\) and \(|\mathfrak q|\) give
\begin{equation*}
 \mathfrak q(\mathfrak p+m\mathfrak q)
 \le\left(\frac1{100}+\frac3{10^4}\right)\mathfrak p_1^2.
\end{equation*}
Consequently, in both cases,
\begin{equation*}
\mathfrak q'+\mathfrak q(\mathfrak p+m\mathfrak q)
 <
 \frac{3}{40}\mathfrak p_1^2
 \le
 \frac{3}{640}\mathfrak p_0^2
 <
 \frac{1}{200\rho_0^2}.
\end{equation*}
Here we used \(\Lambda_{\mathrm{con}}\ge2\) and
\(\mathfrak p_0<1/\rho_0\).  Since \(m-1\ge1\)
and \(\rho\le\rho_0\), the scalar multiple of the identity in
\eqref{eq:un-ricij} is therefore larger than
\(\frac{199}{200\rho_0^2}\), the remaining term
\((\vartheta^2/2)\mathbf K_m\otimes\mathbf K_m\) is nonnegative.

Lemma~\ref{lem:un-ricci} gives
\(\Ric(E_0,E_\alpha)=0\) for \(1\le\alpha\le m+1\), while
\eqref{eq:un-mixed-cancellation} gives
\(\Ric(E_1,E_{i+1})=0\) for \(1\le i\le m\).  Thus \(\Ric\) is block
diagonal with respect to
\begin{equation*}
 \mathbb RE_0\oplus\mathbb RE_1
 \oplus\operatorname{span}\{E_2,\ldots,E_{m+1}\},
\end{equation*}
and the preceding estimates prove \(\Ric_{g_{\mathcal C}}>0\).

Finally,
\(\mathfrak q(0)=-\mathfrak p_1/100\) and
\(\mathfrak q(s_*)=\mathfrak p_1/100\).  Apply
\eqref{eq:un-shape} with outward normal \(-\partial_s\) on the left and
\(+\partial_s\) on the right to obtain
\eqref{eq:un-II-left}--\eqref{eq:un-II-right}.  The spherical block in
\eqref{eq:un-II-right} is
\((\mathfrak p_1/100)\operatorname{Id}_m>0\).  Its Schur complement is
bounded below by
\begin{equation*}
  \mathfrak p_1-\frac{100}{\mathfrak p_1}\frac{\vartheta(s_*)^2}{4}
 >
 \left(1-\frac14\right)\mathfrak p_1
 =\frac34\mathfrak p_1>0,
\end{equation*}
where we used \eqref{eq:un-twist-rate-bound} and
\(\lvert\mathbf K_m\rvert\le1\).  Thus
\(\II^{\mathrm{out}}_{s_*}>0\).
\end{proof}

It remains to attach the left boundary of $\mathcal C$ to the twisted
outer boundary of \(W_m\).  Using \(\Phi_o\) on \(\partial_oW_m\) and
the standard product coordinates on
\(\{0\}\times\mathbb S^1\times\mathbb S^m\), both boundaries carry the
metric \(g_{\ell_0,\rho_0,1/2}\).
Lemma~\ref{lem:relative-gluing-unified} therefore reduces the gluing
problem to proving that the sum of their outward second fundamental
forms is positive definite.  At the left boundary, the outward normal of $\mathcal C$
contributes a negative axial term, a positive spherical term, and an
axial--sphere term controlled by \eqref{eq:un-twist-rate-bound}.  We use
the following elementary form of the Schur-complement criterion.  For a
symmetric quadratic form on
\begin{equation*}
 \mathbb RE_1\oplus
 \operatorname{span}\{E_2,\ldots,E_{m+1}\}
\end{equation*}
with positive lower bounds on its two diagonal restrictions, positivity
follows if the square of half the mixed coefficient is smaller than the
product of those two lower bounds.

\begin{lemma}
\label{lem:un-left-boundary-gluing}
Identify the outer boundary of a quotient \(W_m\), using \(\Phi_o\),
with the boundary
\(\{0\}\times\mathbb S^1\times\mathbb S^m\) of $\mathcal C$ in
Proposition~\ref{prop:holonomy-interpolation-cylinder}.  Suppose that the common induced
metric is \(g_{\ell_0,\rho_0,1/2}\).  With respect to the corresponding
axial--sphere splitting, assume that the outward second fundamental
form of \(W_m\) satisfies
\begin{equation*}
 \II_{W_m}(E_1,E_1)\ge\delta_{\mathrm{ax}}>0,\qquad
 \II_{W_m}|_{T\mathbb S^m}>0,\qquad
 \II_{W_m}(E_1,Y)=0
 \quad\text{for all }Y\in T\mathbb S^m.
\end{equation*}
If
\begin{equation*}
 0<\mathfrak p_0<
 \min\left\{\frac{\delta_{\mathrm{ax}}}{2},
             \frac1{\rho_0}\right\},
\end{equation*}
then the sum of the two outward second fundamental forms is positive
definite.
\end{lemma}

\begin{proof}
At \(s=0\), the outward normal of $\mathcal C$ is \(-\partial_s\).
For \(X=xE_1+Y\), where \(Y\) is tangent to the sphere factor,
\eqref{eq:un-II-left}, the hypotheses on \(W_m\), and
\(|\mathbf K_m|\le1\) give
\begin{equation}
 \bigl(\II_{W_m}+\II^{\mathrm{out}}_0\bigr)(X,X)
 \ge
 (\delta_{\mathrm{ax}}-\mathfrak p_0)x^2
 -|\vartheta(0)|\,|x|\,|Y|
 +\frac{\mathfrak p_1}{100}|Y|^2.
 \label{eq:un-left-quadratic-lower}
\end{equation}
The parameter assumptions imply
\begin{equation*}
 \delta_{\mathrm{ax}}>2\mathfrak p_0
 =2\Lambda_{\mathrm{con}}^2\mathfrak p_1
 \ge8\mathfrak p_1.
\end{equation*}
Consequently,
\begin{equation*}
 (\delta_{\mathrm{ax}}-\mathfrak p_0)
 \frac{\mathfrak p_1}{100}
 >\frac{1}{25}\mathfrak p_1^2.
\end{equation*}
On the other hand, \eqref{eq:un-twist-rate-bound} gives
\begin{equation*}
 \frac{\vartheta(0)^2}{4}
 <
 \frac{\mathfrak p_1^2}{400}.
\end{equation*}
Thus the square of half the mixed coefficient in
\eqref{eq:un-left-quadratic-lower} is smaller than the product of its
two diagonal coefficients.  The Schur-complement criterion proves that
the right-hand side of \eqref{eq:un-left-quadratic-lower}, and hence the
sum of the two second fundamental forms, is positive definite.
\end{proof}

\section{Construction of the complete Ricci-positive manifold}
\label{sec:complete-construction}

Recall that a building block in our construction is a compact cobordism, with product parametrizations of its two boundary components, whose incoming and outgoing product circles satisfy
$[\lambda_{\mathrm{in}}]=[\lambda_{\mathrm{out}}]^2$.  The purpose of
this section is to pass from the local geometric
constructions of the preceding sections to the complete noncompact
manifolds in Theorem~\ref{thm:main}.  We first combine the mapping torus
with a holonomy interpolation cylinder to obtain such a compact Ricci-positive block with untwisted
product boundaries.  The second fundamental form of its incoming
boundary is bounded below by an arbitrarily small negative multiple of
the induced metric, its outgoing boundary is strictly convex, and the
incoming product circle represents the square of the corresponding
outgoing generator.  We then attach the Ricci-positive initial cap and
attach the blocks successively along a ray.  Central subcylinders that
are disjoint from all smoothing collars ensure local finiteness and completeness of
the resulting metric, while the square relation at each stage
determines the fundamental group.

The boundary product parametrizations in the next proposition come from Lemma~\ref{lem:un-boundary-components} and
Proposition~\ref{prop:holonomy-interpolation-cylinder}.

\begin{proposition}
\label{prop:ricci-positive-building-block}
Fix \(m\in\{2,3\}\).  Given
\begin{equation*}
 \ell_I>0,\qquad \rho_I>0,\qquad \delta>0,
 \qquad s_{\min}>0,
\end{equation*}
there are a compact connected \((m+2)\)-manifold \(B_m\), a metric
\(g_B\) with \(\Ric_{g_B}>0\), and boundary product parametrizations
\begin{equation*}
 \Phi_I:\mathbb S^1\times\mathbb S^m\longrightarrow\partial_iB_m,
 \qquad
 \Phi_O:\mathbb S^1\times\mathbb S^m\longrightarrow\partial_oB_m
\end{equation*}
with the following properties.

\begin{enumerate}[label=\textup{(\roman*)}]
\item The boundary has exactly the two components
\begin{equation*}
 \partial B_m=\partial_iB_m\sqcup\partial_oB_m,
\end{equation*}
and
\begin{equation*}
 \Phi_I^*(g_B|_{\partial_iB_m})=g_{\ell_I,\rho_I,0},
 \qquad
 \Phi_O^*(g_B|_{\partial_oB_m})=g_{\ell_O,\rho_O,0}
\end{equation*}
for some \(\ell_O,\rho_O>0\).

\item With respect to the outward normals,
\begin{equation*}
 \II_{\partial_iB_m}>-\delta\,g_B|_{\partial_iB_m},
 \qquad
 \II_{\partial_oB_m}>0.
\end{equation*}

\item
\begin{equation*}
 \pi_1(B_m)\cong\mathbb Z=\langle \gamma\rangle.
\end{equation*}
For any \(y\in\mathbb S^m\), the outgoing product circle
\(\theta\mapsto\Phi_O(\theta,y)\) represents \(\gamma\), whereas the incoming
product circle \(\theta\mapsto\Phi_I(\theta,y)\) represents \(\gamma^2\).
After the loops are transported to one basepoint, these elements are
independent of the transporting paths because \(\pi_1(B_m)\) is
abelian.

\item The block contains a holonomy interpolation cylinder
\begin{equation*}
 \mathcal C=[0,s_*]\times\mathbb S^1\times\mathbb S^m,
 \qquad s_*\ge s_{\min},
\end{equation*}
whose metric is exactly the metric of
Proposition~\ref{prop:holonomy-interpolation-cylinder} on
\begin{equation*}
 [s_*/4,s_*]\times\mathbb S^1\times\mathbb S^m.
\end{equation*}
In particular, the central subcylinder
\begin{equation}\label{def: central-subcylinder}
 [s_*/4,3s_*/4]\times\mathbb S^1\times\mathbb S^m
\end{equation}
is disjoint from the support of every gluing modification.
\end{enumerate}
\end{proposition}

\begin{proof}
Apply Theorem~\ref{thm:ricci-positive-equivariant-cobordism} with prescribed inner
radius \(R_*=\rho_I\) and with
\begin{equation*}
 0<\varepsilon_*<\delta.
\end{equation*}
Choose \(T,\rho_*>0\) so that
\begin{equation*}
 \frac{T\rho_*}{\pi}=\ell_I,
\end{equation*}
and apply Lemma~\ref{lem:un-warp} with this \(\rho_*\) and a number
\(0<\delta_*<\delta\), obtaining
\(f=\rho_*(1+\alpha\psi+\alpha^2\varphi)\).
Form the quotient \(W_m\).  Lemma~\ref{lem:un-boundary-components}
gives
\begin{equation*}
 \Phi_i^*(G|_{\partial_iW_m})=g_{\ell_I,\rho_I,0},
 \qquad
 \Phi_o^*(G|_{\partial_oW_m})
 =g_{\ell_{\mathrm{tw}},\rho_{\mathrm{tw}},1/2},
\end{equation*}
where
\begin{equation*}
 \ell_{\mathrm{tw}}=\frac{T\rho_*(1+\alpha)}{2\pi},
 \qquad \rho_{\mathrm{tw}}=R_o.
\end{equation*}
On \(\partial_iW_m\), the sphere restriction of the outward second
fundamental form is greater than \(-\varepsilon_*G\), its axial entry is
greater than \(-\delta_*\), and its axial--sphere block is zero.  Hence
\begin{equation*}
 \II_{\partial_iW_m}>-\delta\,G|_{\partial_iW_m}.
\end{equation*}
Indeed, the form is block diagonal in the axial--sphere splitting, and
each diagonal block has all eigenvalues greater than \(-\delta\).
On \(\partial_oW_m\), the sphere restriction is positive definite, the
axial--sphere block is zero, and the axial principal curvature is
positive.  Compactness therefore gives \(\delta_{\mathrm{ax}}>0\) such
that the latter is at least \(\delta_{\mathrm{ax}}\).

Choose
\begin{equation*}
 0<\mathfrak p_0<
 \min\left\{\frac{\delta_{\mathrm{ax}}}{2},
                  \frac1{\rho_{\mathrm{tw}}}\right\}.
\end{equation*}
Then choose \(\Lambda_{\mathrm{con}}\) so large that
\begin{equation*}
 \Lambda_{\mathrm{con}}\ge
 \max\left\{2,\frac{40\rho_{\mathrm{tw}}}{\ell_{\mathrm{tw}}}\right\},
 \qquad
 s_*=\frac{\Lambda_{\mathrm{con}}^2-1}{2\mathfrak p_0}
 \ge s_{\min}.
\end{equation*}
Proposition~\ref{prop:holonomy-interpolation-cylinder} gives a Ricci-positive holonomy interpolation cylinder
whose left metric is
\(g_{\ell_{\mathrm{tw}},\rho_{\mathrm{tw}},1/2}\) and whose right metric is
\begin{equation*}
 g_{\Lambda_{\mathrm{con}}\ell_{\mathrm{tw}},\rho_{\mathrm{tw}},0}.
\end{equation*}
Lemma~\ref{lem:un-left-boundary-gluing} verifies the strict second-fundamental-form inequality at \(\partial_oW_m\).  Apply
Lemma~\ref{lem:relative-gluing-unified}, with its support in the collar
$[0,s_*/4)\times\mathbb S^1\times\mathbb S^m$ of $\mathcal C$.  The resulting manifold is
\begin{equation*}
 B_m=W_m\cup_{\Phi_o\circ\iota_0^{-1}}\mathcal C,
 \qquad
 \iota_0(\theta,y)=(0,\theta,y),
\end{equation*}
with incoming boundary \(\partial_iW_m\), outgoing boundary
\(\{s_*\}\times\mathbb S^1\times\mathbb S^m\), and
\begin{equation*}
 \ell_O=\Lambda_{\mathrm{con}}\ell_{\mathrm{tw}},
 \qquad \rho_O=\rho_{\mathrm{tw}}.
\end{equation*}
Define the boundary product parametrizations in the statement by
\begin{equation*}
 \Phi_I=\Phi_i,\qquad
 \Phi_O=\iota_{s_*},\qquad
 \iota_{s_*}(\theta,y)=(s_*,\theta,y).
\end{equation*}
The final assertion of Proposition~\ref{prop:holonomy-interpolation-cylinder} gives
\(\II_{\partial_oB_m}>0\).

Collapsing the holonomy interpolation cylinder in the \(s\)-direction gives a deformation
retraction \(B_m\to W_m\).  Lemma~\ref{lem:un-boundary-components}
therefore gives
\begin{equation*}
 \pi_1(B_m)\cong\pi_1(W_m)\cong\mathbb Z,
\end{equation*}
and identifies the incoming circle with \(\gamma^2\).  Under the same
retraction the outgoing circle becomes the outer circle of \(W_m\), so
it represents \(\gamma\).  If two paths are used to move a boundary loop to
a fixed basepoint, the resulting elements are conjugate and hence equal
because the group is abelian.  The support choice above proves the
final assertion.
\end{proof}

Proposition~\ref{prop:ricci-positive-building-block} provides the
building block used at each stage of the successive construction. To
initiate the construction, we also need a compact Ricci-positive cap
diffeomorphic to
\begin{equation*}
 D^2\times\mathbb S^m.
\end{equation*}
Its boundary
\begin{equation*}
\partial\bigl(D^2\times\mathbb S^m\bigr)
\cong \mathbb S^1\times\mathbb S^m
\end{equation*}
should carry the same untwisted product geometry and be strictly
convex. Moreover, for each \(y\in\mathbb S^m\), the circle
\(\mathbb S^1\times\{y\}\) should bound the disk
\( D^2\times\{y\}\) in the cap. The following elementary doubly
warped construction provides precisely such a metric.

\begin{lemma}
\label{lem:un-cap}
For \(m\in\{2,3\}\) and arbitrary \(\ell,\rho>0\), the manifold
\(D^2\times\mathbb S^m\) carries a Ricci-positive metric whose boundary is the
product \(\mathbb S^1_\ell\times\mathbb S^m_\rho\) and is strictly convex.
\end{lemma}

\begin{proof}
Set
\begin{equation*}
 s_{\mathrm{cap}}=\frac{\pi\ell}{3},
 \qquad
 \ell_{\mathrm{cap}}(s)=2\ell\sin\left(\frac{s}{2\ell}\right),
\end{equation*}
and, for \(\gamma_{\mathrm{cap}}>0\), set
\begin{equation*}
 \rho_{\mathrm{cap}}(s)
 =\rho\exp\left(\gamma_{\mathrm{cap}}
       (s^2-s_{\mathrm{cap}}^2)\right).
\end{equation*}
On
\([0,s_{\mathrm{cap}}]\times\mathbb S^1\times\mathbb S^m\), consider
\begin{equation}\label{eq:un-cap-metric}
 g_{\mathrm{cap}}(\gamma_{\mathrm{cap}})
 =\dd s^2+\ell_{\mathrm{cap}}(s)^2\dd\theta^2
       +\rho_{\mathrm{cap}}(s)^2g_{\mathbb S^m}.
\end{equation}
At \(s=0\),
\begin{align*}
 \ell_{\mathrm{cap}}(s)
 &=s-\frac{s^3}{24\ell^2}+O(s^5),\\
 \rho_{\mathrm{cap}}(s)
 &=\rho_{\mathrm{cap}}(0)
   \left(1+\gamma_{\mathrm{cap}}s^2+O(s^4)\right).
\end{align*}
Thus the circle-radius function has the required odd expansion with
derivative one, and the sphere-radius function has an even expansion.
Consequently, \eqref{eq:un-cap-metric} extends smoothly across the
center of the \(D^2\)-factor.

For \(s>0\), let
\begin{equation*}
E_\theta=\ell_{\mathrm{cap}}^{-1}\partial_\theta,
\qquad
E=\rho_{\mathrm{cap}}^{-1}V,
\end{equation*}
where \(V\in T\mathbb S^m\) is a unit vector with respect to the round
metric \(g_{\mathbb S^m}\). Thus
\(\partial_s\), \(E_\theta\), and \(E\) are unit vectors for
\(g_{\mathrm{cap}}\) in the radial, circle, and sphere directions,
respectively. The standard Ricci formulas for a doubly warped product
give
\begin{align*}
\operatorname{Ric}_{g_{\mathrm{cap}}}(\partial_s,\partial_s)
&=
-\frac{\ell_{\mathrm{cap}}''}{\ell_{\mathrm{cap}}}
-m\frac{\rho_{\mathrm{cap}}''}{\rho_{\mathrm{cap}}},\\
\operatorname{Ric}_{g_{\mathrm{cap}}}(E_\theta,E_\theta)
&=
-\frac{\ell_{\mathrm{cap}}''}{\ell_{\mathrm{cap}}}
-m\frac{\ell_{\mathrm{cap}}'\rho_{\mathrm{cap}}'}
        {\ell_{\mathrm{cap}}\rho_{\mathrm{cap}}},\\
\operatorname{Ric}_{g_{\mathrm{cap}}}(E,E)
&=
\frac{m-1}{\rho_{\mathrm{cap}}^2}
-\frac{\rho_{\mathrm{cap}}''}{\rho_{\mathrm{cap}}}
-(m-1)\left(
\frac{\rho_{\mathrm{cap}}'}{\rho_{\mathrm{cap}}}
\right)^2
-\frac{\ell_{\mathrm{cap}}'\rho_{\mathrm{cap}}'}
       {\ell_{\mathrm{cap}}\rho_{\mathrm{cap}}}.
\end{align*}
All mixed Ricci components with respect to the orthogonal splitting
\begin{equation*}
\mathbb R\partial_s
\oplus\mathbb R E_\theta
\oplus T\mathbb S^m
\end{equation*}
vanish. These formulas extend smoothly across the core \(s=0\).
When \(\gamma_{\mathrm{cap}}=0\), these three quantities are
\begin{equation*}
 \frac1{4\ell^2},\qquad
 \frac1{4\ell^2},\qquad
 \frac{m-1}{\rho^2}.
\end{equation*}
They are strictly positive because \(m\ge2\).  As
\(\gamma_{\mathrm{cap}}\to0\), the metrics
\eqref{eq:un-cap-metric} converge smoothly on the compact manifold
\(D^2\times\mathbb S^m\) to the product metric obtained at
\(\gamma_{\mathrm{cap}}=0\).  Hence
\(\Ric_{g_{\mathrm{cap}}(\gamma_{\mathrm{cap}})}>0\) for every
sufficiently small \(\gamma_{\mathrm{cap}}>0\).

At \(s=s_{\mathrm{cap}}\),
\begin{equation*}
 \ell_{\mathrm{cap}}(s_{\mathrm{cap}})=\ell,
 \qquad
 \rho_{\mathrm{cap}}(s_{\mathrm{cap}})=\rho.
\end{equation*}
Thus the induced boundary metric is
\(\ell^2\dd\theta^2+\rho^2g_{\mathbb S^m}\).  With respect to the
outward unit normal \(+\partial_s\), the boundary principal curvatures
in the circle and sphere directions are
\begin{equation*}
 \frac{\ell_{\mathrm{cap}}'(s_{\mathrm{cap}})}
      {\ell_{\mathrm{cap}}(s_{\mathrm{cap}})}
 =\frac{\sqrt3}{2\ell}>0,
 \qquad
 \frac{\rho_{\mathrm{cap}}'(s_{\mathrm{cap}})}
      {\rho_{\mathrm{cap}}(s_{\mathrm{cap}})}
 =2\gamma_{\mathrm{cap}}s_{\mathrm{cap}}>0,
\end{equation*}
where the second value has multiplicity \(m\).  Therefore the boundary
is strictly convex.
\end{proof}

We now carry out the construction.  Fix $\gamma_{\mathrm{cap}}>0$ as in the proof of Lemma~\ref{lem:un-cap} and write $g_{\mathrm{cap}}=g_{\mathrm{cap}}(\gamma_{\mathrm{cap}})$.  The
Ricci-positive manifold
$(D^2\times\mathbb S^m,g_{\mathrm{cap}})$ is the initial cap introduced
in the introduction.  For every $y\in\mathbb S^m$, its boundary product
circle bounds the disk $D^2\times\{y\}$.

We then attach the building blocks inductively, beginning with the initial cap.  At each stage, we choose the parameter in the
estimate \(\II_{\partial_iB_m}>-\delta g\) to be smaller than half the
least outward principal curvature of the preceding block.  The relative
gluing lemma therefore applies at every interface.
Meanwhile, the lengths of the holonomy interpolation cylinders are chosen to tend rapidly to
infinity.  In each block, the central subcylinder defined in (\ref{def: central-subcylinder}) between the levels
$s_j/4$ and $3s_j/4$ is disjoint from every smoothing collar.

Throughout, \(\bar\ell_j\) and \(\bar\rho_j\) denote the circle and
sphere radius of the incoming boundary of the \(j\)-th block, the bars
distinguish them from the radius profiles \(\ell_j(s)\) and
\(\rho_j(s)\) of its holonomy interpolation cylinder used in
Section~\ref{sec:further-properties}.  Fix \(\bar\ell_0,\bar\rho_0>0\), and take the initial cap from
Lemma~\ref{lem:un-cap} with boundary metric \(g_{\bar\ell_0,\bar\rho_0,0}\).
Let \(\lambda_{\mathrm{cap}}>0\) be its least outward boundary principal
curvature.  Apply Proposition~\ref{prop:ricci-positive-building-block} with this
incoming metric and with
\begin{equation*}
 0<\delta_0<\frac{\lambda_{\mathrm{cap}}}{2},
 \qquad s_{\min}=4,
\end{equation*}
and denote the resulting block by \(B_0\).  Under the boundary product parametrizations, the sum of the two outward second fundamental forms is
bounded below by
\((\lambda_{\mathrm{cap}}-\delta_0)g>0\).
Lemma~\ref{lem:relative-gluing-unified} therefore joins the initial cap to
\(B_0\).

Inductively, suppose the outgoing boundary of \(B_j\) has product metric
\begin{equation*}
 g_{\bar\ell_{j+1},\bar\rho_{j+1},0}
\end{equation*}
and least outward principal curvature \(\lambda_j>0\).  Construct
\(B_{j+1}\) with this incoming metric and with
\begin{equation*}
 0<\delta_{j+1}<\frac{\lambda_j}{2},
 \qquad s_{\min}=2^{j+3}.
\end{equation*}
The identity in the product coordinates is an isometry of the boundary
metrics, and the sum of the outward second fundamental forms is bounded
below by \((\lambda_j-\delta_{j+1})g>0\).  The relative gluing lemma
therefore applies.  Note that no uniform positive lower bound for
\(\lambda_j\) is needed.

Let
\begin{equation*} M=(D^2\times\mathbb S^m)\cup B_0\cup B_1\cup B_2\cup\cdots
\end{equation*}
denote the manifold obtained after all these identifications.  For each
\(j\ge0\), write
\begin{equation*}
 \mathcal C_j=[0,s_j]\times\mathbb S^1\times\mathbb S^m
 \subset B_j
\end{equation*}
for the holonomy interpolation cylinder supplied by
Proposition~\ref{prop:ricci-positive-building-block}.  The preceding choices give
\begin{equation*}
 s_j\ge2^{j+2}.
\end{equation*}
The internal smoothing in \(B_j\) is supported to the left of
\(\{s_j/4\}\times\mathbb S^1\times\mathbb S^m\).  At the outgoing
boundary of \(B_j\), choose the support of the next gluing inside the
collar
\begin{equation*}
 (3s_j/4,s_j]\times\mathbb S^1\times\mathbb S^m.
\end{equation*}
Consequently, we may denote the central subcylinder of $B_j$ by
\begin{equation}\label{eq:un-central-slab}
 \mathcal P_j
 =[s_j/4,3s_j/4]\times\mathbb S^1\times\mathbb S^m
 \subset\mathcal C_j.
\end{equation}
The final metric on $\mathcal P_j$ is exactly the metric on the cylinder
$\mathcal C_j$ from Proposition~\ref{prop:holonomy-interpolation-cylinder}.

Set
\begin{equation*}
 \Sigma_j=\{s_j/2\}\times\mathbb S^1\times\mathbb S^m
 \subset\mathcal C_j,
\end{equation*}
and denote the two closed subcylinders cut out by $\Sigma_j$ by
\begin{equation*}
 \mathcal C_j^-
 =[0,s_j/2]\times\mathbb S^1\times\mathbb S^m,
 \qquad
\mathcal C_j^+
 =[s_j/2,s_j]\times\mathbb S^1\times\mathbb S^m.
\end{equation*}
Let \(M_N\) be the closure of the component of
\(M\setminus\Sigma_N\) containing the initial cap.  The dual graph of the
initial cap and the blocks is a ray, so \(\Sigma_N\) separates a finite initial
segment from the connected infinite tail.  Hence
\begin{equation}\label{eq:compact-exhaustion}
 \partial M_N=\Sigma_N,\qquad
 M_N\subset\operatorname{int}M_{N+1},\qquad
 \bigcup_{N\ge0}\operatorname{int}M_N=M,
\end{equation}
and \(M\setminus M_N\) is connected.  Since \(M_N\) is contained in the
union of the initial cap and finitely many compact blocks, it is compact.

The open sets \(\operatorname{int}M_N\) form a nested cover of \(M\).
Every compact subset is therefore contained in one of them and meets
only finitely many gluing collars.  Thus the smoothing operations are
locally finite and define a smooth metric on \(M\) with
\(\Ric_M>0\).

This completes the geometric assembly.  We next determine the topology
of \(M\).  The complement \(M\setminus M_N\) is the connected tail
beginning at the midpoint of \(\mathcal C_N\), so the exhaustion
\eqref{eq:compact-exhaustion} by compact sets with connected complements
shows that \(M\) has one end.  Applying the Seifert--van Kampen theorem
to the same exhaustion records the square relation introduced by each
successive block.

For \(j<N\), let \(a_j\) be the generator represented by the outgoing
product circle of \(B_j\).  Let \(a_N\) be represented by a product
circle in \(\Sigma_N\).  The product structure of the cylinder
\(\mathcal C_N\) identifies this
class with the outgoing generator of \(B_N\).
For \(0\le j<N\), Proposition~\ref{prop:ricci-positive-building-block} says that
the incoming product circle of \(B_{j+1}\) represents \(a_{j+1}^2\).
Therefore the boundary identification between \(B_j\) and \(B_{j+1}\)
gives
\begin{equation*}
 a_j=a_{j+1}^2.
\end{equation*}
The initial cap is simply connected, and its boundary product circle bounds
the \(D^2\)-factor.  Since this circle is glued to the incoming circle
of \(B_0\), it gives \(a_0^2=1\).  The subcylinder
$\mathcal C_N^-$ deformation retracts to its left boundary.  Repeated
 applications of the Seifert--van Kampen theorem therefore give
\begin{align}
 \pi_1(M_N)
 &=\left\langle a_0,\ldots,a_N\ \middle|\
   a_0^2=1,\ a_j=a_{j+1}^2\ (0\le j<N)\right\rangle\notag\\
 &\cong\langle a_N\mid a_N^{2^{N+1}}=1\rangle
 \cong\mathbb Z/2^{N+1}\mathbb Z.
 \label{eq:un-finite-pi1}
\end{align}
Under these identifications, inclusion sends
\begin{equation*}
 \pi_1(M_N)\longrightarrow\pi_1(M_{N+1}),
 \qquad [k]\longmapsto[2k].
\end{equation*}
This homomorphism is injective.  Every loop and every homotopy of loops
has compact image and is contained in some \(M_N\).  Hence
\begin{equation}\label{eq:un-prufer}
 \pi_1(M)
 =\varinjlim_N\bigl(\mathbb Z/2^{N+1}\mathbb Z,[k]\mapsto[2k]\bigr)
 =\CPrufer.
\end{equation}

It remains to verify that the resulting Ricci-positive metric is
complete.  This is where the quantitative lower bounds on the lengths of
the holonomy interpolation cylinders enter: every curve traversing
infinitely many stages must cross $\mathcal P_j$ for every sufficiently large $j$.
The sum of the corresponding lower bounds diverges.  We conclude by
recording orientability and completing the proof of the main theorem.

Let \(\zeta\) be a piecewise smooth curve joining
\(D^2\times\mathbb S^m\) to \(M\setminus M_N\).  It crosses
$\mathcal P_j$ for every $0\le j<N$.  On each such slab the metric
contains the orthogonal summand \(\dd s^2\).  Parametrizing a crossing
by \(t\), its \(s\)-coordinate traverses
\([s_j/4,3s_j/4]\), so the portion of \(\zeta\) in that slab has length
at least
\begin{equation*}
 \int|s'(t)|\,\dd t\ge\frac{s_j}{2}.
\end{equation*}
The slabs $\mathcal P_j$ are pairwise disjoint.  Summing these contributions
and then taking the infimum over \(\zeta\) gives
\begin{equation*}
 d(D^2\times\mathbb S^m,M\setminus M_N)
 \ge\frac12\sum_{j=0}^{N-1}s_j
 \ge\frac12\sum_{j=0}^{N-1}2^{j+2}\longrightarrow\infty.
\end{equation*}
Fix \(x\in M\) and \(R>0\), and choose
\(p\in D^2\times\mathbb S^m\).  For all sufficiently large \(N\),
\begin{equation*}
 d(D^2\times\mathbb S^m,M\setminus M_N)>R+d(p,x).
\end{equation*}
For \(y\in M\setminus M_N\), the triangle inequality gives
\begin{equation*}
 d(x,y)\ge d(p,y)-d(p,x)>R.
\end{equation*}
Consequently, \(\overline{B}_R(x)\subset M_N\).  Since \(M_N\) is
compact, every closed bounded metric ball is compact.  Thus the length
metric is proper, and the Hopf--Rinow theorem implies geodesic
completeness.

Finally, \(P_m\) is connected and orientable.  The orientation sign of
the diffeomorphism \(\sigma_m\) is therefore constant, so it is enough
to compute the sign of its differential at the fixed point \(y\) of the involution in the proof of Theorem~\ref{thm:ricci-positive-equivariant-cobordism}.  At that point,
\begin{equation*}
 d\sigma_2=\operatorname{diag}(-1,1,-1),
 \qquad d\sigma_3=-\operatorname{Id}_{\mathbb R^4},
\end{equation*}
and both determinants are \(+1\).  Thus each mapping torus, holonomy
interpolation cylinder, and initial cap is orientable.  Since the dual gluing graph is
a ray, their
orientations can be chosen successively so that every boundary
identification reverses boundary orientations.  Hence \(M\) is
orientable.

\begin{proof}[Proof of Theorem~\ref{thm:main}]
For \(m\in\{2,3\}\), the preceding locally finite construction gives a
smooth one-ended \((m+2)\)-manifold with \(\Ric>0\).  Equation
\eqref{eq:un-prufer} identifies its fundamental group with \(\CPrufer\),
the estimates on the slabs $\mathcal P_j$ prove completeness, and orientability was shown above.  The cases \(m=2\)
and \(m=3\) have dimensions four and five, respectively.
\end{proof}

We conclude this section with the proof of
Theorem~\ref{thm:all-subgroups-QZ}, which replaces the involution by
the cyclic actions of Appendix~\ref{app:Zp-equivariant}.

\begin{proof}[Proof of Theorem~\ref{thm:all-subgroups-QZ}]
We first suppose that $\Gamma$ is infinite. Since every finitely generated subgroup of $\mathbb Q/\mathbb Z$ is cyclic, one can choose a strictly increasing exhaustion \begin{equation*} 0=\Gamma_{-1}\subsetneq\Gamma_0\subsetneq\Gamma_1\subsetneq\Gamma_2\subsetneq\cdots, \qquad \Gamma=\bigcup_{j\geq0}\Gamma_j, \end{equation*} where $\Gamma_j$ is cyclic of order $n_j$ and $n_{j-1}\mid n_j$. Set $n_{-1}=1$ and \begin{equation*} k_j=\frac{n_j}{n_{j-1}}\geq2. \end{equation*} At the $j$-th stage we use the $\mathbb Z_{k_j}$-equivariant cobordism of Theorem~\ref{thm:Zp-equivariant-cobordism}, followed by the circle warping, mapping torus, and holonomy interpolation cylinder with the substitutions recorded in Remark~\ref{rem:Zp-downstream}. If $a_j$ denotes the circle generator introduced at that stage, the initial cap and the mapping torus give the relations \begin{equation*} a_{-1}=1, \qquad a_{j-1}=a_j^{k_j}. \end{equation*}
Consequently, the fundamental group of the $N$-th finite stage is \begin{equation*} \pi_1(M_N) \cong \mathbb Z/(k_0k_1\cdots k_N)\mathbb Z = \mathbb Z/n_N\mathbb Z, \end{equation*} and the inclusion $M_N\hookrightarrow M_{N+1}$ induces \begin{equation*} [q]_{n_N}\longmapsto[k_{N+1}q]_{n_{N+1}}. \end{equation*} Identifying $[1]_{n_N}$ with $1/n_N+\mathbb Z$, we obtain \begin{equation*} \pi_1(M_\Gamma) \cong \varinjlim_N\mathbb Z/n_N\mathbb Z \cong \bigcup_{N\geq0} \frac{1}{n_N}\mathbb Z\bigg/\mathbb Z = \Gamma. \end{equation*}

All geometric choices are made after the finite integer $k_j$ has been fixed. Thus, the constants at the $j$-th stage may depend on $k_j$, no estimate uniform in $j$ or in $k_j$ is needed. The circle warping, mapping torus, holonomy interpolation, and boundary smoothing arguments apply stage by stage. The interpolation cylinder inserted at each stage may moreover be chosen arbitrarily long. Choosing their lengths so that every divergent curve has infinite total length gives completeness exactly as in the dyadic construction. The arguments establishing orientability and the existence of a single end are purely topological and remain unchanged. This proves the theorem when $\Gamma$ is infinite.

If $\Gamma$ is finite, then $\Gamma\cong\mathbb Z/n\mathbb Z$ for some $n\geq1$. In this case one may take \begin{equation*} M_\Gamma=L(n,1)\times\Sigma_{\mathrm{cig}}, \end{equation*} where $L(1,1)=\mathbb S^3$ and $\Sigma_{\mathrm{cig}}$ is Hamilton's cigar soliton \cite{Hamilton1988Surfaces}, a complete surface of positive Gaussian curvature.  With the round quotient metric on the lens space, the product metric is complete and has positive Ricci curvature. The resulting manifold is connected, orientable, one-ended, and has fundamental group $\mathbb Z/n\mathbb Z$. 
\end{proof}

\section{Further properties of the constructed manifolds}
\label{sec:further-properties}

Throughout this section we fix \(m\in\{2,3\}\), set \(n=m+2\), and let
\((M^n,g)\) be one of the complete Ricci-positive manifolds constructed
in Section~\ref{sec:complete-construction}.  We denote
\(q\colon\widetilde M\to M\) for its universal covering.  For a base
point \(p\in M\), we deinfe the \emph{lower} and \emph{upper logarithmic
volume-growth exponents} of \(M\) to be
\begin{equation*}
 \liminf_{r\to\infty}\frac{\log\operatorname{Vol}B_r(p)}{\log r}
 \qquad\text{and}\qquad
 \limsup_{r\to\infty}\frac{\log\operatorname{Vol}B_r(p)}{\log r},
\end{equation*}
respectively, they do not depend on \(p\).  We also define the \emph{asymptotic volume
ratio} 
\begin{equation*}
 \operatorname{AVR}(M)=\lim_{r\to\infty}\frac{\operatorname{Vol}B_r(p)}{r^n},
\end{equation*}
which exists by the Bishop--Gromov comparison inequality. 
The same definitions apply to \(\widetilde M\). 
The purpose of this section is to prove the following theorem.

\begin{theorem}\label{thm:further-properties}
In the setting above, the following hold.
\begin{enumerate}[label=\textup{(\roman*)}]
\item \(\widetilde M\) is diffeomorphic to
\(\mathbb R^{m+2}\#_{\infty}(\mathbb S^2\times\mathbb S^m)\).  In
particular, \(\widetilde M\) is one-ended and simply connected at
infinity.
\item \(\operatorname{AVR}(M)=0\) for every choice of the parameters.
For \(m=2\), \(\operatorname{AVR}(\widetilde M)=0\) for every
choice of the parameters.
\item The parameters of the holonomy interpolation cylinders can be
chosen so that, on both \(M\) and \(\widetilde M\), the lower
logarithmic volume-growth exponent equals \(1\) and the upper one
equals \(n-1\).  For such a choice, \(\operatorname{AVR}(\widetilde M)=0\)
and
\begin{equation*}
 \limsup_{r\to\infty}\frac{\operatorname{Vol}_M B_r(p)}{r}
 =\limsup_{r\to\infty}\frac{\operatorname{Vol}_{\widetilde M}
   B_r(\widetilde p)}{r}=\infty.
\end{equation*}
\end{enumerate}
\end{theorem}
Part~\textup{(i)} is proved in Theorem~\ref{thm:universal-cover-topology}, and the
exponents in \textup{(iii)} are computed in
Proposition~\ref{prop:volume-growth-flexibility}.  The remaining
assertions are deduced at the end of the section.

\subsection{Topology of the universal cover}
\label{subs:topology-of-the-universal-cover}

Fix a point \(p\) in the initial cap and a point
\(\widetilde p\in q^{-1}(p)\).  For \(N\ge0\), let \(\widehat M_N\) be
the component of \(q^{-1}(M_N)\) containing \(\widetilde p\), where the
compact stages \(M_N\) are defined in \eqref{eq:compact-exhaustion}.  We
also write \(\widehat M_{-1}\) for the component of the preimage of the
initial cap containing \(\widetilde p\), thus
\(\widehat M_{-1}\cong D^2\times\mathbb S^m\).

We first make explicit the boundary parametrizations used below.  Fix
the round identifications \(\iota_o\colon\mathbb S^m\to\partial_oP_m\)
and \(\iota_+\colon\mathbb S^m\to I_+\) of
Lemma~\ref{lem:un-boundary-components}, and set
\(\iota_-=\sigma_m\circ\iota_+\).  By
Theorem~\ref{thm:ricci-positive-equivariant-cobordism},
\begin{equation}\label{eq:universal-boundary-parametrizations}
 \sigma_m\circ\iota_o
 =\iota_o\circ\exp(\pi K_m),
 \qquad
 \sigma_m\circ\iota_+=\iota_-.
\end{equation}
Let \(T_j\) be the period of the mapping torus in the block \(B_j\).
Its outer and inner boundary product parametrizations are then
\begin{equation}\label{eq:universal-mapping-torus-boundary-maps}
 \begin{aligned}
  \Phi_{o,j}(e^{\mathrm i\theta},y)
  &=\left[\frac{T_j\theta}{2\pi},
    \iota_o\!\left(\exp\left(\frac\theta2K_m\right)y\right)\right],\\
  \Phi_{i,j}(e^{\mathrm i\theta},y)
  &=\left[\frac{T_j\theta}{\pi},\iota_+(y)\right].
 \end{aligned}
\end{equation}
The first identity in \eqref{eq:universal-boundary-parametrizations}
makes \(\Phi_{o,j}\) well defined, and the second shows that the first
return to \(I_+\) occurs after translation by \(2T_j\).  Since the
holonomy interpolation cylinder of \(B_j\) carries the product
coordinates \((s,\theta,y)\), it transports the circle and sphere
coordinates of \eqref{eq:universal-mapping-torus-boundary-maps} to
every intermediate slice, in particular to \(\Sigma_j\) and to the
outgoing boundary of \(B_j\).

The following surgery lemma describes what happens when two manifolds of the form
\((D^2\times\mathbb S^m)\#(\text{copies of }\mathbb S^2\times\mathbb S^m)\)
are glued to the two incoming boundary components of
\(\mathbb S^1\times P_m\).

\begin{lemma}\label{lem:universal-two-core}
Let
\begin{equation*}
 Y_\pm\cong
 (D^2\times\mathbb S^m)
 \mathbin{\#}^{\,r_\pm}(\mathbb S^2\times\mathbb S^m),
 \qquad r_\pm\ge0,
\end{equation*}
where the connected sums are performed away from fixed product collars
of \(\partial Y_\pm\cong\mathbb S^1\times\mathbb S^m\).  Suppose that,
in these product coordinates and after composition with
\(\theta\)-independent product diffeomorphisms that extend over
\(Y_\pm\), the gluing maps \(\partial Y_\pm\to\mathbb S^1\times I_\pm
\subset\mathbb S^1\times P_m\) have the form
\begin{equation}\label{eq:universal-attaching-maps}
 \varphi_\pm(e^{\mathrm i\theta},y)
 =\bigl(e^{\mathrm i\theta},\alpha_\pm(\theta)y\bigr),
 \qquad
 \alpha_\pm\colon\mathbb S^1\longrightarrow SO(m+1).
\end{equation}
Then
\begin{equation}\label{eq:universal-two-core-result}
 Y_+\cup_{\varphi_+}(\mathbb S^1\times P_m)
 \cup_{\varphi_-}Y_-
 \cong
 (D^2\times\mathbb S^m)
 \mathbin{\#}^{\,r_++r_-}(\mathbb S^2\times\mathbb S^m)
 \mathbin{\#}S(\xi),
\end{equation}
where \(S(\xi)\to\mathbb S^2\) is the oriented
\(\mathbb S^m\)-bundle obtained from two copies of
\(D^2\times\mathbb S^m\) by the boundary identification
\begin{equation}\label{eq:universal-clutching-map}
 (e^{\mathrm i\theta},y)
 \longmapsto
 \bigl(e^{-\mathrm i\theta},
  \alpha_-(\theta)\alpha_+(\theta)^{-1}y\bigr).
\end{equation}
Moreover:
\begin{enumerate}[label=\textup{(\alph*)}]
\item if \([\alpha_+]=[\alpha_-]\) in \(\pi_1SO(m+1)\), then
\(S(\xi)\cong\mathbb S^2\times\mathbb S^m\),
\item if \(\alpha_+=\alpha_-=\alpha\) as maps, then the boundary of the
\(D^2\times\mathbb S^m\) summand on the right-hand side of
\eqref{eq:universal-two-core-result} is parametrized, relative to its
product coordinates, by \((e^{\mathrm i\theta},y)\mapsto
(e^{\mathrm i\theta},\alpha(\theta)y)\),
\item if a diffeomorphism of \(Y_+\) has already been fixed and is a
product map on the boundary collar, then the diffeomorphism in
\eqref{eq:universal-two-core-result} can be chosen to agree with it on
all of \(Y_+\).
\end{enumerate}
\end{lemma}

\begin{proof}
The main idea of the proof is, gluing \(Y_+\) and \(Y_-\) to \(\mathbb S^1\times P_m\)
amounts to two framed circle surgeries on \(\mathbb S^1\times
D^{m+1}\).  The first surgery is absorbed into the residual
\(D^2\times\mathbb S^m\) factor, while the second, performed along the
boundary of an embedded core disk, contributes the sphere-bundle summand
\(S(\xi)\), whose clutching loop is the relative framing
\(\alpha_-\alpha_+^{-1}\).

\emph{Step 1: handle structure.}  Relative to its two incoming boundary
components, \(P_m\) admits a Morse function with exactly one critical
point, of index one, equivalently,
\begin{equation*}
 P_m\cong
 \bigl((I_+\sqcup I_-)\times[0,1]\bigr)
 \mathop{\cup}_{\mathbb S^0\times D^m}(D^1\times D^m),
\end{equation*}
where the two components of the attaching region
\(\mathbb S^0\times D^m\) lie in \(I_+\times\{1\}\) and
\(I_-\times\{1\}\).  After taking the product with \(\mathbb S^1\), the
only nonproduct piece is the round \(1\)-handle
\(\mathbb S^1\times D^1\times D^m\), attached along
\(\mathbb S^1\times\mathbb S^0\times D^m\).  Perturbing the product
Morse--Bott function \((\theta,x)\mapsto f(x)+\varepsilon h(\theta)\),
where \(f\) is the Morse function above and \(h\) is a Morse function
on \(\mathbb S^1\) with one minimum and one maximum, splits the round
handle into an ordinary \(1\)-handle and an ordinary \(2\)-handle:
\begin{equation*}
 \mathbb S^1\times D^1\times D^m
 \cong
 \bigl(D^1\times D^{m+1}\bigr)\cup\bigl(D^2\times D^m\bigr).
\end{equation*}
The \(1\)-handle joins \(Y_+\) to \(Y_-\), the attaching circle of the
\(2\)-handle is identified in Step~2.

\emph{Step 2: surgery description and framings.}  Fill the two deleted
\((m+1)\)-disks of \(P_m\), let \(p_\pm\in\operatorname{int}D^{m+1}\)
be their centers, and consider the parallel circles
\(C_\pm=\mathbb S^1\times\{p_\pm\}\subset\mathbb S^1\times D^{m+1}\)
with product tubular neighborhoods
\(\mathcal N(C_\pm)\cong\mathbb S^1\times D^{m+1}\).  Then
\begin{equation*}
 \mathbb S^1\times P_m
 \cong
 (\mathbb S^1\times D^{m+1})\setminus
 \bigl(\operatorname{int}\mathcal N(C_+)\sqcup\operatorname{int}\mathcal N(C_-)\bigr),
\end{equation*}
so that, ignoring for the moment the connected-sum factors in the
interiors of \(Y_\pm\), gluing the residual \(D^2\times\mathbb S^m\)
factors of \(Y_+\) and \(Y_-\) amounts to framed surgeries on \(C_+\)
and \(C_-\).  Let \(\eta\colon[0,1]\to D^{m+1}\) be the straight arc
from \(p_-\) to \(p_+\).  After deleting small neighborhoods of its
endpoints, \(\mathbb S^1\times\eta\) is a properly embedded annulus in
\(\mathbb S^1\times P_m\) with one boundary circle on each
\(\partial\mathcal N(C_\pm)\).

Perform the first surgery, along \(C_+\), using \(Y_+\), and choose a
meridional disk \(D^2\times\{y_0\}\) in its \(D^2\times\mathbb S^m\)
factor.  Under \(\varphi_+\), the circle \(\partial(D^2\times\{y_0\})\)
is identified with the \(C_+\)-boundary of \(\mathbb S^1\times\eta\).
After smoothing the corner,
\begin{equation*}
 D_{\mathrm{core}}
 =
 \bigl(\mathbb S^1\times\eta\bigr)
 \mathop{\cup}_{\varphi_+|_{\partial(D^2\times\{y_0\})}}
 \bigl(D^2\times\{y_0\}\bigr)
\end{equation*}
is a smoothly embedded \(2\)-disk whose boundary is the longitude
\(C_-\) on \(\partial\mathcal N(C_-)\), it is the core of the
\(2\)-handle of Step~1.  Transporting the normal framing of the first
surgery across \(D_{\mathrm{core}}\) gives, in the fixed product
coordinates, the framing \(\alpha_+\) along \(\partial
D_{\mathrm{core}}\), whereas the second surgery uses the framing
\(\alpha_-\) prescribed by \(\varphi_-\).  Hence the second surgery has
relative framing
\begin{equation*}
 \beta\colon\mathbb S^1\longrightarrow SO(m+1),
 \qquad
 \beta(\theta)=\alpha_-(\theta)\alpha_+(\theta)^{-1}.
\end{equation*}

\emph{Step 3: the local model.}  Let \(\xi\to\mathbb S^2\) be the oriented
rank-\((m+1)\) vector bundle with clutching loop \(\beta\).  Writing
\(\mathbb S^2=D^2_+\cup_{\mathbb S^1}D^2_-\), its unit sphere bundle is
\begin{equation*}
 S(\xi)
 =
 \bigl(D^2_+\times\mathbb S^m\bigr)
 \mathop{\cup}_{\Psi_\beta}
 \bigl(D^2_-\times\mathbb S^m\bigr),
 \qquad
 \Psi_\beta(e^{\mathrm i\theta},y)
 =\bigl(e^{-\mathrm i\theta},\beta(\theta)y\bigr),
\end{equation*}
which is exactly the identification \eqref{eq:universal-clutching-map},
the reversal \(e^{\mathrm i\theta}\mapsto e^{-\mathrm i\theta}\) matches
the opposite orientations of \(\partial D^2_\pm\).  Let
\(\mathcal N(D_{\mathrm{core}})\cong D^{m+2}\) be a regular
neighborhood with rounded corners, and let \(U_\beta\) be the result of
the second surgery on \(\partial D_{\mathrm{core}}\), performed on the
outward side of this neighborhood with relative framing \(\beta\).  The
standard local model for framed surgery on the boundary of an embedded
\(2\)-disk gives a diffeomorphism of pairs
\begin{equation*}
 \bigl(U_\beta,\partial U_\beta\bigr)
 \cong
 \bigl(S(\xi)\setminus\operatorname{int}D^{m+2},\,\mathbb S^{m+1}\bigr).
\end{equation*}
Thus the second surgery replaces an interior \((m+2)\)-disk of the
residual \(D^2\times\mathbb S^m\) factor by
\(S(\xi)\setminus\operatorname{int}D^{m+2}\), and, when \(r_+=r_-=0\),
\begin{equation*}
 Y_+\mathop{\cup}_{\varphi_+}(\mathbb S^1\times P_m)\mathop{\cup}_{\varphi_-}Y_-
 \cong
 (D^2\times\mathbb S^m)\mathbin{\#}S(\xi).
\end{equation*}

\emph{Step 4: the connected-sum factors.}  The \(r_++r_-\) summands
\(\mathbb S^2\times\mathbb S^m\) lie in the interiors of \(Y_\pm\),
away from the fixed collars, while all handle moves and surgeries above
are supported in \(\mathbb S^1\times P_m\) and on the outer sides of
those collars.  These summands are therefore unaffected, which proves
\eqref{eq:universal-two-core-result}.

\emph{Step 5: finish the proof.}  If \([\alpha_+]=[\alpha_-]\), then
\([\beta]=[\alpha_-][\alpha_+]^{-1}=1\) in \(\pi_1SO(m+1)\), so \(\xi\)
is trivial by the clutching classification of oriented vector bundles
over \(\mathbb S^2\), and \(S(\xi)\cong\mathbb S^2\times\mathbb S^m\),
which gives~\textup{(a)}.  If \(\alpha_+=\alpha_-=\alpha\), then
\(\beta\equiv\operatorname{Id}\), while the first surgery carries the
common framing \(\alpha\) through the residual \(D^2\times\mathbb S^m\)
factor to its boundary, which implies~\textup{(b)}.  Finally, all moves above
are supported in \(\mathbb S^1\times P_m\) and on the outer sides of the
fixed collars, so they can be performed without changing a prescribed
diffeomorphism of \(Y_+\) that is a product map on its collar, which
finishes~\textup{(c)}.
\end{proof}

We now give the proof of Part~\textup{(i)} in Theorem \ref{thm:further-properties}.
\begin{theorem}\label{thm:universal-cover-topology}
For every \(N\ge0\),
\begin{equation*}
 \widehat M_N\cong
 (D^2\times\mathbb S^m)
 \mathbin{\#}^{\,2^{N+1}-1}
 (\mathbb S^2\times\mathbb S^m).
\end{equation*}
Consequently,
\begin{equation*}
 \widetilde M\cong
 \mathbb R^{m+2}\mathbin{\#}_{\infty}
 (\mathbb S^2\times\mathbb S^m),
\end{equation*}
and \(\widetilde M\) is one-ended and simply connected at infinity.
\end{theorem}

\begin{proof}
We compute the finite lifts \(\widehat M_N\) by induction on \(N\),
applying Lemma~\ref{lem:universal-two-core} at each stage and keeping
track of the boundary framings.  We then show that the
\(\widehat M_N\) exhaust \(\widetilde M\) and identify their union.

\emph{Step 1: the finite lifts.}  By \eqref{eq:un-finite-pi1},
\(\pi_1(M_N)\cong\mathbb Z/2^{N+1}\mathbb Z\), and the inclusion
\(\pi_1(M_N)\to\pi_1(M)\) is injective.  Hence
\(q|_{\widehat M_N}\colon\widehat M_N\to M_N\) is the universal covering
of \(M_N\). It has degree \(2^{N+1}\), and \(\widehat M_N\) is compact
and simply connected.  The lift of the mapping torus of \(B_N\) that
lies in \(\widehat M_N\) has the product identification
\begin{equation}\label{eq:universal-cyclic-lift}
 \frac{\mathbb R\times P_m}
 {(t,x)\sim(t+2^{N+1}T_N,x)}
 \cong\mathbb S^1\times P_m,
 \qquad
 (e^{\mathrm i\theta},x)
 \longmapsto
 \left[\frac{2^{N+1}T_N\theta}{2\pi},x\right].
\end{equation}
Its two incoming boundary components carry no \(\theta\)-dependent
sphere rotation.  Substituting \(2^{N+1}\theta\) for the circle
parameter in the first formula of
\eqref{eq:universal-mapping-torus-boundary-maps} shows that the outgoing
boundary is rotated by
\begin{equation}\label{eq:universal-lifted-twist}
 \theta\longmapsto\exp(2^N\theta K_m).
\end{equation}

\emph{Step 2: the induction of $\widehat{M}_N$.}  Set
\begin{equation*}
 \alpha_{-1}(\theta)=\operatorname{Id},
 \qquad
 \alpha_N(\theta)=\exp\bigl((2^{N+1}-1)\theta K_m\bigr)
 \quad(N\ge0).
\end{equation*}
We show by induction, starting from
\(\widehat M_{-1}=D^2\times\mathbb S^m\), that
\begin{equation}\label{eq:universal-finite-core}
 \widehat M_N\cong
 (D^2\times\mathbb S^m)
 \mathbin{\#}^{\,2^{N+1}-1}
 (\mathbb S^2\times\mathbb S^m),
\end{equation}
with outgoing boundary parametrized by
\((e^{\mathrm i\theta},y)\mapsto(e^{\mathrm i\theta},\alpha_N(\theta)y)\)
relative to the residual \(D^2\times\mathbb S^m\) factor.

Suppose this holds for \(N-1\).  The preimage of \(M_{N-1}\) in
\(\widehat M_N\) has two components, each diffeomorphic to
\(\widehat M_{N-1}\), and they are attached to the two incoming
boundary components of the product \eqref{eq:universal-cyclic-lift}.
On the first component we use the circle parameter \(\theta\), the
natural parameter on the second differs by
\(\theta\mapsto\theta-\pi/2^N\).  Since \(\alpha_{N-1}\) lies in the
one-parameter subgroup generated by \(K_m\), this translation multiplies
\(\alpha_{N-1}\) by a constant orthogonal transformation, which we
absorb into a \(\theta\)-independent product diffeomorphism of the
second component, it extends over the residual \(D^2\times\mathbb S^m\)
factor, and the first component and the identification
\eqref{eq:universal-cyclic-lift} are left unchanged.  Both gluing maps
then have the form \eqref{eq:universal-attaching-maps} with
\(\alpha_+=\alpha_-=\alpha_{N-1}\).  By
Lemma~\ref{lem:universal-two-core}\textup{(a)}, the gluing contributes
one new untwisted summand \(\mathbb S^2\times\mathbb S^m\), in addition
to the \(2(2^N-1)\) summands of the two copies of
\(\widehat M_{N-1}\), the total is \(2^{N+1}-1\).  By
Lemma~\ref{lem:universal-two-core}\textup{(b)}, the common framing
\(\alpha_{N-1}\) passes to the boundary of the residual factor, and
combining it with the twist \eqref{eq:universal-lifted-twist} gives
\begin{equation*}
 \alpha_N(\theta)
 =\exp(2^N\theta K_m)\alpha_{N-1}(\theta)
 =\exp\bigl((2^{N+1}-1)\theta K_m\bigr).
\end{equation*}
This completes the induction.  For \(m=2\), the loop \(\alpha_N\)
represents the nonzero element of \(\pi_1SO(3)\), because \(2^{N+1}-1\)
is odd, this causes no twisting, since only the relative loop
\(\alpha_-\alpha_+^{-1}\equiv\operatorname{Id}\) enters the clutching
map \eqref{eq:universal-clutching-map}.  (For \(m=3\), simultaneous
rotation in two orthogonal planes is even null-homotopic in \(SO(4)\).)

\emph{Step 3: the lifts exhaust \(\widetilde M\).}  The inclusions
\(M_N\subset\operatorname{int}M_{N+1}\) give
\(\widehat M_N\subset\operatorname{int}\widehat M_{N+1}\).  We claim
that
\begin{equation}\label{eq:universal-cover-exhaustion}
 \widetilde M=\bigcup_{N\ge-1}\widehat M_N.
\end{equation}
Indeed, every \(\widetilde x\in\widetilde M\) lies in a translate
\(\gamma\widehat M_j\) for some \(j\) and some deck transformation
\(\gamma\).  By \eqref{eq:un-prufer},
\(\pi_1(M)=\bigcup_{N\ge0}\langle a_N\rangle\), so
\(\gamma\in\langle a_N\rangle\) for some \(N\ge j\).  This subgroup is
the stabilizer of \(\widehat M_N\), and hence
\(\gamma\widehat M_j\subset\gamma\widehat M_N=\widehat M_N\).

\emph{Step 4: exhaustion.}  Set \(Y_0=D^2\times\mathbb S^m\).
Given \(Y_r\), regard the \(D^2\times\mathbb S^m\) factor meeting
\(\partial Y_r\) as one half of a new copy of
\(\mathbb S^2\times\mathbb S^m\).  Glue to
\(\partial Y_r=\mathbb S^1\times\mathbb S^m\) a second copy of
\(D^2\times\mathbb S^m\) with an interior \((m+2)\)-disk removed, this
completes a punctured copy of \(\mathbb S^2\times\mathbb S^m\) with
boundary \(\mathbb S^{m+1}\).  Then glue along this sphere another copy
of \(D^2\times\mathbb S^m\) with an interior disk removed, and call the
result \(Y_{r+1}\).  Inductively,
\begin{equation*}
 Y_r\cong(D^2\times\mathbb S^m)\mathbin{\#}^{\,r}(\mathbb S^2\times\mathbb S^m),
 \qquad
 Y_r\subset\operatorname{int}Y_{r+1}.
\end{equation*}
When completing the residual half, we give the complementary half the
same fiberwise boundary parametrization, so the relative clutching loop
is the identity and the new summand is untwisted.  The boundary
parametrization of the next residual half is independent of this gluing,
because that half is attached along its \(\mathbb S^{m+1}\) boundary, we
choose it so that, for \(r=2^{N+1}-1\), the boundary of \(Y_r\) is
parametrized by \(\alpha_N\).

\emph{Step 5: identification of \(\widetilde M\).}
Lemma~\ref{lem:universal-two-core}\textup{(c)} and Step~2 give,
inductively, diffeomorphisms
\begin{equation*}
 F_N\colon\widehat M_N\longrightarrow Y_{2^{N+1}-1},
 \qquad
 F_N|_{\widehat M_{N-1}}=F_{N-1},
\end{equation*}
which are product maps on the boundary collars.  By Step~3, they define
a diffeomorphism from \(\widetilde M\) onto \(\bigcup_{r\ge0}Y_r\).  To
identify this union, choose pairwise disjoint balls in
\(\mathbb R^{m+2}\) tending to infinity along a proper ray, and replace
each of them by a punctured copy of \(\mathbb S^2\times\mathbb S^m\).
Cutting each inserted copy along its separating
\(\mathbb S^1\times\mathbb S^m\) into two \(D^2\times\mathbb S^m\)
halves produces an exhaustion whose successive members are related
exactly as \(Y_r\subset\operatorname{int}Y_{r+1}\).  The two
exhaustions are therefore diffeomorphic piece by piece, relative to
their boundary collars, and
\begin{equation*}
 \bigcup_{r\ge0}Y_r
 \cong
 \mathbb R^{m+2}\mathbin{\#}_{\infty}(\mathbb S^2\times\mathbb S^m).
\end{equation*}

\emph{Step 6: ends and fundamental group at infinity.}  For every \(r\),
the complement of \(Y_r\) is the connected infinite tail of the
construction, so the union has one end.  Every piece of this tail is
either \(D^2\times\mathbb S^m\) or such a manifold with an interior
\((m+2)\)-disk removed, these pieces are simply connected because
\(m\ge2\), and consecutive pieces meet along the connected manifolds
\(\mathbb S^1\times\mathbb S^m\) or \(\mathbb S^{m+1}\).  By the
Seifert--van Kampen theorem, every finite portion of the tail is simply
connected, and since every loop in the tail lies in a finite portion,
the whole complement of \(Y_r\) is simply connected.  Hence
\(\widetilde M\) is simply connected at infinity.
\end{proof}

\subsection{Prescribing the growth exponents}
\label{subs:prescribing-exponents}

In the following, we recall several notations for the holonomy interpolation cylinders.  For
\(j\ge0\), let
\begin{equation*}
 \mathcal C_j=[0,s_j]\times\mathbb S^1\times\mathbb S^m
\end{equation*}
be the holonomy interpolation cylinder of the block \(B_j\), and let
\(\ell_j(s),\rho_j(s),\omega_j(s)\) be the functions in its metric
\eqref{eq:interpolation-cylinder-metric}.  Let \(\mathfrak p_{0,j}\)
and \(\Lambda_j\) be the parameters \(\mathfrak p_0\) and
\(\Lambda_{\mathrm{con}}\) of
Proposition~\ref{prop:holonomy-interpolation-cylinder} used for
\(\mathcal C_j\), so that
\begin{equation}
 \mathfrak p_{1,j}
 =\frac{\mathfrak p_{0,j}}{\Lambda_j^2},
 \qquad
 s_j
 =\frac{\Lambda_j^2-1}{2\mathfrak p_{0,j}}.
 \label{eq:indexed-cylinder-parameters}
\end{equation}
We call
\(g_{\mathcal C_j}=\dd s^2+g_{\ell_j(s),\rho_j(s),\omega_j(s)}\) the
\emph{model metric} on \(\mathcal C_j\), its radius profiles are
\begin{equation}
 \ell_j(s)
 =\ell_j(0)\sqrt{1+2\mathfrak p_{0,j}s},
 \qquad
 \rho_j(s)
 =\rho_j(0)
 \exp\left(
   -\frac{\mathfrak p_{1,j}}{100s_j}s(s_j-s)
 \right).
 \label{eq:indexed-cylinder-functions}
\end{equation}
Recall that \(\ell_j(0)\) and \(\rho_j(0)\) differ from the incoming
radius \(\bar\ell_j\) and \(\bar\rho_j\) of \(B_j\), by
\eqref{eq:indexed-cylinder-functions} and
Proposition~\ref{prop:holonomy-interpolation-cylinder}, the outgoing
radius are \(\bar\ell_{j+1}=\Lambda_j\ell_j(0)\) and
\(\bar\rho_{j+1}=\rho_j(0)\).  For \(0\le a<b\le1\) we abbreviate
\begin{equation*}
 \mathcal C_j[a,b]
 =
 [as_j,bs_j]\times\mathbb S^1\times\mathbb S^m.
\end{equation*}
By \eqref{eq:un-central-slab}, the central subcylinder
\(\mathcal P_j=\mathcal C_j[1/4,3/4]\) is disjoint from every smoothing
collar, so the final metric of \(M\) agrees with \(g_{\mathcal C_j}\)
there.

\begin{lemma}
\label{lem:protected-cylinder-volume}
For \(\Lambda\ge2\) and \(0\le x\le1\), set
\begin{equation}
 \mathcal I_m(\Lambda,x)
 =
 \int_0^x
 \sqrt{1+(\Lambda^2-1)t}\,
 \exp\left(
   -\frac{m(\Lambda^2-1)}{200\Lambda^2}t(1-t)
 \right)\dd t.
 \label{eq:dimensionless-volume-function}
\end{equation}
Then the following hold for every \(j\ge0\) and \(0\le a<b\le1\).
\begin{enumerate}[label=\textup{(\alph*)}]
\item The model metric satisfies
\begin{equation}
 \operatorname{Vol}_{g_{\mathcal C_j}}\bigl(\mathcal C_j[a,b]\bigr)
 =
 2\pi\operatorname{Vol}(\mathbb S^m)\,
 \ell_j(0)\rho_j(0)^m\, s_j\,
 \bigl(\mathcal I_m(\Lambda_j,b)-\mathcal I_m(\Lambda_j,a)\bigr),
 \label{eq:protected-cylinder-volume}
\end{equation}
and the same identity holds for the final metric of \(M\) when
\(1/4\le a<b\le3/4\).
\item For every \(\Lambda\ge2\),
\begin{equation}
 c_{m,a,b}\,\Lambda
 \le\mathcal I_m(\Lambda,b)-\mathcal I_m(\Lambda,a)
 \le(b-a)\Lambda,
 \qquad
 c_{m,a,b}=\frac23e^{-m/800}\bigl(b^{3/2}-a^{3/2}\bigr)>0.
 \label{eq:dimensionless-volume-asymptotic}
\end{equation}
Consequently, \(\operatorname{Vol}_{g_{\mathcal C_j}}(\mathcal C_j[a,b])\)
is comparable to
\begin{equation}
 \ell_j(0)\rho_j(0)^m\,\Lambda_j\,s_j
 \label{eq:cylinder-volume-scale}
\end{equation}
with constants depending only on \(m\), \(a\), and \(b\).
\item For \(0\le x\le s_j\),
\begin{equation}
 C^{-1}\ell_j(0)\rho_j(0)^m
   \bigl(x+\sqrt{\mathfrak p_{0,j}}\,x^{3/2}\bigr)
 \le\operatorname{Vol}_{g_{\mathcal C_j}}\{0\le s\le x\}
 \le C\ell_j(0)\rho_j(0)^m
   \bigl(x+\sqrt{\mathfrak p_{0,j}}\,x^{3/2}\bigr),
 \label{eq:vg-partial-cylinder}
\end{equation}
where \(C\) depends only on \(m\).
\end{enumerate}
\end{lemma}

\begin{proof}
At each fixed \(s\), replacing the circle vector \(\partial_\theta\) by
\(\partial_\theta-\omega_j(s)K_m\) diagonalizes
\(g_{\ell_j(s),\rho_j(s),\omega_j(s)}\) relative to a
\(g_{\mathbb S^m}\)-orthonormal frame.  This change of frame has
determinant one, so the twist does not affect the volume density, and
\begin{equation*}
 \dd V_{g_{\mathcal C_j}}
 =
 \ell_j(s)\rho_j(s)^m\,
 \dd s\,\dd\theta\,\dd V_{g_{\mathbb S^m}}.
\end{equation*}
Substituting \(s=ts_j\) and using \eqref{eq:indexed-cylinder-parameters}
and \eqref{eq:indexed-cylinder-functions}, so that
\(2\mathfrak p_{0,j}s_j=\Lambda_j^2-1\) and
\(\mathfrak p_{1,j}s_j=(\Lambda_j^2-1)/(2\Lambda_j^2)\), gives
\eqref{eq:protected-cylinder-volume}.  Since the final metric agrees
with \(g_{\mathcal C_j}\) on \(\mathcal P_j\), the identity also holds
for the final metric when \([a,b]\subset[1/4,3/4]\).  This
proves~\textup{(a)}.

For~\textup{(b)}, note that \(t(1-t)\le1/4\), so the exponential factor
in \eqref{eq:dimensionless-volume-function} lies in
\([e^{-m/800},1]\), and that
\(\Lambda^2t\le1+(\Lambda^2-1)t\le\Lambda^2\) for \(0\le t\le1\).
Hence the integrand lies between \(e^{-m/800}\Lambda\sqrt t\) and
\(\Lambda\), integrating over \([a,b]\) gives
\eqref{eq:dimensionless-volume-asymptotic}, and the comparison with
\eqref{eq:cylinder-volume-scale} follows from~\textup{(a)}.

For~\textup{(c)}, \eqref{eq:un-rho-bound} gives
\(e^{-1/800}\rho_j(0)\le\rho_j(s)\le\rho_j(0)\), so the volume density
is comparable, with constants depending only on \(m\), to
\(\ell_j(0)\rho_j(0)^m\sqrt{1+2\mathfrak p_{0,j}s}\).  Since
\(\sqrt{1+2\mathfrak p_{0,j}s}\) is comparable to
\(1+\sqrt{\mathfrak p_{0,j}s}\), integrating over \([0,x]\) gives
\eqref{eq:vg-partial-cylinder}.
\end{proof}

By Lemma \ref{lem:protected-cylinder-volume} and the dominated convergence theorem, \((\mathcal I_m(\Lambda,b)-\mathcal
I_m(\Lambda,a))/\Lambda\) converges to
\(\int_a^b\sqrt t\,e^{-mt(1-t)/200}\dd t\) as \(\Lambda\to\infty\), so
the bounds in \eqref{eq:dimensionless-volume-asymptotic} are sharp up
to constants.

The next lemma shows that $M$ and its universal cover have the same lower and upper logarithmic volume-growth exponents.

\begin{lemma}
\label{lem:vg-cover-comparison}
Let \(\widetilde p\in q^{-1}(p)\), and suppose that
\(\log s_j\ge(j+1)^2\) for every \(j\ge0\).  Then, for all sufficiently
large \(r\),
\begin{equation}
 1\le
 \frac{\operatorname{Vol}_{\widetilde M}(B_r(\widetilde p))}
      {\operatorname{Vol}_M(B_r(p))}
 \le 2\exp\!\bigl((\log2)\sqrt{\log(4r)}\bigr).
 \label{eq:vg-cover-bound}
\end{equation}
Consequently,
\begin{equation}\label{eq:vg-cover-log-comparison}
 \lim_{r\to\infty}
 \frac{\log\operatorname{Vol}_{\widetilde M}(B_r(\widetilde p))
       -\log\operatorname{Vol}_M(B_r(p))}{\log r}=0,
\end{equation}
so \(M\) and \(\widetilde M\) have the same lower and upper logarithmic
volume-growth exponents.  Furthermore,
\begin{equation}
 \lim_{r\to\infty}
 \frac{\operatorname{Vol}_{\widetilde M}(B_r(\widetilde p))}
      {\operatorname{Vol}_M(B_r(p))}
 =\infty.
 \label{eq:vg-cover-ratio}
\end{equation}
\end{lemma}

\begin{proof}
Recall the compact stages \(M_j\), with boundaries \(\Sigma_j\), from
\eqref{eq:compact-exhaustion}, and their lifts \(\widehat M_j\) from
Subsection~\ref{subs:topology-of-the-universal-cover}, by
\eqref{eq:un-finite-pi1}, the covering \(\widehat M_j\to M_j\) has
degree \(2^{j+1}\).  Set \(d_j^-=\min_{y\in\Sigma_j}d(p,y)\).  Every
path from \(p\) to \(\Sigma_j\) crosses the slab
\(\mathcal C_j[1/4,1/2]\), on which the metric contains the summand
\(\dd s^2\), hence \(d_j^-\ge s_j/4\).

Given \(r\), let \(j\ge0\) be the smallest index with \(r<d_j^-\).  Then
\(B_r(p)\subset M_j\), so every minimizing geodesic of length less than
\(r\) starting at \(\widetilde p\) projects into \(M_j\), and
\(B_r(\widetilde p)\subset\widehat M_j\).  Since \(q\) maps
\(B_r(\widetilde p)\) onto \(B_r(p)\) with at most \(2^{j+1}\) preimages
of each point,
\begin{equation*}
 \operatorname{Vol}_M(B_r(p))
 \le\operatorname{Vol}_{\widetilde M}(B_r(\widetilde p))
 \le 2^{j+1}\operatorname{Vol}_M(B_r(p)).
\end{equation*}
If \(j\ge1\), the minimality of \(j\) gives
\(r\ge d_{j-1}^-\ge s_{j-1}/4\ge e^{j^2}/4\), so
\(j\le\sqrt{\log(4r)}\), this proves \eqref{eq:vg-cover-bound}.  Taking
logarithms and dividing by \(\log r\) gives
\begin{equation*}
 0\le
 \frac{\log\operatorname{Vol}_{\widetilde M}(B_r(\widetilde p))
       -\log\operatorname{Vol}_M(B_r(p))}{\log r}
 \le
 \frac{\log2+(\log2)\sqrt{\log(4r)}}{\log r},
\end{equation*}
and the right-hand side tends to zero, which proves
\eqref{eq:vg-cover-log-comparison}.

For \eqref{eq:vg-cover-ratio}, fix \(N\ge0\) and consider the group
\(\langle a_N\rangle\cong\pi_1(M_N)\) of deck transformations, which
has order \(2^{N+1}\) by \eqref{eq:un-finite-pi1}.  Set
\(D_N=\max_{\gamma\in\langle a_N\rangle}d(\widetilde p,\gamma\widetilde p)\).
If \(x\in B_{r-D_N}(p)\) and \(\widetilde x\) is the lift of \(x\) with
\(d(\widetilde p,\widetilde x)<r-D_N\), then all \(2^{N+1}\) points
\(\gamma\widetilde x\), \(\gamma\in\langle a_N\rangle\), lie in
\(B_r(\widetilde p)\).  Integrating this multiplicity bound and applying
the Bishop--Gromov inequality on \(M\) gives
\begin{equation*}
 \frac{\operatorname{Vol}_{\widetilde M}(B_r(\widetilde p))}
      {\operatorname{Vol}_M(B_r(p))}
 \ge 2^{N+1}
 \frac{\operatorname{Vol}_M(B_{r-D_N}(p))}{\operatorname{Vol}_M(B_r(p))}
 \ge 2^{N+1}\left(1-\frac{D_N}{r}\right)^n.
\end{equation*}
Letting \(r\to\infty\) and then \(N\to\infty\) yields
\eqref{eq:vg-cover-ratio}.
\end{proof}

In the following lemma, we will give the lower bound for the volume of the necks inside the mapping torus.  Recall that the necks of Proposition~\ref{prop:unified-neck} are
constructed in Lemmas~\ref{lem:slow-logarithmic-neck}
and~\ref{lem:neck-endpoint-normalization}, and that the outer boundaries
of the equivariant cobordisms are rounded by the path cylinders
\eqref{eq:path-cylinder-metric} of
Theorem~\ref{thm:path-cylinder-unified}.  The key point is that the
slowly varying neck of Lemma~\ref{lem:slow-logarithmic-neck} is very
long compared with how thin it becomes: its waist \(w\) satisfies
\(-\log w\lesssim\sqrt{\log(\text{length})}\).  A long neck therefore
fills a ball whose volume is almost \((m+1)\)-dimensional.

\begin{lemma}\label{lem:vg-neck-growth}
Suppose that there is a constant \(C>1\) such that, for every
\(j\ge1\), each point of \(\partial_iB_j\) can be joined to \(p\) by a
path of length at most \(Cs_{j-1}\) in the union of the initial cap and
\(B_0,\ldots,B_{j-1}\).  Then, after possibly increasing \(C\), there
are radius \(r\to\infty\) such that
\begin{equation}
 \operatorname{Vol}_M B_r(p)
 \ge C^{-1}r^{m+1}\exp\!\bigl(-C\sqrt{\log r}\bigr).
 \label{eq:vg-neck-ball}
\end{equation}
Consequently, the upper logarithmic volume-growth exponent of \(M\), and
hence that of \(\widetilde M\), is at least \(m+1\).
\end{lemma}

\begin{proof}
We will prove the lower bound for the volume of the middle portions of the two necks
in the \(j\)-th mapping torus, and show that these portions lie in a
ball of radius comparable to the neck length.  Throughout, \(C\) denotes
a constant independent of \(j\).

\emph{Step 1: introduce notations.}  For the block \(B_j\), let
\(r^{(j)},\rho^{(j)},w^{(j)}\) be the parameters of
Proposition~\ref{prop:unified-neck}, and let
\(s^{(j)},\ell^{(j)},\lambda^{(j)},A^{(j)},B^{(j)}\) be those of
Lemma~\ref{lem:neck-endpoint-normalization}.  The block is built with
\(R_*=\bar\rho_j\), so the homothety in the proof of
Theorem~\ref{thm:ricci-positive-equivariant-cobordism} is
\begin{equation*}
 c^{(j)}=\frac{\bar\rho_j\lambda^{(j)}}{\rho^{(j)}},
\end{equation*}
before smoothing, each neck therefore joins a round sphere of radius
\(\bar\rho_j\) to an angular sphere and has length
\(c^{(j)}\ell^{(j)}\).  By Lemma~\ref{lem:relative-gluing-unified}, all
smoothing collars can be chosen to avoid the middle portions
\(\{\ell^{(j)}/2\le s^{(j)}\le3\ell^{(j)}/4\}\) of the necks, and by
part~\textup{(c)} of Lemma~\ref{lem:relative-gluing-unified}, we may choose them at each stage so that the final metric lies between one half and twice the piecewise metric.
Thus, the middle portions are unchanged, and lengths of paths change by
at most a fixed factor.

\emph{Step 2: the incoming radius do not decrease.}  We show that
\begin{equation}
 \bar\ell_j\ge\bar\ell_0>0,\qquad \bar\rho_j\ge\bar\rho_0>0.
 \label{eq:vg-neck-incoming-radius}
\end{equation}
For the circle radius, \eqref{eq:un-axial-ratio} shows that the
outgoing circle radius of \(W_m\) exceeds \(\bar\ell_j/2\), and the
holonomy interpolation cylinder multiplies it by
\(\Lambda_j\ge2\), hence \(\bar\ell_{j+1}>\bar\ell_j\).  For the sphere
radius, we follow the volumes of the slices from the incoming sphere to
the outgoing one.  By \eqref{eq:gaussian-normal}, the mean curvature of
the neck slices, oriented towards the angular end, is strictly
decreasing, it is positive at that end by
Lemma~\ref{lem:neck-endpoint-normalization}, hence positive throughout,
and the slice volumes increase along the neck.  By
Lemma~\ref{lem:angular-latitude-form}, translating the center does not
change the induced metric of an angular sphere, and by
\eqref{eq:unified-base-3-77}--\eqref{eq:unified-base-3-78}, applied
with the angular radius varying from \(r^{(j)}\) to \(4r^{(j)}\), the
angular spheres between these radius have positive mean curvature.
Finally, the same argument applies to the path cylinder
\eqref{eq:path-cylinder-metric}, whose slices also have positive mean
curvature.  Hence the outgoing round sphere of the cobordism has larger
volume than the incoming sphere of radius \(\bar\rho_j\).  Since the
relative gluings preserve the exposed boundary metrics and the holonomy
interpolation cylinder has equal sphere radius at its two ends, we get
\(\bar\rho_{j+1}>\bar\rho_j\).

\emph{Step 3: the angular sphere estimates.}  By \eqref{eq:unified-base-3-26},
the maximum \(u_*^{(j)}w^{(j)}\) of \(A_{\mathrm{ang}}^{(j)}\) is at
least its mean \(\cos r^{(j)}\ge\cos(\pi/10)\), and by
\eqref{eq:unified-base-3-31} it is less than one.  Thus
\begin{equation}
 \cos(\pi/10)\le u_*^{(j)}w^{(j)}<1,
 \qquad
 -\log w^{(j)}\le\log u_*^{(j)}-\log\cos(\pi/10),
 \label{eq:vg-neck-angular}
\end{equation}
where \(u_*^{(j)}\) is defined in
Lemma~\ref{lem:angular-slice-interpolation}.  Since
\(A_{\mathrm{ang}}^{(j)}<1\), the two intervals of length \(\pi/6\)
outside \([-\pi/3,\pi/3]\) contribute at most \(\pi/3\) to the integral
in \eqref{eq:unified-base-3-26}, while \(\cos^{m-1}\varphi\ge2^{1-m}\)
on \([-\pi/3,\pi/3]\).  Hence
\begin{equation}
 \int_{-\pi/2}^{\pi/2}A_{\mathrm{ang}}^{(j)}(\varphi)
                  \cos^{m-1}\varphi\,\dd\varphi
 \ge \frac{\pi}{2^{m-1}}
             \left(\cos\frac{\pi}{10}-\frac13\right)>0.
 \label{eq:vg-neck-weighted-area}
\end{equation}

\emph{Step 4: the waist versus the neck length.}  We use the parameters
\(\alpha^{(j)},\beta^{(j)},\sigma^{(j)},t_0^{(j)},t_1^{(j)}\) and the
function \(v^{(j)}\) from the proof of
Lemma~\ref{lem:slow-logarithmic-neck}, recall that
\(t_1^{(j)}>2t_0^{(j)}\).  At an interior maximum of
\(A_{\mathrm{ang}}^{(j)}\), the bound
\(\sec_{\widetilde g_{u,v_\alpha(u)}}\ge1+\mu\) from that proof and
\eqref{eq:unified-base-3-29} give \(u_*^{(j)}v^{(j)}(t_1^{(j)})<1\).
Together with the parameter relation \eqref{eq:unified-base-3-48}, this
yields
\begin{equation}
 \frac{\sigma^{(j)}}m\log u_*^{(j)}
 <\log\frac{v^{(j)}(t_1^{(j)})}{w^{(j)}}
 <-\log\cos(\pi/10).
 \label{eq:vg-neck-sigma}
\end{equation}
Next, at the same latitude and at \(t=2t_0^{(j)}\), positivity of the
radial Ricci curvature \eqref{eq:unified-base-3-61} gives
\begin{equation*}
 \log(2t_0^{(j)})>
 2+\beta^{(j)}
       \frac{m+\alpha^{(j)}(\alpha^{(j)}-2)}{m-\alpha^{(j)}}.
\end{equation*}
Since \(1<\alpha^{(j)}<m\), we have
\(m+\alpha(\alpha-2)\ge m-1\) and \(\alpha(m-\alpha)\le m^2/4\), so
\begin{equation*}
 \frac{m+\alpha^{(j)}(\alpha^{(j)}-2)}
      {\alpha^{(j)}(m-\alpha^{(j)})}
 \ge\frac{4(m-1)}{m^2}.
\end{equation*}
Combining this with \eqref{eq:I t0 bound}, \eqref{eq:unified-base-3-48},
and \eqref{eq:unified-base-3-53}, which give
\(\alpha^{(j)}\beta^{(j)}\ge\frac12\log u_*^{(j)}\), we obtain
\begin{equation*}
 \log t_0^{(j)}>
 2-\log2+\frac{2(m-1)}{m^2}\log u_*^{(j)}.
\end{equation*}
The tail identity \eqref{eq:I t0 identity}, the formula
\eqref{eq:unified-base-3-50}, and the upper bound in
\eqref{eq:I t0 bound} now give, using \eqref{eq:vg-neck-sigma} in the
last step,
\begin{equation}
 \log\frac{t_1^{(j)}}{t_0^{(j)}}
 \ge\frac{\log t_0^{(j)}}{\sigma^{(j)}}
 \ge C^{-1}(\log u_*^{(j)})^2.
 \label{eq:vg-neck-time}
\end{equation}
On the other hand, the normalization formulas in the proof of
Lemma~\ref{lem:neck-endpoint-normalization} give
\begin{equation}
 c^{(j)}\ell^{(j)}
 =\frac{\bar\rho_j}{\rho^{(j)}}
       \left(\frac{t_1^{(j)}}{t_0^{(j)}}-1\right).
 \label{eq:vg-neck-normalization}
\end{equation}
Since \(\rho^{(j)}<1\) and \(\bar\rho_j\ge\bar\rho_0\),
\eqref{eq:vg-neck-angular}--\eqref{eq:vg-neck-normalization} imply
\begin{equation}
 -\log w^{(j)}
 \le C\sqrt{\log(c^{(j)}\ell^{(j)})}+C
 \qquad\text{whenever }c^{(j)}\ell^{(j)}\ge1.
 \label{eq:vg-neck-logarithmic}
\end{equation}

\emph{Step 5: the volume of the middle portions.}  By
\eqref{eq:gaussian-normal} and the inequality
\(\lvert\II\rvert^2\ge(\tr\II)^2/m\), the function
\(\bigl(A^{(j)}(B^{(j)})^{m-1}\bigr)^{1/m}\) is concave in \(s^{(j)}\)
at each interior latitude.  Comparing it with its chord gives
\begin{equation*}
 A^{(j)}(s^{(j)},\varphi)B^{(j)}(s^{(j)})^{m-1}
 \ge\left(\frac{s^{(j)}}{\ell^{(j)}}\right)^m
      A_{\mathrm{ang}}^{(j)}(\varphi)(w^{(j)})^{m-1}.
\end{equation*}
Moreover, \(\ell^{(j)}<1\): by
Lemma~\ref{lem:neck-endpoint-normalization}, \(B^{(j)\prime\prime}<0\),
\(B^{(j)}(\ell^{(j)})=w^{(j)}\), and
\(B^{(j)\prime}(\ell^{(j)})>w^{(j)}\), so
\begin{equation*}
 w^{(j)}>w^{(j)}-B^{(j)}(0)
 \ge\ell^{(j)}B^{(j)\prime}(\ell^{(j)})
 >\ell^{(j)}w^{(j)}.
\end{equation*}
In the mapping torus \eqref{def:mapping-torus}, the two exchanged necks
have first-return time \(2T^{(j)}\) in the circle direction, the proof
of Lemma~\ref{lem:un-warp} gives \(f^{(j)}\ge\rho_*^{(j)}\), and
\eqref{eq:un-inner-boundary-metric} gives
\(T^{(j)}\rho_*^{(j)}=\pi\bar\ell_j\).  Integrating the chord bound over
the middle portion, using \eqref{eq:vg-neck-weighted-area}, and applying
the homothety \(c^{(j)}\), we obtain
\begin{equation}
 \operatorname{Vol}_M
 \left\{\frac{\ell^{(j)}}2\le s^{(j)}\le\frac{3\ell^{(j)}}4\right\}
 \ge C^{-1}\bar\ell_j(c^{(j)})^{m+1}\ell^{(j)}(w^{(j)})^{m-1}
 \ge C^{-1}\bar\ell_j\bigl(c^{(j)}\ell^{(j)}\bigr)^{m+1}(w^{(j)})^{m-1},
 \label{eq:vg-neck-local}
\end{equation}
where the set on the left is the image in \(W_m\) of both middle neck
portions, including the circle direction, and the last step uses
\(\ell^{(j)}<1\).

\emph{Step 6: finish the proof.}  By Section~\ref{sec:complete-construction},
the incoming curvature of \(B_j\) is less than half the least principal
curvature of the outgoing boundary of \(B_{j-1}\), which is at most
\(\mathfrak p_{1,j-1}/100\) by \eqref{eq:un-II-right}.  Together with
\eqref{eq:unified-base-3-75} and \eqref{eq:indexed-cylinder-parameters},
this gives, for \(j\ge1\),
\begin{equation*}
 \frac{\rho^{(j)}}{\bar\rho_j}
 <\frac{\mathfrak p_{1,j-1}}{200}
 <\frac1{400s_{j-1}},
\end{equation*}
so \eqref{eq:vg-neck-normalization} and \(t_1^{(j)}>2t_0^{(j)}\) imply
\(c^{(j)}\ell^{(j)}>400s_{j-1}\to\infty\).  Every point of the middle
portions can be reached from a point of \(\partial_iB_j\) by a radial
path in the neck, at fixed circle coordinate, of length at most a fixed
multiple of \(c^{(j)}\ell^{(j)}\).  By hypothesis, the middle portions
therefore lie in \(B_{r_j}(p)\) with \(r_j:=Cc^{(j)}\ell^{(j)}\).
Combining \eqref{eq:vg-neck-incoming-radius},
\eqref{eq:vg-neck-logarithmic}, and \eqref{eq:vg-neck-local} proves
\eqref{eq:vg-neck-ball} for the radius \(r_j\to\infty\), after
increasing \(C\).  Taking logarithms and dividing by \(\log r_j\) shows
that the upper logarithmic exponent of \(M\) is at least \(m+1\).
Finally, \(q\) maps \(B_r(\widetilde p)\) onto \(B_r(p)\) and is a local
isometry, so \(\operatorname{Vol}_{\widetilde M}B_r(\widetilde p)\ge
\operatorname{Vol}_MB_r(p)\), and the same conclusion holds for
\(\widetilde M\).
\end{proof}

The construction of Section~\ref{sec:complete-construction} leaves two
free parameters at each stage \(j\): the neck parameter \(t_0^{(j)}\) of
Lemma~\ref{lem:slow-logarithmic-neck}, which governs the length of the
necks in the mapping torus, and the pair
\((\Lambda_j,\mathfrak p_{0,j})\) of the holonomy interpolation
cylinder.  Once \(\Lambda_j\) is fixed, \(\Lambda_j\) prescribes the
expansion factor of the circle radius, while \(\mathfrak p_{0,j}\)
determines the length \(s_j\) of the cylinder.  Long necks produce balls
of almost \((m+1)\)-dimensional volume, and very long cylinders produce
balls whose volume grows only linearly.  In the following Proposition, we show that all growth exponents between these two extreme can occur along sequences of radius
tending to infinity, by choosing these parameters
inductively.

\begin{proposition}
\label{prop:volume-growth-flexibility}
The neck and cylinder parameters in
Section~\ref{sec:complete-construction} can be chosen so that the
resulting manifold \(M^n\) and its universal cover \(\widetilde M\)
satisfy
\begin{equation}
 \liminf_{r\to\infty}
 \frac{\log\operatorname{Vol}_M B_r(p)}{\log r}
 =\liminf_{r\to\infty}
 \frac{\log\operatorname{Vol}_{\widetilde M}B_r(\widetilde p)}{\log r}
 =1,
 \label{eq:base-growth-oscillation}
\end{equation}
\begin{equation}
 \limsup_{r\to\infty}
 \frac{\log\operatorname{Vol}_M B_r(p)}{\log r}
 =\limsup_{r\to\infty}
 \frac{\log\operatorname{Vol}_{\widetilde M}B_r(\widetilde p)}{\log r}
 =n-1,
 \label{eq:cover-growth-oscillation}
\end{equation}
where \(p\) lies in the initial cap and \(\widetilde p\) is a lift of
\(p\).
\end{proposition}

\begin{proof}
At each stage we first lengthen the necks of the mapping torus, which
creates balls of volume at least \(r^{m+1-o(1)}\) and pushes the upper
exponent up to \(m+1\), we then make the holonomy interpolation cylinder
so long that the ball reaching its middle has volume comparable to its
radius, which pushes the lower exponent down to \(1\).

\emph{Step 1: conventions.}  We use the notation of
Lemma~\ref{lem:vg-neck-growth}, the superscript \((j)\) referring to the
\(j\)-th stage, and write \(W_m\) for the mapping torus of \(B_j\).  Let
\(K_j\) be the compact region formed by the initial cap, the blocks
\(B_0,\ldots,B_{j-1}\), and the mapping torus of \(B_j\), thus
\(K_j\) ends at the incoming boundary of \(\mathcal C_j\).  Volumes and
diameters without a subscript refer to the piecewise metric obtained
from the cap, the mapping torus, and the holonomy interpolation
cylinders before their mutual interfaces are smoothed, diameters are
intrinsic.  We call \(\Lambda_j\) \emph{admissible} if it satisfies the
lower bound of Proposition~\ref{prop:holonomy-interpolation-cylinder}.
By Lemma~\ref{lem:relative-gluing-unified}\textup{(c)}, we choose
disjoint smoothing collars, avoiding the middle neck portions and the
slabs \(\mathcal P_j\), such that the final metric lies between
\(\tfrac12\) and \(2\) times the piecewise metric.  Consequently,
\begin{equation*}
 \operatorname{Vol}_M(K_j)\le2^{n/2}\operatorname{Vol}(K_j),
 \qquad
 \operatorname{diam}_M(K_j)\le\sqrt2\operatorname{diam}(K_j),
\end{equation*}
and these factors do not accumulate, because the collars are disjoint.
All constants \(C\) below are independent of \(j\).

\emph{Step 2: initialization.}  Fix the cap, the first mapping torus,
and an admissible \(\Lambda_0\), and let \(\mathfrak p_{0,0}\to 0\).
By \eqref{eq:indexed-cylinder-parameters} and
Lemma~\ref{lem:protected-cylinder-volume},
\begin{equation*}
 s_0\longrightarrow\infty,\qquad
 \operatorname{Vol}(K_0\cup\mathcal C_0)=O(s_0),\qquad
 \bar\ell_1\bar\rho_1^m=\Lambda_0\ell_0(0)\rho_0(0)^m,
\end{equation*}
and the last quantity does not depend on \(\mathfrak p_{0,0}\).  We may
therefore choose \(\mathfrak p_{0,0}\) so that the conditions
\eqref{eq:scale-separation} below hold for \(j=0\) and so that the
following inductive hypothesis holds for \(j=1\):
\begin{equation}
 \bar\ell_j\bar\rho_j^m\le s_{j-1}^m,\qquad
 2^{n/2}\operatorname{Vol}(K_{j-1}\cup\mathcal C_{j-1})
 \le s_{j-1}^2.
 \label{eq:vg-induction}
\end{equation}

\emph{Step 3: the necks at stage \(j\).}  Let \(j\ge1\), and assume
\eqref{eq:vg-induction}.  As in Step~6 of the proof of
Lemma~\ref{lem:vg-neck-growth},
\(\rho^{(j)}/\bar\rho_j<1/(400s_{j-1})\).  By \eqref{eq:gaussian-normal}
and Lemma~\ref{lem:neck-endpoint-normalization}, the \(m\)-th root of
the slice volume density of the neck in \(P_m\) is concave in arclength
and has initial logarithmic derivative \(\rho^{(j)}/\bar\rho_j\).
Bounding it by its initial tangent line, and using
\(1\le f/\rho_*\le2\) from Lemma~\ref{lem:un-warp} for the circle
factor, we obtain
\begin{equation}
 \operatorname{Vol}_M\{0\le c^{(j)}s^{(j)}\le x\}
 \le C\bar\ell_j\bar\rho_j^m
       \int_0^x\left(1+\frac{\rho^{(j)}}{\bar\rho_j}t\right)^m\dd t,
 \qquad 0\le x\le c^{(j)}\ell^{(j)}.
 \label{eq:vg-partial-neck}
\end{equation}
Here the set on the left consists of both exchanged neck portions,
including the circle direction and their smoothing collars.

We now lengthen the necks.  Fix the angular data, \(\sigma^{(j)}\), and
the normalized outer rounding path, the construction of
Lemma~\ref{lem:slow-logarithmic-neck} allows \(t_0^{(j)}\) to increase
with these data fixed.  By \eqref{eq:I t0 bound}, \eqref{eq:I t0 identity},
and \eqref{eq:unified-base-3-54},
\begin{equation*}
 \log\frac{t_1^{(j)}}{t_0^{(j)}}
 \ge\frac{\log t_0^{(j)}}{\sigma^{(j)}},\qquad
 v^{(j)}(t_1^{(j)})
 =\rho^{(j)}(u_*^{(j)})^{-1/\alpha^{(j)}}>w^{(j)},
\end{equation*}
and the second quantity does not depend on \(t_0^{(j)}\).  Hence, as
\(t_0^{(j)}\to\infty\), \eqref{eq:vg-neck-normalization} gives
\begin{equation*}
 \ell^{(j)}\longrightarrow
 \frac{w^{(j)}}{v^{(j)}(t_1^{(j)})}>0,
 \qquad c^{(j)}\ell^{(j)}\longrightarrow\infty.
\end{equation*}
The Ricci estimates of Lemma~\ref{lem:slow-logarithmic-neck} hold for
all sufficiently large \(t_0^{(j)}\), and the terminal normalized
principal curvatures tend to \(v^{(j)}(t_1^{(j)})/w^{(j)}>1\) (see
\eqref{eq:unified-base-3-56}--\eqref{eq:unified-base-3-63} and
Lemma~\ref{lem:neck-endpoint-normalization}), so this enlargement keeps
all the hypotheses of the construction.  The circle factor stays
bounded by a fixed multiple of \(\bar\ell_j\), while the remaining part
of the cobordism (the angular body and the rounding cylinder) is
rescaled by \(c^{(j)}\), so its volume and the area of its outgoing
boundary scale like \((c^{(j)})^{m+1}\) and \((c^{(j)})^m\).  Together
with \eqref{eq:vg-partial-neck}, this gives
\begin{equation*}
 \operatorname{Vol}(W_m)=O_j\bigl((c^{(j)}\ell^{(j)})^{m+1}\bigr),
 \qquad
 \ell_j(0)\rho_j(0)^m=O_j\bigl((c^{(j)}\ell^{(j)})^m\bigr),
\end{equation*}
where the constants depend on the stage but not on \(t_0^{(j)}\), the
warping parameter is chosen afterwards with \(1\le f/\rho_*\le2\).
Taking \(t_0^{(j)}\) large, we may therefore impose
\begin{equation}
 \begin{gathered}
 c^{(j)}\ell^{(j)}\ge\max\{s_{j-1},1\},\qquad
 2^{n/2}\operatorname{Vol}(W_m)
 \le(c^{(j)}\ell^{(j)})^{m+1+1/(j+1)},\\
 \ell_j(0)\rho_j(0)^m
 \le(c^{(j)}\ell^{(j)})^{m+1/(j+1)},
 \end{gathered}
 \label{eq:vg-body-bounds}
\end{equation}
the excess powers absorb all constants depending on the stage,
including the smoothing comparison.

\emph{Step 4: the cylinder at stage \(j\).}  The region \(K_j\) and the
radius \(\ell_j(0),\rho_j(0)\) are now fixed.  Fix an admissible
\(\Lambda_j\), and let \(\mathfrak p_{0,j}>0\) decrease, staying below
the gluing bound of Proposition~\ref{prop:ricci-positive-building-block}
and below \((c^{(j)}\ell^{(j)})^{-1}\).  Then
\(s_j=(\Lambda_j^2-1)/(2\mathfrak p_{0,j})\to\infty\), and by
Lemma~\ref{lem:protected-cylinder-volume},
\begin{equation*}
 C^{-1}\ell_j(0)\rho_j(0)^m\Lambda_js_j
 \le\operatorname{Vol}(\mathcal C_j)
 \le C\ell_j(0)\rho_j(0)^m\Lambda_js_j.
\end{equation*}
Thus the cylinder volume grows linearly in \(s_j\), while all quantities
attached to \(K_j\) remain fixed.  We may therefore require
simultaneously
\begin{equation}
 \begin{gathered}
 \ell_j(0)\rho_j(0)^m\Lambda_j\le s_j^m,\qquad
 2^{n/2}\operatorname{Vol}(K_j\cup\mathcal C_j)\le\tfrac12s_j^2,\\
 2^{n/2}\operatorname{Vol}(K_j)+1
 \le\ell_j(0)\rho_j(0)^m\Lambda_js_j,\qquad
 s_j\ge2^{j+2},\qquad\log s_j\ge(j+1)^2,\\
 \frac{\sqrt2\operatorname{diam}(K_j)
       +2\pi(\Lambda_j\ell_j(0)+\rho_j(0))}{s_j}
 \le\frac1{j+1}.
 \end{gathered}
 \label{eq:scale-separation}
\end{equation}
Since \(\bar\ell_{j+1}=\Lambda_j\ell_j(0)\) and
\(\bar\rho_{j+1}=\rho_j(0)\), the first two conditions give
\eqref{eq:vg-induction} for \(j+1\), which completes the inductive
choice.  All quantities in \eqref{eq:vg-induction}--\eqref{eq:scale-separation}
refer to the model metrics, which later smoothings do not change, and
the comparison factors of Step~1 transfer every bound to the final
metric.

\emph{Step 5: the lower exponent equals \(1\).}  For the midpoint slice
\(\Sigma_j\) of \(\mathcal C_j\) (see \eqref{eq:compact-exhaustion}),
set
\begin{equation*}
 d_j^-=\min_{\Sigma_j}d(p,\cdot),\qquad
 d_j^+=\max_{\Sigma_j}d(p,\cdot),\qquad r_j=d_j^++s_j/8,
\end{equation*}
where all distances refer to the final metric.  Paths through the
protected slab \(\mathcal C_j[1/4,1/2]\), paths through \(K_j\), and
paths within \(\Sigma_j\) give
\begin{equation*}
 \begin{gathered}
 d_j^-\ge s_j/4,\qquad
 d_j^+\le\sqrt2\operatorname{diam}(K_j)+Cs_j,\\
 d_j^+-d_j^-\le2\pi(\Lambda_j\ell_j(0)+\rho_j(0))=o(s_j),
 \qquad C^{-1}s_j\le r_j\le Cs_j,
 \end{gathered}
\end{equation*}
where we used \eqref{eq:scale-separation}.  Hence the slab
\(\mathcal C_j[1/2,5/8]\) lies in \(B_{r_j}(p)\), up to its boundary,
while reaching the slice \(\{s=3s_j/4\}\) requires length at least
\(d_j^-+s_j/4>r_j\) for large \(j\), in particular, \(B_{r_j}(p)\)
does not meet any later block.  Lemma~\ref{lem:protected-cylinder-volume}
and \eqref{eq:scale-separation} therefore give
\begin{equation}
 C^{-1}\ell_j(0)\rho_j(0)^m\Lambda_js_j
 \le\operatorname{Vol}_M B_{r_j}(p)
 \le C\ell_j(0)\rho_j(0)^m\Lambda_js_j.
 \label{eq:ball-cylinder-sandwich}
\end{equation}
The quantity \(\ell_j(0)\rho_j(0)^m\Lambda_j\) is fixed before
\(\mathfrak p_{0,j}\) is chosen, and \(\log r_j=\log s_j+O(1)\).
Decreasing \(\mathfrak p_{0,j}\) further, we may assume
\begin{equation*}
 \left|\frac{\log\operatorname{Vol}_M B_{r_j}(p)}{\log r_j}-1\right|
 \le\frac{C\bigl(1+|\log(\ell_j(0)\rho_j(0)^m\Lambda_j)|\bigr)}
            {\log s_j}
 \le\frac1j.
\end{equation*}
These bounds are uniform over all later smoothings, so this is a
legitimate choice at stage \(j\).  Since \(r_j\ge s_j/4\to\infty\), the
lower logarithmic exponent of \(M\) is at most \(1\).

Conversely, the volume of balls in a complete manifold with
nonnegative Ricci curvature grows at least linearly.  Indeed, let \(x\)
lie at distance \(t>2\) from \(p\) on a minimizing ray.  Then
\(B_1(p)\subset B_{t+1}(x)\setminus B_{t-1}(x)\) and
\(B_{t+1}(x)\subset B_{2t+1}(p)\), so the Bishop--Gromov inequality at
\(x\) gives
\begin{equation*}
 \operatorname{Vol}_M B_1(p)
 \le\left[1-\left(\frac{t-1}{t+1}\right)^n\right]
           \operatorname{Vol}_M B_{t+1}(x)
 \le\frac{2n}{t+1}\operatorname{Vol}_M B_{2t+1}(p).
\end{equation*}
With \(r=2t+1\), this reads
\(\operatorname{Vol}_M B_r(p)\ge\operatorname{Vol}_M B_1(p)\,r/(4n)\)
for \(r>5\).  Hence the lower logarithmic exponent of \(M\) equals \(1\).

\emph{Step 6: the upper exponent is at most \(m+1\).}  For large \(r\),
let \(B_j\) be the last block met by \(B_r(p)\).  To meet \(B_j\), the
ball must cross the protected slab of \(\mathcal C_{j-1}\), so
\(s_{j-1}\le Cr\), moreover, \(j\to\infty\) as \(r\to\infty\).  By
\eqref{eq:vg-induction}, the region \(K_{j-1}\cup\mathcal C_{j-1}\)
contributes volume at most \(s_{j-1}^2\le Cr^{m+1}\).  It remains to
estimate \(B_r(p)\cap W_m\) and \(B_r(p)\cap\mathcal C_j\), where
\(W_m\) is the mapping torus of \(B_j\).

If the ball does not reach the angular body of \(W_m\), then it meets
\(W_m\) only in the necks, at radial distance at most \(Cr\) from their
incoming ends.  By \eqref{eq:vg-partial-neck}, \eqref{eq:vg-induction},
and \(\rho^{(j)}/\bar\rho_j<1/(400s_{j-1})\),
\begin{equation*}
 \operatorname{Vol}_M(B_r(p)\cap W_m)
 \le C\bar\ell_j\bar\rho_j^m r
       \left(1+\frac{\rho^{(j)}}{\bar\rho_j}r\right)^m
 \le Cs_{j-1}^m r\left(1+\frac r{s_{j-1}}\right)^m
 \le Cr^{m+1}.
\end{equation*}
If the ball reaches the angular body, then it contains a path crossing
a whole neck, so \(c^{(j)}\ell^{(j)}\le Cr\), and
\eqref{eq:vg-body-bounds} bounds the volume of the whole mapping torus
by \(Cr^{m+1+1/(j+1)}\).  If the ball also reaches \(\mathcal C_j\), we
apply \eqref{eq:vg-partial-cylinder} with \(x=\min\{s_j,Cr\}\) and use
\eqref{eq:vg-body-bounds} together with
\(\mathfrak p_{0,j}\le(c^{(j)}\ell^{(j)})^{-1}\):
\begin{equation*}
 \operatorname{Vol}_M(B_r(p)\cap\mathcal C_j)
 \le C\bigl((c^{(j)}\ell^{(j)})^{m+1/(j+1)}r
       +(c^{(j)}\ell^{(j)})^{m-1/2+1/(j+1)}r^{3/2}\bigr)
 \le Cr^{m+1+1/(j+1)},
\end{equation*}
where the last step uses \(c^{(j)}\ell^{(j)}\le Cr\) and the positivity
of both exponents of \(c^{(j)}\ell^{(j)}\).  Since \(B_j\) is the last
block met by the ball, these cases exhaust \(B_r(p)\), and
\begin{equation}
 \frac{\log\operatorname{Vol}_M B_r(p)}{\log r}
 \le m+1+\frac1{j+1}+\frac{\log C}{\log r}.
 \label{eq:vg-all-radius-upper}
\end{equation}
Letting \(r\to\infty\), and hence \(j\to\infty\), shows that the upper
logarithmic exponent of \(M\) is at most \(m+1\).

\emph{Step 7: the upper exponent equals \(m+1\).}  By
\eqref{eq:scale-separation}, every point of the incoming boundary of
\(B_{j+1}\) can be joined to \(p\) by a path of length at most
\(C(\operatorname{diam}(K_j)+s_j)\le Cs_j\): cross \(K_j\), follow the
product direction of \(\mathcal C_j\), and finish within its outgoing
slice.  Thus the hypothesis of Lemma~\ref{lem:vg-neck-growth} holds,
and along radius \(r\to\infty\),
\begin{equation*}
 \frac{\log\operatorname{Vol}_M B_r(p)}{\log r}
 \ge m+1-\frac C{\sqrt{\log r}}-\frac C{\log r}.
\end{equation*}
Hence the upper logarithmic exponent of \(M\) equals \(m+1=n-1\).

\emph{Step 8: proof of the universal cover.}  Since \(\log s_j\ge(j+1)^2\) by
\eqref{eq:scale-separation}, Lemma~\ref{lem:vg-cover-comparison} applies
and shows that \(M\) and \(\widetilde M\) have the same lower and upper
logarithmic exponents.  This proves \eqref{eq:base-growth-oscillation}
and \eqref{eq:cover-growth-oscillation}.
\end{proof}

We can now complete the proof of the main result of this section.

\begin{proof}[Proof of Theorem~\ref{thm:further-properties}]
Part~\textup{(i)} is given by Theorem~\ref{thm:universal-cover-topology}.

For~\textup{(ii)}, since \(\pi_1(M)\cong\CPrufer\) is infinite, the
theorem of Li \cite{Li1986} and Anderson \cite{AndersonTopology1990}
gives \(\operatorname{AVR}(M)=0\).  When \(m=2\), the manifold \(M\) is
four-dimensional and \(\pi_1(M)\) is not finitely generated, so by
Huang--Huang \cite{HuangHuang2025} the universal cover \(\widetilde M\)
cannot have Euclidean volume growth, that is,
\(\operatorname{AVR}(\widetilde M)=0\).

For~\textup{(iii)}, choose the parameters as in
Proposition~\ref{prop:volume-growth-flexibility}.  Since the lower
exponent of \(\widetilde M\) equals \(1<n\), there are radius
\(r_i\to\infty\) with \(\operatorname{Vol}B_{r_i}(\widetilde p)\le
r_i^2\), so \(\operatorname{Vol}B_{r_i}(\widetilde p)/r_i^n\to0\), since
this quotient is nonincreasing in \(r\) by the Bishop--Gromov
inequality, \(\operatorname{AVR}(\widetilde M)=0\).  Since the upper
exponents of \(M\) and \(\widetilde M\) equal \(n-1=m+1\ge3\), there are
radius \(r_i\to\infty\) along which the volume of the corresponding ball
is at least \(r_i^2\), dividing by \(r_i\) gives the last assertion.
\end{proof}

\appendix
\section{The Ricci-positive $\mathbb Z_2$-equivariant cobordism in base dimensions three and four}\label{app:perelman-neck}
In this appendix we prove Proposition~\ref{prop:unified-neck}.  The
argument follows Perelman's neck construction
\cite[Assertion and Section~3]{Perelman1997}, with the quantitative
normalization and equivariance required in
Section~\ref{sec:ricci-positive-equivariant-cobordism} recorded
explicitly.  The proof has three steps.
Lemmas~\ref{lem:concave-warping}--\ref{lem:angular-latitude-form}
construct the positively curved angular spheres and determine their
boundary geometry, the warping profile of Lemma~\ref{lem:concave-warping}
is also used in Appendix~\ref{app:Zp-equivariant}.
Lemmas~\ref{lem:angular-slice-interpolation}
and~\ref{lem:variable-angular-ricci} introduce a path of slice metrics
and compute the curvature of the associated neck.  Finally,
Lemmas~\ref{lem:slow-logarithmic-neck}
and~\ref{lem:neck-endpoint-normalization} choose a sufficiently slow
interpolation, prove positivity of its Ricci tensor, and normalize the
two ends.

We begin with the warping profile \(\mathcal R\).  It is obtained by
solving the comparison equation
\(\mathcal R'/\mathcal R=\xi\cot t\) for a cutoff \(\xi\le1\), the
inequality \(\xi<1\) away from \(t=0\) produces a small waist, while
\(\xi\) close to $1$ keeps the angular spheres positively curved.

\begin{lemma}\label{lem:concave-warping}\label{lem:appb-profile}
Fix $0<r<T<\pi/8$.  For every sufficiently small $w>0$ there is a
smooth function $\mathcal R:[0,T]\to\mathbb R$ with
$\mathcal R(t)=\sin t$ near $t=0$, so that
\begin{equation}
 \mathcal R(0)=0,\qquad \mathcal R'(0)=1,\qquad
 \mathcal R^{(2j)}(0)=0\quad(j\ge1),
 \label{eq:appb-R-smooth}
\end{equation}
and such that, on $(0,T]$,
\begin{align}
 &\mathcal R>0,\qquad 0<\mathcal R'<1,\qquad \mathcal R''<0,\qquad
 -\frac{\mathcal R''}{\mathcal R}\ge1,
 \label{eq:unified-base-3-3}\\
 &\mathcal R\le\sin t,\qquad \mathcal R'\le\cos t,\qquad
 0<\frac{\mathcal R'}{\mathcal R}\tan t\le1,
 \label{eq:appb-R-compare}
\end{align}
and
\begin{equation}
 \cot r\,\mathcal R(r)=w.
 \label{eq:unified-base-3-5}
\end{equation}
Moreover, for $0\le t\le r$, the quantities
\begin{equation*}
 q_{\mathcal R}=-\frac{\mathcal R''}{\mathcal R},
 \qquad
 x_{\mathcal R}=\frac{\mathcal R'}{\mathcal R}\tan t
\end{equation*}
satisfy
\begin{equation}
 x_{\mathcal R}>\cos^2r,
 \qquad
 q_{\mathcal R}(\tan^2r-\tan^2t)
   +x_{\mathcal R}\sec^2t>1.
 \label{eq:profile-key}
\end{equation}
All of these statements except \eqref{eq:unified-base-3-5} remain true
with $\mathcal R$ replaced by
\begin{equation}
 \mathcal R_u(t)=(1-u)\sin t+u\mathcal R(t),
 \qquad 0\le u\le1.
 \label{eq:appb-Ru}
\end{equation}
\end{lemma}

\begin{proof}
Choose
\begin{equation}
 \max\{\cos^2r,\tan^2T\}<c<1.
 \label{eq:appb-c-choice}
\end{equation}
Let $\zeta:[0,\infty)\to[0,1]$ be smooth and nonincreasing, with
$\zeta=1$ on $[0,1]$ and $\zeta=0$ on $[2,\infty)$.  For
$0<\varepsilon<r/2$, put $\xi_\varepsilon(t)=c+(1-c)\zeta(t/\varepsilon)$
and define
\begin{equation}
 \mathcal R_\varepsilon(t)=\sin t\exp\left(
  \int_0^t(\xi_\varepsilon(s)-1)\cot s\dd s\right).
 \label{eq:appb-R-epsilon}
\end{equation}
The integrand vanishes near zero, so $\mathcal R_\varepsilon=\sin t$
there, which gives \eqref{eq:appb-R-smooth}.  Direct differentiation
gives
\begin{equation}
 \frac{\mathcal R_\varepsilon'}{\mathcal R_\varepsilon}\tan t
  =\xi_\varepsilon,
 \qquad
 -\frac{\mathcal R_\varepsilon''}{\mathcal R_\varepsilon}
  =\xi_\varepsilon+\xi_\varepsilon(1-\xi_\varepsilon)\cot^2t
    -\xi_\varepsilon'\cot t.
 \label{eq:appb-x-is-xi}
\end{equation}
Since $\xi_\varepsilon\le1$, \eqref{eq:appb-R-epsilon} gives
$\mathcal R_\varepsilon\le\sin t$, and then
$0<\mathcal R_\varepsilon'=\mathcal R_\varepsilon\xi_\varepsilon\cot t
\le\xi_\varepsilon\cos t\le\cos t<1$.  Furthermore,
\begin{equation*}
 -\frac{\mathcal R_\varepsilon''}{\mathcal R_\varepsilon}-1
 =(1-\xi_\varepsilon)(\xi_\varepsilon\cot^2t-1)
   -\xi_\varepsilon'\cot t\ge0,
\end{equation*}
because $\xi_\varepsilon'\le0$ and
$\xi_\varepsilon\ge c>\tan^2T\ge\tan^2t$.  This proves
\eqref{eq:unified-base-3-3} and \eqref{eq:appb-R-compare} for
$\mathcal R_\varepsilon$.

On $[2\varepsilon,r]$ one has $\xi_\varepsilon=c$, and therefore
\begin{equation*}
 \log\frac{\mathcal R_\varepsilon(r)}{\sin r}
 \le(c-1)\int_{2\varepsilon}^{r}\cot s\dd s
 \longrightarrow-\infty
 \qquad \text{ as } \varepsilon\downarrow0 .
\end{equation*}
Since $\cot r\,\mathcal R_\varepsilon(r)$ depends continuously on
$\varepsilon$ and tends to zero as $\varepsilon\downarrow0$, the
intermediate value theorem gives $\cot r\,\mathcal R_\varepsilon(r)=w$
for every sufficiently small $w>0$.  We fix such an $\varepsilon$ and
set $\mathcal R=\mathcal R_\varepsilon$, which gives
\eqref{eq:unified-base-3-5}.

By \eqref{eq:appb-c-choice} and \eqref{eq:appb-x-is-xi},
$x_{\mathcal R}\ge c>\cos^2r$.  Since $q_{\mathcal R}\ge1$ and $t\le r$,
\begin{equation*}
 q_{\mathcal R}(\tan^2r-\tan^2t)+x_{\mathcal R}\sec^2t
 >\tan^2r-\tan^2t+\cos^2r\sec^2t\ge1,
\end{equation*}
where the last inequality follows from the identity
\begin{equation*}
 \tan^2r-\tan^2t+\cos^2r\sec^2t-1
 =\frac{(\cos^2t-\cos^2r)\sin^2r}{\cos^2r\cos^2t}\ge0 .
\end{equation*}
Finally, at each point, $q_{\mathcal R_u}$ and $x_{\mathcal R_u}$ are
convex combinations, with the same weights $(1-u)\sin t/\mathcal R_u$
and $u\mathcal R/\mathcal R_u$, of the corresponding quantities for
$\sin t$ and for $\mathcal R$.  Since $q_{\sin t}=x_{\sin t}=1$, the
inequalities for $q$ and $x$ persist along \eqref{eq:appb-Ru}, the
remaining properties of $\mathcal R_u$ are immediate.
\end{proof}

For the rest of this appendix, we fix
\begin{equation*}
 0<r<r_{\mathrm{amb}}=\frac{\pi}{40},
\end{equation*}
and apply Lemma~\ref{lem:concave-warping} with $T=5r<\pi/8$.

\begin{lemma}
\label{lem:neck-ambient-curvature}
Let \(m\in\{2,3\}\), \(k=m-1\), \(0<r<r_{\mathrm{amb}}\), and let
\(\mathcal R\) be the function of Lemma~\ref{lem:concave-warping} with
\(T=5r\).  Then the metric \eqref{eq:unified-base-3-11} is smooth and
has positive sectional curvature on \(\{0\le t\le5r\}\), which is a
neighborhood of \(B_{4r}^{\angle}(o_0)\).
\end{lemma}

\begin{proof}
For the unscaled metric
\begin{equation*}
 \widehat g_{\mathrm{amb},m}
 =\dd t^2+\cos^2t\,\dd z^2+\mathcal R(t)^2g_{\mathbb S^k}
\end{equation*}
and orthonormal vertical vectors $V,W\in T\mathbb S^k$, the sectional
curvatures of the coordinate planes are
\begin{equation*}
 \sec\bigl(\partial_t,(\cos t)^{-1}\partial_z\bigr)=1,\qquad
 \sec(\partial_t,V)=-\frac{\mathcal R''}{\mathcal R},\qquad
 \sec\bigl((\cos t)^{-1}\partial_z,V\bigr)
 =\frac{\mathcal R'}{\mathcal R}\tan t,
\end{equation*}
\begin{equation}
 \sec(V,W)=\frac{1-(\mathcal R')^2}{\mathcal R^2},
 \label{eq:unified-base-3-18}
\end{equation}
where \eqref{eq:unified-base-3-18} occurs only when $m=3$.  By
\eqref{eq:unified-base-3-3} and \eqref{eq:appb-R-compare}, all of them
are positive.  Since $\mathcal R=\sin t$ near $t=0$, near the collapsed
axis the metric is the round metric
$\dd t^2+\cos^2t\,\dd z^2+\sin^2t\,g_{\mathbb S^k}$ in join coordinates,
and hence is smooth.  Since
$\varrho_0(t,z)=\arccos(\cos t\cos z)\ge t$, we have
$B_{4r}^{\angle}(o_0)\subset\{0\le t\le4r\}$.  Finally, the constant
factor \(\cot^2r\) divides all sectional curvatures by \(\cot^2r\),
without changing their signs.
\end{proof}

The next lemma gives the required sectional-curvature lower bound on
the angular spheres and computes their principal curvatures.
\begin{lemma}
\label{lem:angular-sphere-geometry}
Under the hypotheses of Lemma~\ref{lem:neck-ambient-curvature}, for
every \(z_0\in\mathbb R\), the hypersurface
\(S_r^{\angle}(o_{z_0})\) is smooth. The induced metric
\(g_{\mathrm{ang},m}\) satisfies
\begin{equation*}
                         \sec_{g_{\mathrm{ang},m}}>1 .
\end{equation*}
All principal curvatures for the normal pointing out of
\(B_r^{\angle}(o_{z_0})\) belong to \((0,1]\), and the principal curvature in the meridional direction, tangent to the curves in \(S_r^{\angle}(o_{z_0})\) along which the \(\mathbb S^k\)-coordinate is constant, equals \(1\).
\end{lemma}

\begin{proof}
Fix $z_0\in\mathbb R$ and consider $S_r^{\angle}(o_{z_0})$.
The projection onto the $(t,z)$-factor, with the round metric $\dd t^2+\cos^2t\,\dd z^2$, is length nonincreasing, and its
radial geodesics lift with constant $\mathbb S^k$-coordinate.  Hence, on
$|z-z_0|<\pi$, $\varrho_{z_0}$ is also the distance from $o_{z_0}$ for
$\widehat g_{\mathrm{amb},m}$.  Since the angular sphere $S_r^{\angle}(o_{z_0})$ is the level set
$\{\varrho_{z_0}=r\}$, its outward unit normal is
$\nu=\nabla\varrho_{z_0}$.  Differentiating the identity $\cos\varrho_{z_0}(t,z)=\cos t\cos(z-z_0)$ with respect to $t$, with $z$ fixed, gives
\begin{equation*}
 -\sin\varrho_{z_0}\,\partial_t\varrho_{z_0}
   =-\sin t\cos(z-z_0).
\end{equation*}
On $S_r^{\angle}(o_{z_0})$ we have $\varrho_{z_0}=r$ and
$\cos(z-z_0)=\cos r/\cos t$.  Because $\partial_t$ is a unit vector,
$\partial_t\varrho_{z_0}
=\langle\nabla\varrho_{z_0},\partial_t\rangle$, hence
\begin{equation}
\langle\nu,\partial_t\rangle=\cot r\tan t.                       \label{eq:unified-base-3-19}
\end{equation}
We now specify the tangent vectors used below.  Away from the two poles
of the angular sphere, let $E$ be a unit vector spanning
\begin{equation*}
 TS_r^{\angle}(o_{z_0})\cap\vspan\{\partial_t,\partial_z\},
\end{equation*}
and let $V,W$ be orthonormal vectors tangent to the $\mathbb S^k$
factor, with orthonormality measured using
$\widehat g_{\mathrm{amb},m}$.  Since $\varrho_{z_0}$ is independent of
the $\mathbb S^k$-coordinate, these vertical vectors are also tangent to
the angular sphere.  Thus the orthogonal splitting is
\begin{equation*}
 TS_r^{\angle}(o_{z_0})
   =\vspan\{E\}\oplus T\mathbb S^k.
\end{equation*}
In particular, $E,V,W$ are all unit tangent vectors to the same angular
sphere: $E$ is the meridional direction, whereas $V,W$ are vertical
directions.  The vector $W$ is needed only when $k=2$, equivalently
$m=3$.  The formulas at the two poles are understood by smooth
extension.
The Gauss equation gives the sectional curvatures of the metric
induced by $\widehat g_{\mathrm{amb},m}$:
\begin{align}
 \sec_{\widehat g_{\mathrm{amb},m}|_{S_r^{\angle}(o_{z_0})}}(E,V)
 &=-\frac{\mathcal R''}{\mathcal R}
       \bigl(1-\cot^2r\tan^2t\bigr)
   +\frac{\mathcal R'}{\mathcal R}\tan t\,\cot^2r\tan^2t
   +\cot^2r\frac{\mathcal R'}{\mathcal R}\tan t,                 \label{eq:unified-base-3-20}\\
 \sec_{\widehat g_{\mathrm{amb},m}|_{S_r^{\angle}(o_{z_0})}}(V,W)
 &=\mathcal R^{-2}\left(
      1-(\mathcal R')^2\bigl(1-\cot^2r\tan^2t\bigr)\right),       \label{eq:unified-base-3-21}
\end{align}
where \eqref{eq:unified-base-3-21} occurs only for $m=3$.  On the angular sphere $\cos t\ge\cos r$, so $0\le t\le r$.
Multiplying \eqref{eq:unified-base-3-20} by $\tan^2r$ gives
\begin{equation*}
 \tan^2r\,\sec_{\widehat g_{\mathrm{amb},m}|_{S_r^{\angle}(o_{z_0})}}(E,V)
 =q_{\mathcal R}(\tan^2r-\tan^2t)+x_{\mathcal R}\sec^2t>1
\end{equation*}
by \eqref{eq:profile-key}.  When $m=3$, the bounds
$\mathcal R\le\sin t$ and $\mathcal R'\le\cos t$ in
\eqref{eq:appb-R-compare} and \eqref{eq:unified-base-3-21} give
\begin{equation*}
 \sec_{\widehat g_{\mathrm{amb},m}|_{S_r^{\angle}(o_{z_0})}}(V,W)
 \ge\frac1{\sin^2t}-\cot^2t\bigl(1-\cot^2r\tan^2t\bigr)
 =1+\cot^2r,
\end{equation*}
so $\tan^2r\,\sec(V,W)\ge\sec^2r>1$.  Since the scaling by $\cot^2r$
multiplies sectional curvatures by $\tan^2r$, this proves
$\sec_{g_{\mathrm{ang},m}}>1$.

We next compute the principal curvatures.  Denote
\begin{equation*}
 e_z=(\cos t)^{-1}\partial_z.
\end{equation*}
Differentiating the defining equation of $\varrho_{z_0}$ with respect
to $z$, with $t$ fixed, and using \eqref{eq:unified-base-3-19} gives,
on $S_r^{\angle}(o_{z_0})$,
\begin{equation*}
 \nu=(\cot r\tan t)\partial_t
      +\frac{\sin(z-z_0)}{\sin r}\,e_z.
\end{equation*}
The squares of the two coefficients sum to one.  Thus, up to sign,
\begin{equation*}
 E=\frac{\sin(z-z_0)}{\sin r}\,\partial_t
      -(\cot r\tan t)e_z.
\end{equation*}
The metric $\dd t^2+\cos^2t\,\dd z^2$ has constant curvature one.  Its radial
distance function therefore satisfies
\begin{equation*}
 \Hess\varrho_{z_0}
   =\cot\varrho_{z_0}
      \bigl(\dd t^2+\cos^2t\,\dd z^2
            -\dd\varrho_{z_0}^{2}\bigr).
\end{equation*}
Since $\dd\varrho_{z_0}(E)=0$ and $|E|=1$, the second fundamental form in the direction $E$ is
\begin{equation*}
 \II_{\widehat g_{\mathrm{amb},m}}(E,E)
   =\langle\nabla_E\nu,E\rangle
   =\Hess\varrho_{z_0}(E,E)=\cot r.
\end{equation*}
For a unit vertical vector $V\in T\mathbb S^k$, the warped-product connection formula
and the fact that $\mathcal R$ depends only on $t$ give
\begin{equation*}
 \nabla_V\nu=\nu(\log\mathcal R)V,
 \qquad
 \II_{\widehat g_{\mathrm{amb},m}}(V,V)
   =\nu(\log\mathcal R)
   =\frac{\mathcal R'}{\mathcal R}
      \langle\nu,\partial_t\rangle
   =\frac{\mathcal R'}{\mathcal R}\cot r\tan t.
\end{equation*}
More generally, for vertical vectors $V,W$,
\begin{equation*}
 \II_{\widehat g_{\mathrm{amb},m}}(V,W)
   =\nu(\log\mathcal R)
      \widehat g_{\mathrm{amb},m}(V,W),
 \qquad
 \II_{\widehat g_{\mathrm{amb},m}}(E,V)=0.
\end{equation*}
Thus the splitting
$TS_r^{\angle}(o_{z_0})=\vspan\{E\}\oplus T\mathbb S^k$
diagonalizes the second fundamental form.  The principal curvature in
the direction $E$ is $\cot r$, and the vertical principal curvature,
of multiplicity $k=m-1$, is $(\mathcal R'/\mathcal R)\cot r\tan t$.

Finally, set $g_{\mathrm{amb},m}=\cot^2r\,
\widehat g_{\mathrm{amb},m}$.  Under a constant rescaling
$g\mapsto c^2g$, the Levi--Civita connection is unchanged, while unit
vectors are divided by $c$, consequently every principal curvature is
divided by $c$.  Here $c=\cot r$.  The principal curvatures with respect
to $g_{\mathrm{amb},m}$ are therefore
\begin{equation*}
\tan r\,\cot r=1,
 \qquad
 \tan r\,\frac{\mathcal R'}{\mathcal R}\cot r\tan t
   =x_{\mathcal R}.
\end{equation*}
By \eqref{eq:appb-R-compare}, the vertical value lies in $(0,1]$, at the
two collapsed poles it is understood by its positive smooth limit.
This proves the principal-curvature bounds and completes the proof.
\end{proof}

Next, we prove that the induced metric $g_{\mathrm{ang},m}$ on \(S_r^{\angle}(o_{z_0})\) has the form \eqref{eq:unified-base-3-12} in the latitude coordinate.
\begin{lemma}
\label{lem:angular-latitude-form}
Under the hypotheses of
Lemma~\ref{lem:angular-sphere-geometry},
\(S_r^{\angle}(o_{z_0})\cong\mathbb S^m\).  The relation
\begin{equation}\label{eq:latitude-form}
\cos\varphi=\frac{\mathcal R(t)}{\mathcal R(r)},\qquad
 \sgn\varphi=\sgn(z-z_0)
\end{equation}
defines a global latitude coordinate
\(\varphi\in[-\pi/2,\pi/2]\).  In this coordinate, the induced metric $g_{\mathrm{ang},m}$
has the form \eqref{eq:unified-base-3-12}, where
\begin{equation*}
 A_{\mathrm{ang}}\in
 C^\infty([-\pi/2,\pi/2],(0,\infty)),\qquad
 A_{\mathrm{ang}}(\pm\pi/2)=w,
\end{equation*}
\begin{equation*}
 \frac1\pi\int_{-\pi/2}^{\pi/2}
 A_{\mathrm{ang}}(\varphi)\,\dd\varphi=\cos r .
\end{equation*}
The function \(A_{\mathrm{ang}}\) is independent of \(z_0\), and at
both poles
\begin{equation}\label{eq:pole expansion}
\frac{A_{\mathrm{ang}}(\varphi)}w
      =1+O(\cos^2\varphi).
\end{equation}

\end{lemma}

\begin{proof}
 Since
 $\mathcal R'>0$ on $[0,r]$, the relation \eqref{eq:latitude-form}
 determines
 \begin{equation}\label{eq:def of t varphi}
    t(\varphi)=\mathcal R^{-1}
       \bigl(\mathcal R(r)\cos\varphi\bigr),
  \qquad -\frac\pi2\le\varphi\le\frac\pi2.
 \end{equation}
 The defining identity
 \begin{equation*}
  \cos r
  =\cos t(\varphi)\cos\bigl(z(\varphi)-z_0\bigr)
 \end{equation*}
 then determines its \(z\)-coordinate by
 \begin{equation}\label{eq:def of z varphi}
     z(0)=z_0,
  \qquad
  z(\varphi)=z_0+\sgn(\varphi)
       \arccos\!\left(\frac{\cos r}{\cos t(\varphi)}\right)
       \quad (\varphi\ne0).
 \end{equation}
 Thus $(\varphi,p)\mapsto(t(\varphi),z(\varphi),p)$ parametrizes
 $S_r^{\angle}(o_{z_0})$, with the $\mathbb S^k$ factor collapsed at
 $\varphi=\pm\pi/2$.  Pulling back the ambient metric
 \eqref{eq:unified-base-3-11} gives
 \begin{equation}\label{eq:def of A ang}
    A_{\mathrm{ang}}(\varphi)
  =\cot r\sqrt{t'(\varphi)^2
                +\cos^2t(\varphi)\,z'(\varphi)^2},
 \end{equation}
 and, by \eqref{eq:unified-base-3-5},
 \begin{equation*}
  \cot^2r\,\mathcal R(t(\varphi))^2g_{\mathbb S^k}
  =\cot^2r\,\mathcal R(r)^2\cos^2\varphi\,g_{\mathbb S^k}
  =w^2\cos^2\varphi\,g_{\mathbb S^k}.
 \end{equation*}
 This proves the metric formula \eqref{eq:unified-base-3-12} and also
 shows directly that $A_{\mathrm{ang}}$ is independent of $z_0$.

 Differentiating \eqref{eq:def of t varphi} and
 \eqref{eq:def of z varphi} gives, when
 $0<|\varphi|<\pi/2$,
 \begin{equation*}
  t'(\varphi)
    =-\frac{\mathcal R(r)\sin\varphi}
            {\mathcal R'(t(\varphi))},
  \qquad
  z'(\varphi)
    =-\tan t(\varphi)\cot\bigl(z(\varphi)-z_0\bigr)t'(\varphi).
 \end{equation*}
 Consequently,
 \begin{equation*}
  A_{\mathrm{ang}}(\varphi)
   =\cot r\,
     \frac{\mathcal R(r)|\sin\varphi|\cos t(\varphi)\sin r}
     {\mathcal R'(t(\varphi))
      \sqrt{\cos^2t(\varphi)-\cos^2r}}.
 \end{equation*}
 The apparent singularity at $\varphi=0$ is removable.  Indeed, \eqref{eq:def of t varphi} and \eqref{eq:def of z varphi} give
 \begin{equation*}
  t(\varphi)
   =r-\frac{\mathcal R(r)}{2\mathcal R'(r)}\varphi^2+O(\varphi^4),
  \qquad
  z(\varphi)-z_0
   =\sqrt{\frac{\mathcal R(r)\tan r}{\mathcal R'(r)}}\,\varphi
      +O(\varphi^3).
 \end{equation*}
 Hence $A_{\mathrm{ang}}$, defined by \eqref{eq:def of A ang}, is
 smooth and nondegenerate at $\varphi=0$.  At either pole,
 $\mathcal R(t)=t+O(t^3)$ gives
 $|t'(\varphi)|\longrightarrow\mathcal R(r)$, while the same defining
 identity gives $z'(\varphi)\longrightarrow0$.  Hence
 \begin{equation*}
  A_{\mathrm{ang}}(\pm\pi/2)=\cot r\,\mathcal R(r)=w.
 \end{equation*}
 Smoothness of the rotationally symmetric induced metric at either
 collapsed pole gives the even expansion
 \begin{equation*}
   \frac{A_{\mathrm{ang}}(\varphi)}w
        =1+O(\cos^2\varphi).
 \end{equation*}
 Finally, the meridional curve projects to a semicircle of radius
 $\sin r$ in the unit round two-sphere.  Its length is therefore
 $\pi\sin r$ before scaling and $\pi\cos r$ after scaling by
 $\cot^2r$.  Thus
\begin{equation}
\frac1\pi\int_{-\pi/2}^{\pi/2}A_{\mathrm{ang}}(\varphi)\,\dd\varphi=\cos r.    \label{eq:unified-base-3-26}
\end{equation}
\end{proof}
We next interpolate between $g_{\mathrm{ang},m}$ and a round metric
while keeping the required sectional-curvature lower bound.
\begin{lemma}\label{lem:angular-slice-interpolation}
Let \(m\in\{2,3\}\), \(k=m-1\), and suppose that
\begin{equation*}
 0<\rho<1,\qquad \pi\rho<r,\qquad
 w^{m-1}<\rho^m,
\end{equation*}
and \(g_{\mathrm{ang},m}\) is the angular spherical metric in
Lemmas~\ref{lem:angular-sphere-geometry}
and~\ref{lem:angular-latitude-form}.  Define \(u_*\), \(\eta(\varphi)\),
\(\widetilde g_{u,v}\), $\alpha_0$, and \(v_0(u)\) by
\begin{equation*}
u_*=\max_\varphi\frac{A_{\mathrm{ang}}(\varphi)}w,
 \qquad
 \eta(\varphi)
   =\frac{A_{\mathrm{ang}}(\varphi)/w-1}{u_*-1}.                  
\end{equation*}
\begin{equation*}
\widetilde g_{u,v}
  =v^2\left(
       \bigl(1+(u-1)\eta(\varphi)\bigr)^2\dd\varphi^2
       +\cos^2\varphi\,g_{\mathbb S^k}\right),               
\end{equation*}
for $1\le u\le u_*$ and $v>0$, and
\begin{equation*}
\alpha_0=\frac{\log u_*}{\log\rho-\log w},
 \qquad
 v_0(u)=\rho u^{-1/\alpha_0},\qquad 1\le u\le u_*.              
\end{equation*}

Then
\begin{equation*}
 u_*>\frac{\rho}{w}>1,\qquad u_*w<1,\qquad
 1<\alpha_0<m,
\end{equation*}
\begin{equation*}
 \widetilde g_{1,v_0(1)}=\rho^2g_{\mathbb S^m},\qquad
 \widetilde g_{u_*,v_0(u_*)}=g_{\mathrm{ang},m},
\end{equation*}
and
\begin{equation*}
 \sec_{\widetilde g_{u,v_0(u)}}>1
 \qquad \text{ for } 1\leq u\leq u_*.
\end{equation*}
\end{lemma}

\begin{proof}
By \eqref{eq:unified-base-3-26},
$u_*\ge \cos r/w>\rho/w>1$, here $\rho<r/\pi<1/10<\cos r$, and
$w<\rho$ because $w^{m-1}<\rho^m$ and $\rho<1$.  Moreover,
$A_{\mathrm{ang}}>0$ implies
$-1/(u_*-1)<\eta\le1$.  Consequently,
$1+(u-1)\eta(\varphi)>0$ whenever $1\le u\le u_*$.
The pole expansion \eqref{eq:pole expansion} in
Lemma~\ref{lem:angular-latitude-form} gives
\begin{equation}\label{eq:expansion of eta}
 \eta=O(\cos^2\varphi),\qquad
 \eta'=O(\cos\varphi),\qquad
 \eta'\tan\varphi=O(1),\qquad
 \eta\tan\varphi=O(\cos\varphi).
\end{equation}

By construction, $\widetilde g_{1,v}=v^2g_{\mathbb S^m}$, whereas
$\widetilde g_{u_*,w}=g_{\mathrm{ang},m}$.  Denote
\begin{equation*}
 f_u(\varphi)=1+(u-1)\eta(\varphi).
\end{equation*}
The preceding pole expansions \eqref{eq:pole expansion} and \eqref{eq:expansion of eta} show that
\(\widetilde g_{u,v}\) is a smooth metric on \(\mathbb S^m\) for every \((u,v)\in[1,u_*]\times(0,\infty)\), and that it depends smoothly on \((u,v)\).
Let $E=(vf_u)^{-1}\partial_\varphi$, and let $V,W$ be orthonormal
vectors tangent to the $\mathbb S^k$ factor.  The warped-product
curvature formulas give
\begin{align}
 \sec_{\widetilde g_{u,v}}(E,V)
 &=\frac1{v^2}\left(
      \frac1{f_u^2}
      -\frac{(u-1)\eta'\tan\varphi}{f_u^3}\right),                \label{eq:unified-base-3-29}\\
 \sec_{\widetilde g_{u,v}}(V,W)
 &=\frac1{v^2}\left(
      \frac1{\cos^2\varphi}
      -\frac{\tan^2\varphi}{f_u^2}\right),                       \label{eq:unified-base-3-30}
\end{align}
where \eqref{eq:unified-base-3-30} occurs only for $m=3$.  At the
poles, the expressions \eqref{eq:unified-base-3-29} and \eqref{eq:unified-base-3-30} are understood by their smooth limits.

If $\eta(\varphi)=1$, then
$A_{\mathrm{ang}}(\varphi)$ is maximal.  Since
$A_{\mathrm{ang}}(\pm\pi/2)=w<u_*w$, such a point is interior and
$\eta'(\varphi)=0$.  Formula \eqref{eq:unified-base-3-29} at
$(u,v)=(u_*,w)$ gives the curvature $\sec_{\widetilde g_{u,v}}(E,V)=(u_*w)^{-2}$, which is greater than one by Lemma~\ref{lem:angular-sphere-geometry}.  Hence
\begin{equation}
u_*w<1.                              \label{eq:unified-base-3-31}
\end{equation}
The inequality $u_*>\rho/w$ gives $\alpha_0>1$, while
\eqref{eq:unified-base-3-10} and \eqref{eq:unified-base-3-31} give
\begin{equation*}
 \log u_*<-\log w<m(\log\rho-\log w).
\end{equation*}
Thus
\begin{equation*}
1<\alpha_0<m,                        
\end{equation*}
and $v_0(1)=\rho$, $v_0(u_*)=w$.  We claim that every sectional
curvature of $\widetilde g_{u,v_0(u)}$ is greater than one.

First observe that the numerator
\begin{equation*}
 1+(u-1)\bigl(\eta-\eta'\tan\varphi\bigr)
\end{equation*}
in \eqref{eq:unified-base-3-29} is positive at $u=1$ and at $u=u_*$:
at the first endpoint it equals one, and at the second its positivity
follows from $\sec_{g_{\mathrm{ang},m}}>1$.  Since it is affine in
$u$, it is positive throughout $[1,u_*]$.  When $m=3$, the numerator
$f_u^2-\sin^2\varphi$ in \eqref{eq:unified-base-3-30} is also
positive.  If $\eta<0$, it decreases with $u$ and is positive at
$u=u_*$, if $\eta\ge0$, then $f_u\ge1$.

Differentiating \eqref{eq:unified-base-3-29} and \eqref{eq:unified-base-3-30} with respect to $u$ gives
\begin{align}
 \frac{\dd}{\dd u}\log\sec_{\widetilde g_{u,v_0(u)}}(E,V)
 &=\frac1{\alpha_0u}\left(
  2+\frac{\alpha_0u(\eta-\eta'\tan\varphi)}
           {1+(u-1)(\eta-\eta'\tan\varphi)}
   -\frac{3\alpha_0u\eta}{f_u}\right),                           \label{eq:unified-base-3-34}\\
 \frac{\dd}{\dd u}\log\sec_{\widetilde g_{u,v_0(u)}}(V,W)
 &=\frac2{\alpha_0u}\left(
  1+\frac{\alpha_0u\eta\sin^2\varphi}
          {f_u(f_u^2-\sin^2\varphi)}\right).                     \label{eq:unified-base-3-35}
\end{align}
For \eqref{eq:unified-base-3-35}, if $\eta\ge0$ the derivative is
nonnegative.  If $\eta<0$, its bracket is decreasing because
\begin{equation*}
 \frac{\dd}{\dd u}\frac{u}{f_u(f_u^2-\sin^2\varphi)}>0.
\end{equation*}
Thus its sign can change only from positive to negative, so the
curvature $\sec_{\widetilde g_{u,v_0(u)}}(V,W)$ has no interior minimum.

For \eqref{eq:unified-base-3-34}, if
$\eta'\tan\varphi<0$, then $f_u\le u$ and
\begin{equation*}
\sec_{\widetilde g_{u,v_0(u)}}(E,V)
 >\frac1{v_0(u)^2f_u^2}
 \ge\frac1{u^2v_0(u)^2}>1.                                      
\end{equation*}
Indeed, $uv_0(u)$ is increasing because $\alpha_0>1$, and its terminal
value is $u_*w<1$.  It remains to consider the case when
$\eta'\tan\varphi\ge0$.  If $\eta\ge0$, the right side of
\eqref{eq:unified-base-3-34} equals
\begin{equation*}
\frac1{\alpha_0uf_u}\left(
 2(1-\eta)+2\eta u(1-\alpha_0)
  -\frac{\alpha_0u\eta'\tan\varphi}
         {1+(u-1)(\eta-\eta'\tan\varphi)}\right).                
\end{equation*}
The bracket is decreasing because
\begin{equation*}
 \frac{\dd}{\dd u}\left(
   \frac u{1+(u-1)(\eta-\eta'\tan\varphi)}\right)
 =\frac{1-\eta+\eta'\tan\varphi}
        {[1+(u-1)(\eta-\eta'\tan\varphi)]^2}\ge0.
\end{equation*}
If $\eta<0$, \eqref{eq:unified-base-3-34} becomes
\begin{equation*}
\frac1{\alpha_0u}\left(
 2-\frac{\alpha_0u}{f_u}\left(
  2\eta+\frac{\eta'\tan\varphi}
    {1+(u-1)(\eta-\eta'\tan\varphi)}\right)\right).             
\end{equation*}
At every zero of its bracket the last parenthesis is positive.  Also,
\begin{equation*}
 \left(\frac u{f_u}\right)'=\frac{1-\eta}{f_u^2}>0
\end{equation*}
and
\begin{equation*}
 \frac{\dd}{\dd u}\left(
   \frac{\eta'\tan\varphi}
        {1+(u-1)(\eta-\eta'\tan\varphi)}\right)
 =\frac{\eta'\tan\varphi\,(\eta'\tan\varphi-\eta)}
        {[1+(u-1)(\eta-\eta'\tan\varphi)]^2}\ge0.
\end{equation*}
Hence this bracket can cross zero only from positive to negative, so
the curvature $\sec_{\widetilde g_{u,v_0(u)}}(E,V)$ also has no interior minimum.  At the endpoints, both curvatures are greater than one: at $u=1$ the metric is the round
sphere of radius $\rho<1$, and at $u=u_*$ it is
$g_{\mathrm{ang},m}$.  This proves the claim.
\end{proof}

We next define the neck metric on $I\times \mathbb S^m$ and calculate its curvature.
\begin{lemma}\label{lem:variable-angular-ricci}
Let \(u_*\), \(\eta\), \(f_u=1+(u-1)\eta\), and
\(\widetilde g_{u,v}\) be as in
Lemma~\ref{lem:angular-slice-interpolation}.  Let
\(I\subset(0,\infty)\) be an interval, and let the functions
\begin{equation*}
 u:I\longrightarrow[1,u_*],\qquad v:I\longrightarrow(0,\infty)
\end{equation*}
be smooth.  On \(I\times\mathbb S^m\), define the neck metric
\begin{equation}
\widehat g_{\mathrm{neck},m}=\dd t^2+t^2v(t)^2\left(
 f_{u(t)}(\varphi)^2\dd\varphi^2
 +\cos^2\varphi\,g_{\mathbb S^k}\right).                         \label{eq:unified-base-3-39}
\end{equation}
Then the sectional and Ricci curvatures of \(\widehat g_{\mathrm{neck},m}\) are given by the identities \eqref{eq:unified-base-3-40}--\eqref{eq:unified-base-3-47} below.
All off-diagonal Ricci components vanish except possibly the
\((\partial_t,E)\)-component in
\eqref{eq:unified-base-3-45}.
\end{lemma}

\begin{proof}
Let $E=(tvf_u)^{-1}\partial_\varphi$, and let $V,W$ be orthonormal
vectors tangent to the $\mathbb S^k$ factor.  Here and throughout this
calculation, we write $u=u(t)$, $v=v(t)$, and
$f_u=f_{u(t)}(\varphi)$ for brevity.
The principal curvatures of the $t$-slices with respect to
$\partial_t$ are
\begin{equation}
\II_t(E,E)=\frac1t+\frac{v'}v+\frac{u'\eta}{f_u},
 \qquad
 \II_t(V,V)=\frac1t+\frac{v'}v.                                 \label{eq:unified-base-3-40}
\end{equation}
Since the metric induced on the $t$-slice is
$t^2\widetilde g_{u(t),v(t)}$, the Gauss equation then gives
\begin{align}
 \sec_{\widehat g_{\mathrm{neck},m}}(E,V)
   &=\frac1{t^2}\sec_{\widetilde g_{u,v}}(E,V)
           -\II_t(E,E)\II_t(V,V),                                      \label{eq:unified-base-3-41}\\
 \sec_{\widehat g_{\mathrm{neck},m}}(V,W)
   &=\frac1{t^2}\sec_{\widetilde g_{u,v}}(V,W)
           -\II_t(V,V)^2.                                               \label{eq:unified-base-3-42}
\end{align}
The second formula \eqref{eq:unified-base-3-42} occurs only for $m=3$.  The normal variation of the
slice metric gives
\begin{align}
 \sec_{\widehat g_{\mathrm{neck},m}}(\partial_t,V)
   &=-\frac{(tv)''}{tv},                                                 \label{eq:unified-base-3-43}\\
 \sec_{\widehat g_{\mathrm{neck},m}}(\partial_t,E)
   &=-\frac{\partial_t^2(tv f_u)}{tv f_u}\notag\\
   &=-\frac{(tv)''}{tv}
      -2\left(\frac1t+\frac{v'}v\right)\frac{u'\eta}{f_u}
      -\frac{u''\eta}{f_u}.                                           \label{eq:unified-base-3-44}
\end{align}
The only possibly nonzero mixed component of the Ricci tensor is
\begin{equation}
\Ric_{\widehat g_{\mathrm{neck},m}}(\partial_t,E)
 =-\frac{k\eta(\varphi)u'(t)\tan\varphi}
         {tv(t)f_{u(t)}(\varphi)^2}.                              \label{eq:unified-base-3-45}
\end{equation}
Smoothness at the poles gives
$\eta(\varphi)\tan\varphi=O(\cos\varphi)$, so this expression
extends smoothly to $\varphi=\pm\pi/2$.  Tracing the sectional
curvatures yields
\begin{equation}
\begin{aligned}
 \Ric_{\widehat g_{\mathrm{neck},m}}(\partial_t,\partial_t)
   &=\sec_{\widehat g_{\mathrm{neck},m}}(\partial_t,E)
     +k\sec_{\widehat g_{\mathrm{neck},m}}(\partial_t,V),\\
 \Ric_{\widehat g_{\mathrm{neck},m}}(E,E)
   &=\sec_{\widehat g_{\mathrm{neck},m}}(\partial_t,E)
     +k\sec_{\widehat g_{\mathrm{neck},m}}(E,V),\\
 \Ric_{\widehat g_{\mathrm{neck},m}}(V,V)
   &=\sec_{\widehat g_{\mathrm{neck},m}}(\partial_t,V)
     +\sec_{\widehat g_{\mathrm{neck},m}}(E,V)
     +(k-1)\sec_{\widehat g_{\mathrm{neck},m}}(V,W).
 \end{aligned}                                                     \label{eq:unified-base-3-46}
\end{equation}
The last identity is absent when $k=1$, and all off-diagonal components
other than \eqref{eq:unified-base-3-45} vanish.  In particular,
\eqref{eq:unified-base-3-43}--\eqref{eq:unified-base-3-44} imply
\begin{equation}
\Ric_{\widehat g_{\mathrm{neck},m}}(\partial_t,\partial_t)
 =-m\frac{(tv)''}{tv}
   -\frac{\eta}{f_u}\left(
       u''+2\left(\frac1t+\frac{v'}v\right)u'\right).           \label{eq:unified-base-3-47}
\end{equation}
\end{proof}
We then construct a family of neck metrics $\widehat g_{\mathrm{neck},m}$ with positive Ricci curvature.
\begin{lemma}
\label{lem:slow-logarithmic-neck}
Under the hypotheses of
Lemma~\ref{lem:angular-slice-interpolation}, for every sufficiently small \(\sigma>0\) there are numbers \(t_0<t_1(\sigma)\) and smooth functions
\begin{equation*}
 u_\sigma:[t_0,t_1(\sigma)]\longrightarrow[1,u_*],\qquad
 v_\sigma:[t_0,t_1(\sigma)]\longrightarrow(0,\infty)
\end{equation*}
such that
\begin{equation*}
 u_\sigma(t_0)=1,\quad v_\sigma(t_0)=\rho,\qquad
 u_\sigma(t_1)=u_*,\quad v_\sigma(t_1)>w,
\end{equation*}
\begin{equation*}
 t_1(\sigma)\longrightarrow\infty\quad \text{ as }\sigma\downarrow0,
\qquad
\Ric_{\widehat g_{\mathrm{neck},m}}>0 ,
\end{equation*}
where \(\widehat g_{\mathrm{neck},m}\) is defined by
\eqref{eq:unified-base-3-39} with
\(u=u_\sigma\) and \(v=v_\sigma\).
\end{lemma}

\begin{proof}
We write \(u=u_\sigma\) and \(v=v_\sigma\).  We choose the parameters and the functions $u$, $v$ so that the path of slice metrics stays positively curved and the leading coefficient of the radial Ricci curvature stays positive.
By compactness, there exists $\mu>0$ such that all sectional curvatures of
$\widetilde g_{u,v_0(u)}$, $1\le u\le u_*$, are at least
$1+2\mu$.  By smooth dependence on the parameters and compactness, there exists \(\delta_\alpha>0\) such that
\begin{equation*}
 \sec_{\widetilde g_{u,\rho u^{-1/a}}}\geq1+\mu
 \quad\text{for}\quad
 1\leq u\leq u_*,\quad
 \alpha_0\leq a\leq\alpha_0+\delta_\alpha .
\end{equation*}
Choose $0<\varepsilon_0<1$ so that
\begin{equation*}
 \frac{\alpha_0}{1-\varepsilon_0}
 <\min\{m,\alpha_0+\delta_\alpha\}.
\end{equation*}
Then choose $\sigma_0>0$ so that, for every
$0<\sigma\le\sigma_0$,
\begin{equation}
\alpha=\frac{1+\sigma}{1-\varepsilon_0}\alpha_0
 <\min\{m,\alpha_0+\delta_\alpha\},
 \qquad
 v_\alpha(u)=\rho u^{-1/\alpha},\qquad 1\le u\le u_*,           \label{eq:unified-base-3-48}
\end{equation}
and the sectional curvatures of
$\widetilde g_{u,v_\alpha(u)}$ are at least $1+\mu$.  The parameters chosen so far satisfy
\begin{equation*}
1<u_*<\frac1w,\qquad
 1<\alpha_0<\alpha<m,\qquad
 0<\varepsilon_0<1,\qquad 0<\sigma\le\sigma_0,\qquad \mu>0.      
\end{equation*}

Fix a cut-off function
\begin{equation*}
 \chi\in C^\infty([0,\infty),[0,1]),\qquad
 \chi(0)=0,\quad \chi'(0)>0,\quad \chi'\geq0,
\end{equation*}
\begin{equation*}
 \chi(x)>0\ \text{ for } x>0,\qquad \chi(x)=1\ \text{ for } x\geq1.
\end{equation*}
Choose \(t_0\) large, denote \(L=\log(2t_0)\), and define
\begin{equation}
\Xi(t)=\frac{L}{t(\log t)^2}
 \chi\!\left(\frac{t-t_0}{t_0}\right),
 \qquad t\geq t_0.                                                \label{eq:unified-base-3-50}
\end{equation}
Set \(q(t)=L/[t(\log t)^2]\), then
\begin{equation*}
 \Xi'=q'\chi+\frac q{t_0}\chi',\qquad
 \Xi'+\frac{2\Xi}{t}
 =q\left[
 \frac{\chi'}{t_0}
 +\chi\left(\frac1t-\frac2{t\log t}\right)\right].
\end{equation*}
The continuous function
\begin{equation*}
 x\longmapsto (1+x)\chi'(x)+\frac12\chi(x)
\end{equation*}
has a positive minimum on \([0,1]\).  Since \(\chi=1\) on
\([1,\infty)\), there are constants \(c,C>0\), independent of \(t_0\), such that
\begin{align}
 0\le\Xi&\le\frac{CL}{t(\log t)^2}, \qquad
 |\Xi'|\le\frac{CL}{t^2(\log t)^2}.                           \label{eq:unified-base-3-51}
\end{align}
\begin{align*}
\Xi'+\frac{2\Xi}{t}
 &\ge\frac{cL}{t^2(\log t)^2}.                                   
\end{align*}
Moreover, since \(\chi=1\) for \(t\geq2t_0\),
\begin{equation}\label{eq:I t0 bound}
 1=\int_{2t_0}^{\infty}\frac{L}{t(\log t)^2}\,\dd t
 \leq I_{t_0}=\int_{t_0}^{\infty}\Xi(t)\,\dd t
 \leq\int_{t_0}^{\infty}\frac{L}{t(\log t)^2}\,\dd t
 =\frac{L}{\log t_0}\leq2
\end{equation}
after increasing \(t_0\).
Define
\begin{equation}
\beta=\frac{(1-\varepsilon_0)(\log\rho-\log w)}
              {I_{t_0}},                                         \label{eq:unified-base-3-53}
\end{equation}
and solve
\begin{equation}
\frac{u'}u=\alpha\beta\Xi,\qquad
 \frac{v'}v=-\beta\Xi,\qquad u(t_0)=1,\quad v(t_0)=\rho.    \label{eq:unified-base-3-54}
\end{equation}
The two differential equations force $v=\rho u^{-1/\alpha}$, so the
slice metrics follow the path fixed in \eqref{eq:unified-base-3-48}.
The definitions of \(\alpha\) and \(\beta\) give
\begin{equation*}
 \alpha\beta I_{t_0}=(1+\sigma)\log u_*.
\end{equation*}
Since \(\Xi>0\) on \((t_0,\infty)\), there is a unique \(t_1\)
such that \(u(t_1)=u_*\), and
\begin{equation}\label{eq:I t0 identity}
 \int_{t_1}^{\infty}\Xi(t)\,\dd t
 =\frac{\sigma}{1+\sigma}I_{t_0}.
\end{equation}
For all sufficiently small \(\sigma\), the right-hand side of \eqref{eq:I t0 identity} is less than \(1\), whereas
\(\int_{2t_0}^{\infty}\Xi\,\dd t=1\), hence \(t_1>2t_0\).
Equation \eqref{eq:unified-base-3-50} then yields
\begin{equation*}
 \frac{L}{\log t_1}
 =\frac{\sigma}{1+\sigma}I_{t_0},
\end{equation*}
and therefore
\begin{equation*}
v(t_1)=\rho u_*^{-1/\alpha}>w,\qquad
                         t_1\longrightarrow\infty
                         \quad(\sigma\downarrow0).                
\end{equation*}
We now restrict the metric \eqref{eq:unified-base-3-39} to
$[t_0,t_1]\times\mathbb S^m$.

The bounds \(1\leq I_{t_0}\leq2\) from \eqref{eq:I t0 bound} show that
\begin{equation*}
 0<c_\beta\leq\beta\leq C_\beta<\infty .
\end{equation*}
Here \(c,C,c_i,C_i,c_\beta,C_\beta\) may depend on the fixed parameters
\begin{equation*}
 (m,r,\rho,w,\varepsilon_0,\chi),
\end{equation*}
but are independent of sufficiently large \(t_0\), of
\(0<\sigma\leq\sigma_0\), and of \(t\in[t_0,t_1]\).
Equations \eqref{eq:unified-base-3-51}--\eqref{eq:unified-base-3-53} give
\begin{equation}
\beta\Xi\le\frac{CL}{t(\log t)^2},\qquad
 \beta\left|\Xi'+\frac{2\Xi}{t}\right|
 \le\frac{CL}{t^2(\log t)^2},\qquad
 \beta^2\Xi^2\le\frac{CL^2}{t^2(\log t)^4},                \label{eq:unified-base-3-56}
\end{equation}
and
\begin{equation*}
\beta\left(\Xi'+\frac{2\Xi}{t}\right)
 \ge\frac{cL}{t^2(\log t)^2}.                                   
\end{equation*}
Substituting \eqref{eq:unified-base-3-54} into
\eqref{eq:unified-base-3-40} gives
\begin{equation}
\II_t(E,E)=\frac1t-\beta\Xi
       +\frac{\alpha\beta\Xi u\eta}{f_u},
 \qquad
 \II_t(V,V)=\frac1t-\beta\Xi.                              \label{eq:unified-base-3-58}
\end{equation}
It follows from \eqref{eq:unified-base-3-56} and the boundedness of
$u\eta/f_u$ that every product of two slice principal curvatures is
$t^{-2}+O(L/[t^2(\log t)^2])$.  The sectional curvatures of
$\widetilde g_{u(t),v(t)}$ are at least $1+\mu$, hence after increasing $t_0$, the Gauss
equations \eqref{eq:unified-base-3-41}--\eqref{eq:unified-base-3-42}
give
\begin{equation}
\sec_{\widehat g_{\mathrm{neck},m}}(E,V)\ge\frac{c_0}{t^2},\qquad
 \sec_{\widehat g_{\mathrm{neck},m}}(V,W)\ge\frac{c_0}{t^2}\quad(m=3),      \label{eq:unified-base-3-59}
\end{equation}
whereas
\begin{equation}
|\sec_{\widehat g_{\mathrm{neck},m}}(\partial_t,E)|
 +|\sec_{\widehat g_{\mathrm{neck},m}}(\partial_t,V)|
 +|\Ric_{\widehat g_{\mathrm{neck},m}}(\partial_t,E)|
 \le\frac{C_0L}{t^2(\log t)^2}.                                 \label{eq:unified-base-3-60}
\end{equation}
Here the last estimate uses the boundedness of
$\eta\tan\varphi/f_u^2$.  Then the identities \eqref{eq:unified-base-3-46}, together with
\eqref{eq:unified-base-3-59}--\eqref{eq:unified-base-3-60}, imply that
$\Ric_{\widehat g_{\mathrm{neck},m}}(E,E)$ and
$\Ric_{\widehat g_{\mathrm{neck},m}}(V,V)$ are at least
$c_1/t^2$.

Finally, substituting \eqref{eq:unified-base-3-54} into \eqref{eq:unified-base-3-47} gives
\begin{equation}
\begin{aligned}
 \Ric_{\widehat g_{\mathrm{neck},m}}(\partial_t,\partial_t)
 &=\beta\left(m-\alpha\frac{u\eta}{f_u}\right)
          \left(\Xi'+\frac{2\Xi}{t}\right)\\
 &\quad-\beta^2\Xi^2
        \left(m+\alpha(\alpha-2)\frac{u\eta}{f_u}\right).
\end{aligned}                                                     \label{eq:unified-base-3-61}
\end{equation}
If $\eta<0$, then $u\eta/f_u<0$, if $\eta\ge0$, then
$0\le u\eta/f_u\le1$.  Since $\alpha<m$, the first coefficient in
\eqref{eq:unified-base-3-61} has a positive uniform lower bound, while the second
is uniformly bounded.  Consequently,
\begin{equation*}
\Ric_{\widehat g_{\mathrm{neck},m}}(\partial_t,\partial_t)
 \ge\frac{c_3L}{t^2(\log t)^2}
     -\frac{C_3L^2}{t^2(\log t)^4}>0.                             
\end{equation*}
Moreover,
\begin{equation}
\Ric_{\widehat g_{\mathrm{neck},m}}(\partial_t,\partial_t)
   \Ric_{\widehat g_{\mathrm{neck},m}}(E,E)
 -\Ric_{\widehat g_{\mathrm{neck},m}}(\partial_t,E)^2
 \ge\frac{c_4L}{t^4(\log t)^2}
     -\frac{C_4L^2}{t^4(\log t)^4}>0.                             \label{eq:unified-base-3-63}
\end{equation}
Thus the $2\times2$ Ricci block on
$\vspan\{\partial_t,E\}$ has positive first diagonal
entry and positive determinant, hence is positive definite.  The
vertical diagonal entry is positive and has no mixed components with
this block.  Therefore
$\Ric_{\widehat g_{\mathrm{neck},m}}>0$ for $m=2,3$.
\end{proof}

We now normalize the neck metric in the form \eqref{eq:unified-base-3-13} on
\([0,\ell]\times\mathbb S^m\). 
\begin{lemma}\label{lem:neck-endpoint-normalization}
Let \(\widehat g_{\mathrm{neck},m}\) be the family of neck metrics constructed in Lemma~\ref{lem:slow-logarithmic-neck}.  For all
sufficiently small \(\sigma>0\), there exist \(\kappa,\ell,\lambda>0\) such that, under the change of variables
\begin{equation*}
 s=\kappa(t-t_0),\qquad t=t_0+\frac{s}{\kappa},
\end{equation*}
the metric \(\kappa^2\widehat g_{\mathrm{neck},m}\) has the form
\eqref{eq:unified-base-3-13} on
\([0,\ell]\times\mathbb S^m\).  The boundary metric on $\{0\}\times \mathbb S^m$ is
\((\rho/\lambda)^2g_{\mathbb S^m}\), with every outward
principal curvature equal to \(-\lambda\).  The boundary metric on $\{\ell\}\times \mathbb S^m$
is \(g_{\mathrm{ang},m}\), with every outward principal
curvature greater than \(1\).  Moreover,
\begin{equation*}
 B'(s)>0,\qquad B''(s)<0\qquad(0\le s\le\ell).
\end{equation*}
This normalization is natural under simultaneous translations
\(z\mapsto z+c\), \(z_0\mapsto z_0+c\), and under orthogonal
transformations of \(\mathbb S^k\).
\end{lemma}

\begin{proof}

Set
\begin{equation*}
\kappa=\frac{w}{t_1v(t_1)}.                                     
\end{equation*}
Multiply $\widehat g_{\mathrm{neck},m}$ by $\kappa^2$.  Since
$f_{u_*}(\varphi)=A_{\mathrm{ang}}(\varphi)/w$ and
$\kappa t_1v(t_1)=w$, the metric induced
at $t=t_1$ is exactly $g_{\mathrm{ang},m}$.  At $t=t_0$, the identity
$f_1(\varphi)=1$ shows that the induced metric is round of radius
\begin{equation*}
\kappa t_0\rho=\frac\rho\lambda,
 \qquad \lambda=\frac1{\kappa t_0}.                             
\end{equation*}
Since $\Xi(t_0)=0$, \eqref{eq:unified-base-3-58} shows that the
principal curvatures at $t=t_0$ with respect to $\partial_t$ all equal
$1/t_0$.  The constant rescaling \(\widehat g\mapsto\kappa^2\widehat g\)
divides principal curvatures by \(\kappa\), so every outward principal
curvature on $\{0\}\times\mathbb S^m$ equals $-1/(\kappa t_0)=-\lambda$.  The  vertical
principal curvatures  on $\{\ell\}\times \mathbb S^m$ are
\begin{equation*}
\frac{v(t_1)}{w}\bigl(1-\beta t_1\Xi(t_1)\bigr),             
\end{equation*}
and the meridional one is 
\begin{equation*}
\frac{v(t_1)}{w}\bigl(1-\beta t_1\Xi(t_1)\bigr)+\frac{t_1v(t_1)}{w}
   \frac{\eta(\varphi)u'(t_1)}{f_{u_*}(\varphi)}.                
\end{equation*}
Because $u\eta/f_u$ is bounded on the compact parameter set, all 
principal curvatures on $\{\ell\}\times \mathbb S^m$ are bounded below by
\begin{equation}
\frac{v(t_1)}{w}\left[1-C\beta t_1\Xi(t_1)\right].           \label{eq:unified-base-3-68}
\end{equation}
As $\sigma\downarrow0$,
\begin{equation*}
\frac{v(t_1)}{w}\longrightarrow
       u_*^{\varepsilon_0/\alpha_0}>1,\qquad
 t_1\Xi(t_1)=\frac{L}{(\log t_1)^2}\longrightarrow0.         
\end{equation*}
Thus \eqref{eq:unified-base-3-68} is greater than one after reducing $\sigma$.
Finally,
\begin{equation}
(tv)'=v(1-\beta t\Xi)>0,\qquad
 \frac{(tv)''}{tv}
 =-\beta\left(\Xi'+\frac{2\Xi}{t}\right)
       +\beta^2\Xi^2<0                                        \label{eq:unified-base-3-70}
\end{equation}
for large $t_0$.  Set
\begin{equation*}
 s=\kappa(t-t_0),\qquad \ell=\kappa(t_1-t_0),\qquad
 t=t_0+\frac{s}{\kappa},
\end{equation*}
and choose
\begin{equation*}
 A(s,\varphi)=\kappa tv(t)f_{u(t)}(\varphi),
 \qquad B(s)=\kappa tv(t).
\end{equation*}
In the $s$-coordinate, $\kappa^2\widehat g_{\mathrm{neck},m}$ is precisely
\eqref{eq:unified-base-3-13}, and \eqref{eq:unified-base-3-70} gives $B'>0$ and $B''<0$.
Every function used in the construction depends on the ambient
\(z\)-coordinate only through \(z-z_0\), and depends on the
\(\mathbb S^k\)-variable only through \(g_{\mathbb S^k}\).  This proves
the asserted naturality.
\end{proof}

\begin{proof}[Proof of Proposition~\ref{prop:unified-neck}]
Choose \(r_*=r_{\mathrm{amb}}=\pi/40\).  For \(0<r<r_*\), choose
\(w>0\) so small that Lemma~\ref{lem:concave-warping} applies with
\(T=5r\) and \(w^{m-1}<\rho^m\), which is
\eqref{eq:unified-base-3-10}.  With the resulting function
\(\mathcal R\), Lemma~\ref{lem:neck-ambient-curvature} proves
\textup{(i)}.

By \eqref{eq:unified-base-3-5}, \(\mathcal R(r)=w\tan r\).
  Hence
Lemmas~\ref{lem:angular-sphere-geometry}
and~\ref{lem:angular-latitude-form} prove \textup{(ii)}.
Since \(r<\pi/40\) and \(\pi\rho<r\), we have
\begin{equation*}
                 0<\rho<\frac1{10}<\cos r .
\end{equation*}
Together with \eqref{eq:unified-base-3-10}, these parameters satisfy the hypotheses of Lemma~\ref{lem:angular-slice-interpolation}.
Lemmas~\ref{lem:variable-angular-ricci}
and~\ref{lem:slow-logarithmic-neck} then produce a family of
Ricci-positive metrics $\widehat g_{\mathrm{neck},m}$.
Lemma~\ref{lem:neck-endpoint-normalization} proves all the
claims in \textup{(iii)}.

A simultaneous translation \(z\mapsto z+c\) and
\(z_0\mapsto z_0+c\) preserves the difference \(z-z_0\) and therefore
carries the construction centered at \(o_{z_0}\) isometrically to the
construction centered at \(o_{z_0+c}\).  Moreover, every metric and
auxiliary function depends on the \(\mathbb S^k\)-variable only through
\(g_{\mathbb S^k}\).  Hence the entire construction is natural under
translations in the \(z\)-direction and orthogonal transformations of
\(\mathbb S^k\).
\end{proof}

\section{The Ricci-positive
\texorpdfstring{$\mathbb Z_p$}{Zp}-equivariant cobordism in base dimension four} \label{app:Zp-equivariant}

This appendix gives the analogue of
Theorem~\ref{thm:ricci-positive-equivariant-cobordism} in base dimension
four with an arbitrary fixed cyclic order $p$.  For $p\ge2$, set
\begin{equation*}
 P_{3,p}=D^4\setminus
 \bigsqcup_{j=0}^{p-1}\operatorname{int}D^4_j,
 \qquad
 \partial P_{3,p}=\partial_oP_{3,p}\sqcup
 \bigsqcup_{j=0}^{p-1}I_j .
\end{equation*}
As before, all second fundamental forms are computed with respect to the
outward unit normal of the manifold under discussion.  We prove the
following statement, using the notation $K_3$ fixed in
Section~\ref{sec:ricci-positive-equivariant-cobordism}.

\begin{theorem}\label{thm:Zp-equivariant-cobordism}
Let $p\ge2$, $R_*>0$, and $\varepsilon_*>0$.  The manifold $P_{3,p}$
admits a metric $h_{3,p}$ with $\Ric_{h_{3,p}}>0$ and an isometry $\sigma_{3,p}:P_{3,p}\to P_{3,p}$ of order $p$ such that:
\begin{enumerate}[label=\textup{(\roman*)}]
\item $\sigma_{3,p}(I_j)=I_{j+1}$, with indices taken modulo $p$,
\item every $I_j$ is a round three-sphere of radius $R_*$ and
\begin{equation*}
 \II_{I_j}>-\varepsilon_*h_{3,p}|_{I_j},
\end{equation*}
\item $\partial_oP_{3,p}$ is round and its second fundamental form is
positive definite,
\item after identifying the round outer boundary with a standard round
$\mathbb S^3\subset\mathbb C^2$,
\begin{equation*}
 \sigma_{3,p}|_{\partial_oP_{3,p}}
   =\exp(2\pi K_3/p).
\end{equation*}
\end{enumerate}
In particular, the action on the outer boundary is free.
\end{theorem}

The relative equivariant gluing lemma,
Lemma~\ref{lem:relative-gluing-unified}, and the Ricci-positive path
cylinder, Theorem~\ref{thm:path-cylinder-unified}, will be used without
restatement.  Fix $0<r<T<\pi/8$ and let $\mathcal R$ and $\mathcal R_u$
be the warping profile and its interpolation from
Lemma~\ref{lem:concave-warping}.  The new points are the torusc ambient
geometry built from $\mathcal R$ and the resulting neck metric on
$[0,\ell]\times\mathbb S^3$.

We use the profile $\mathcal R$ to define the
four-dimensional ambient geometry used to construct the
$\mathbb Z_p$-equivariant cobordism.  The basic model consists of a
round spherical disk together with a two-dimensional disk normal to it.
The function $\mathcal R$ controls the size of the normal circle, while
the factor $\cos t$ warps the metric along the spherical disk.

Let $(\Sigma^2,g_\Sigma)$ be a closed geodesic disk in the unit
round sphere $\mathbb S^2$, centered at $o$.  Consider the
four-dimensional manifold
\begin{equation*}
 \Sigma\times D^2,
\end{equation*}
and let $(t,\eta)$, $0\le t\le T$, be polar coordinates on the $D^2$ factor.  Away from
the zero section $\Sigma\times\{0\}$, define
\begin{equation}
 \widehat g_{\mathrm{amb},3}
 =
 \dd t^2+\cos^2t\,g_\Sigma+\mathcal R(t)^2\dd\eta^2,
 \qquad
 g_{\mathrm{amb},3}
 =
 \cot^2r\,\widehat g_{\mathrm{amb},3}.
 \label{eq:appb-ambient}
\end{equation}
At $t=0$, the $\eta$-circle collapses to the zero section
$\Sigma\times\{0\}\cong\Sigma$.  The conditions on $\mathcal R$ at
$t=0$ in Lemma~\ref{lem:appb-profile}, together with the evenness of
$\cos^2t$, imply that
$\widehat g_{\mathrm{amb},3}$, and hence
$g_{\mathrm{amb},3}$, extends smoothly across the zero section.

Let $X$ and $Y$ be tangent to $\Sigma$ and orthonormal with respect to $\widehat g_{\mathrm{amb},3}$.  The vector
$\mathcal R^{-1}\partial_\eta$ is then unit in the circle direction.
Before applying the constant rescaling, the multiply warped-product
curvature formulas give
\begin{align}
 \sec(\partial_t,X)&=1,
 &
 \sec\bigl(\partial_t,\mathcal R^{-1}\partial_\eta\bigr)
 &=-\frac{\mathcal R''}{\mathcal R},
 \label{eq:appb-ambient-K1}\\
 \sec(X,Y)&=1,
 &
 \sec\bigl(X,\mathcal R^{-1}\partial_\eta\bigr)
 &=\tan t\,\frac{\mathcal R'}{\mathcal R}.
 \label{eq:appb-ambient-K2}
\end{align}
The estimates from
Lemma~\ref{lem:appb-profile} show that they are all strictly positive.
Since constant rescaling preserves the sign of sectional curvature, it
follows that
\begin{equation*}
 \sec_{g_{\mathrm{amb},3}}>0
\end{equation*}
throughout the region under consideration.

Fix $x_0\in\Sigma$ and regard it as the point
$(0, x_0)$ on the zero section $\{0\}\times\Sigma$.  Consider the
projection
\begin{equation*}
 \pi\colon D^2\times\Sigma\longrightarrow[0,T]\times\Sigma,
 \qquad
 \pi(t,\eta,x)=(t,x),
\end{equation*}
where the target is equipped with the metric
\begin{equation*}
 \bar g=\dd t^2+\cos^2t\,g_\Sigma.
\end{equation*}
Since $\pi$ simply discards the nonnegative term
$\mathcal R(t)^2\dd\eta^2$, it is length nonincreasing.  Conversely, a
minimizing geodesic in $([0,T]\times\Sigma,\bar g)$ issuing from
$(0,x_0)$ admits a horizontal lift with constant $\eta$ and with the
same length.  Therefore, the distance from $x_0$ with respect to
$\widehat g_{\mathrm{amb},3}$ agrees with the corresponding distance
for $\bar g$.

The metric $\bar g$ is the unit round metric written in Fermi
coordinates about an equatorusal $2$-sphere.  Hence, if
$\varrho_{x_0}$ denotes the distance from $x_0$, the spherical law of
cosines gives
\begin{equation}
 \cos\varrho_{x_0}(t,x,\eta)
 =
 \cos t\,\cos d_\Sigma(x,x_0).
 \label{eq:appb-distance}
\end{equation}
Accordingly, we write
\begin{equation*}
 B_a^\angle(x_0)=\{\varrho_{x_0}\leq a\},
 \qquad
 S_a^\angle(x_0)=\{\varrho_{x_0}=a\},
\end{equation*}
in agreement with the notation of
Section~\ref{sec:ricci-positive-equivariant-cobordism}.

Introduce polar coordinates $(s,\psi)$ centered at $x_0$ in $\Sigma$,
so that
\begin{equation*}
 g_\Sigma=\dd s^2+\sin^2s\,\dd\psi^2.
\end{equation*}
On the level set $S_r^\angle(x_0)$,
equation~\eqref{eq:appb-distance} becomes
\begin{equation}
 \cos r=\cos t\cos s(t),
 \qquad 0\leq t\leq r.
 \label{eq:appb-s-of-t}
\end{equation}
In particular,
\begin{equation*}
 s(0)=r,\qquad s(r)=0.
\end{equation*}
Restricting $\widehat g_{\mathrm{amb},3}$ to this level set gives
\begin{equation*}
 \bigl(1+\cos^2t\,s'(t)^2\bigr)\dd t^2
 +\cos^2t\sin^2s(t)\,\dd\psi^2
 +\mathcal R(t)^2\dd\eta^2.
\end{equation*}
Using \eqref{eq:appb-s-of-t}, we have
\begin{equation*}
 \cos^2t\sin^2s(t)=\sin^2r-\sin^2t
\end{equation*}
and
\begin{equation*}
 1+\cos^2t\,s'(t)^2
 =
 \frac{\sin^2r\cos^2t}{\sin^2r-\sin^2t}.
\end{equation*}
Thus the metric induced by $g_{\mathrm{amb},3}$ on
$S_r^\angle(x_0)$ is
\begin{equation}
 g_{\mathrm{ang},3}
 =
 \cot^2r\left\{
  A_r(t)^2\dd t^2
  +(\sin^2r-\sin^2t)\dd\psi^2
  +\mathcal R(t)^2\dd\eta^2
 \right\},
 \label{eq:appb-angular}
\end{equation}
where
\begin{equation*}
 A_r(t)^2
 =
 1+\cos^2t\,s'(t)^2
 =
 \frac{\sin^2r\cos^2t}{\sin^2r-\sin^2t}.
\end{equation*}

For $0<t<r$, the coordinates $(\psi,\eta)$ parametrize a $T^2$-orbit.  At
$t=0$, the $\eta$-circle collapses, whereas at $t=r$, the $\psi$-circle
collapses.  Thus $S_r^\angle(x_0)$ is obtained from
$[0,r]\times T^2$ by collapsing one circle factor at each endpoint,
and is therefore diffeomorphic to $\mathbb S^3$.  The endpoint
conditions in~\eqref{eq:appb-R-smooth} ensure that
\eqref{eq:appb-angular} extends smoothly across both collapsed orbits.
Consequently, $g_{\mathrm{ang},3}$ is a smooth $T^2$-invariant metric
on $\mathbb S^3$.

For each $u\in[0,1]$, we first form the profile
$\mathcal R_u$ defined in~\eqref{eq:appb-Ru}.  Replacing $\mathcal R$
by $\mathcal R_u$ in~\eqref{eq:appb-ambient}, we then define
\begin{equation*}
 \widehat g_{\mathrm{amb},3}(u)
 =
 \dd t^2+\cos^2t\,g_\Sigma+\mathcal R_u(t)^2\dd\eta^2,
 \qquad
 g_{\mathrm{amb},3}(u)
 =
 \cot^2r\,\widehat g_{\mathrm{amb},3}(u).
\end{equation*}
The distance function $\varrho_{x_0}$ and its level set
$S_r^\angle(x_0)$ are independent of $u$, since they are determined by
the projection onto
$\dd t^2+\cos^2t\,g_\Sigma$.  Let
\begin{equation*}
 g_{\mathrm{ang},3}(u)
 =
 \left.g_{\mathrm{amb},3}(u)\right|_{S_r^\angle(x_0)}.
\end{equation*}
Thus
\begin{equation}
 g_{\mathrm{ang},3}(u)
 =
 \cot^2r\left\{
  A_r(t)^2\dd t^2
  +(\sin^2r-\sin^2t)\dd\psi^2
  +\mathcal R_u(t)^2\dd\eta^2
 \right\}.
 \label{eq:appb-angular-path}
\end{equation}
We abbreviate
\begin{equation*}
 g_{\mathrm{ang},3}=g_{\mathrm{ang},3}(1).
\end{equation*}

\begin{lemma}\label{lem:appb-angular}
For every $u\in[0,1]$,
\begin{equation}
 \sec_{g_{\mathrm{ang},3}(u)}>1.
 \label{eq:appb-angular-sec}
\end{equation}
With respect to the normal pointing out of $B_r^\angle(x_0)$, all
principal curvatures of
$\bigl(S_r^\angle(x_0),g_{\mathrm{ang},3}(u)\bigr)$ lie in $(0,1]$.
Moreover,
\begin{equation}
 g_{\mathrm{ang},3}(0)
 =
 \cos^2r\,g_{\mathbb S^3}.
 \label{eq:appb-round-end}
\end{equation}
\end{lemma}

\begin{proof}
For $0<t<r$, the following three vectors are unit tangent vectors with
respect to $g_{\mathrm{ang},3}(u)$:
\begin{equation*}
 \frac{\tan r}{A_r(t)}
 \bigl(\partial_t+s'(t)\partial_s\bigr),
 \qquad
 \frac{\tan r}{\sqrt{\sin^2r-\sin^2t}}\,\partial_\psi,
 \qquad
 \frac{\tan r}{\mathcal R_u(t)}\,\partial_\eta.
\end{equation*}
The Gauss equation, together with the ambient curvature formulas
\eqref{eq:appb-ambient-K1}--\eqref{eq:appb-ambient-K2}, gives
\begin{align}
 &\sec_{g_{\mathrm{ang},3}(u)}
 \left(
  \frac{\tan r}{A_r(t)}
  \bigl(\partial_t+s'(t)\partial_s\bigr),
  \frac{\tan r}{\sqrt{\sin^2r-\sin^2t}}\,\partial_\psi
 \right)
 =\sec^2r,
 \label{eq:appb-sec-tpsi}\\
 &\sec_{g_{\mathrm{ang},3}(u)}
 \left(
  \frac{\tan r}{\sqrt{\sin^2r-\sin^2t}}\,\partial_\psi,
  \frac{\tan r}{\mathcal R_u(t)}\,\partial_\eta
 \right)
 =\frac{x_{\mathcal R_u}(t)}{\cos^2r}.
 \label{eq:appb-sec-psieta}
\end{align}
\begin{align*}
 &\sec_{g_{\mathrm{ang},3}(u)}
 \left(
  \frac{\tan r}{A_r(t)}
  \bigl(\partial_t+s'(t)\partial_s\bigr),
  \frac{\tan r}{\mathcal R_u(t)}\,\partial_\eta
 \right)
 =q_{\mathcal R_u}(t)
   \bigl(\tan^2r-\tan^2t\bigr)
  +x_{\mathcal R_u}(t)\sec^2t.
\end{align*}
The curvature operator is diagonal with respect to these $2$-planes, since the metric is diagonal with coefficients depending only on $t$ and the reflections $\psi\mapsto-\psi$ and $\eta\mapsto-\eta$ are isometries.  By Lemma~\ref{lem:appb-profile},
\begin{equation*}
 q_{\mathcal R_u}\geq1,
 \qquad
 \cos^2r<x_{\mathcal R_u}\leq1.
\end{equation*}
The right-hand sides of
\eqref{eq:appb-sec-tpsi} and
\eqref{eq:appb-sec-psieta} are therefore strictly larger than one.
For the remaining plane,
\begin{align*}
 &q_{\mathcal R_u}(t)
   \bigl(\tan^2r-\tan^2t\bigr)
  +x_{\mathcal R_u}(t)\sec^2t>
 \tan^2r-\tan^2t+\cos^2r\,\sec^2t
 \geq1.
\end{align*}
Indeed, the difference between the final expression and one is
\begin{equation*}
 \frac{(\cos^2t-\cos^2r)\sin^2r}
      {\cos^2r\cos^2t}
 \geq0.
\end{equation*}
This proves~\eqref{eq:appb-angular-sec}.

Let $\mathrm{II}_u$ denote the second fundamental form with respect to
the outward unit normal.  Its values on the three unit principal
directions are
\begin{align*}
 &\mathrm{II}_u\left(
  \frac{\tan r}{A_r(t)}
  \bigl(\partial_t+s'(t)\partial_s\bigr),
  \frac{\tan r}{A_r(t)}
  \bigl(\partial_t+s'(t)\partial_s\bigr)
 \right)=1,\\
 &\mathrm{II}_u\left(
  \frac{\tan r}{\sqrt{\sin^2r-\sin^2t}}\,\partial_\psi,
  \frac{\tan r}{\sqrt{\sin^2r-\sin^2t}}\,\partial_\psi
 \right)=1,\\
 &\mathrm{II}_u\left(
  \frac{\tan r}{\mathcal R_u(t)}\,\partial_\eta,
  \frac{\tan r}{\mathcal R_u(t)}\,\partial_\eta
 \right)
 =
 \frac{\mathcal R_u'(t)}{\mathcal R_u(t)}\tan t
 =
 x_{\mathcal R_u}(t).
\end{align*}
Since $0<x_{\mathcal R_u}\leq1$, all three principal curvatures lie in
$(0,1]$.  The formulas extend to the two collapsed circle orbits by
smoothness.

Finally, $\mathcal R_0(t)=\sin t$, and hence
\begin{equation*}
 \widehat g_{\mathrm{amb},3}(0)
 =
 \dd t^2+\cos^2t\,g_\Sigma+\sin^2t\,\dd\eta^2
\end{equation*}
is the unit round metric in the standard join coordinates.  Its
distance sphere of radius $r$ has induced metric
$\sin^2r\,g_{\mathbb S^3}$.  Multiplication by $\cot^2r$ therefore
gives
\begin{equation*}
 g_{\mathrm{ang},3}(0)
 =
 \cot^2r\sin^2r\,g_{\mathbb S^3}
 =
 \cos^2r\,g_{\mathbb S^3},
\end{equation*}
which proves~\eqref{eq:appb-round-end}.
\end{proof}

In particular, by compactness, there is $\varepsilon_0>0$ such that
\begin{equation}
 \sec_{g_{\mathrm{ang},3}(u)}\ge1+\varepsilon_0
 \qquad(0\le u\le1).
 \label{eq:appb-uniform-sec}
\end{equation}

We now attach a Ricci-positive neck that connects the angular metric
$g_{\mathrm{ang},3}$ to a round metric.  This plays the same role as
Lemmas~\ref{lem:angular-slice-interpolation}--
\ref{lem:neck-endpoint-normalization} in Appendix~\ref{app:perelman-neck}.
In the present setting, however, the path
$g_{\mathrm{ang},3}(u)$ changes both circle directions of the
$T^2$-invariant metric.  We therefore first use Moser's argument \cite{Moser1965} to place the path in a fixed-volume gauge, and then vary it sufficiently slowly
along the neck.

Let
\begin{equation*}
 V_*=\operatorname{Vol}(\mathbb S^3,g_{\mathbb S^3})=2\pi^2
\end{equation*}
and define the volume scale
\begin{equation*}
 \Theta(u)=
 \left(\frac{\operatorname{Vol}(\mathbb S^3,g_{\mathrm{ang},3}(u))} {V_*}\right)^{1/3}.
\end{equation*}
From~\eqref{eq:appb-angular-path}, we obtain
\begin{equation*}
 \operatorname{Vol}
   (\mathbb S^3,g_{\mathrm{ang},3}(u))
 =
 4\pi^2\cot^3r\sin r
 \int_0^r\cos t\,\mathcal R_u(t)\,\dd t.
\end{equation*}
Since
\begin{equation*}
 \mathcal R_u=(1-u)\sin t+u\mathcal R(t),
 \qquad
 \mathcal R(t)\leq\sin t,
\end{equation*}
the function $\Theta$ is nonincreasing, and
\begin{equation*}
 \Theta(0)=\cos r.
\end{equation*}
Moreover, $\mathcal R$ is increasing and
$\mathcal R(r)=w\tan r$, so
\begin{equation}
 \Theta(1)^3
 \leq 2\cos^2r\,w.
 \label{eq:appb-volume-waist}
\end{equation}

We next place the path in a fixed-volume gauge.  Define
\begin{equation*}
 \omega_u
 =
 \Theta(u)^{-3}
 \dd\operatorname{vol}_{g_{\mathrm{ang},3}(u)}.
\end{equation*}
Then
\begin{equation*}
 \int_{\mathbb S^3}\omega_u=V_*
\end{equation*}
for every $u$, and
\begin{equation*}
 \omega_0=\dd\operatorname{vol}_{g_{\mathbb S^3}}.
\end{equation*}
In particular, $\dot\omega_u$ has zero integral and is therefore exact.
Choose a smooth family of $2$-forms $\alpha_u$ satisfying
\begin{equation*}
 \dd\alpha_u=\dot\omega_u.
\end{equation*}
Since $\omega_u$ is $T^2$-invariant, averaging $\alpha_u$ over $T^2$
allows us to assume that $\alpha_u$ is also $T^2$-invariant.

Following Moser's argument in~\cite{Moser1965}, define the time-dependent vector field
$X_u$ uniquely by
\begin{equation*}
 \iota_{X_u}\omega_u=-\alpha_u,
\end{equation*}
and let $\Phi_u$ be its flow:
\begin{equation*}
 \frac{\dd}{\dd u}\Phi_u
 =
 X_u\circ\Phi_u,
 \qquad
 \Phi_0=\Id_{\mathbb S^3}.
\end{equation*}
Because both $\omega_u$ and $\alpha_u$ are $T^2$-invariant, so is
$X_u$, and hence $\Phi_u$ commutes with the $T^2$-action.  Furthermore,
\begin{align*}
 \frac{\dd}{\dd u}\Phi_u^*\omega_u
 &=
 \Phi_u^*\bigl(\dot\omega_u+\mathcal L_{X_u}\omega_u\bigr)=
 \Phi_u^*\bigl(\dd\alpha_u+\dd\iota_{X_u}\omega_u\bigr)
 =0.
\end{align*}
Consequently,
\begin{equation}
 \Phi_u^*
 \dd\operatorname{vol}_{g_{\mathrm{ang},3}(u)}
 =
 \Theta(u)^3
 \dd\operatorname{vol}_{g_{\mathbb S^3}}.
 \label{eq:appb-Moser-volume}
\end{equation}

Set
\begin{equation*}
 \widetilde g_u
 =
 \Theta(u)^{-2}
 \Phi_u^*g_{\mathrm{ang},3}(u),
 \qquad
 L_u
 =
 \frac12\widetilde g_u^{-1}\dot{\widetilde g}_u.
\end{equation*}
Equation~\eqref{eq:appb-Moser-volume} gives
\begin{equation*}
 \dd\operatorname{vol}_{\widetilde g_u}
 =
 \dd\operatorname{vol}_{g_{\mathbb S^3}},
 \qquad
 \widetilde g_0=g_{\mathbb S^3},
 \qquad
 \tr L_u=0.
\end{equation*}
By~\eqref{eq:appb-uniform-sec},
\begin{equation}
 (\Ric_{\widetilde g_u})^\sharp
 \geq
 2(1+\varepsilon_0)\Theta(u)^2\Id
 \geq
 2(1+\varepsilon_0)\Theta(1)^2\Id.
 \label{eq:appb-Ric-gtilde}
\end{equation}

Fix $0<\rho<1$ and choose $w>0$ sufficiently small that
\begin{equation}
 2\cos^2r\,w<(\rho/2)^3.
 \label{eq:appb-waist-choice}
\end{equation}
With $\varepsilon_0$ fixed, choose
$0<\delta<1$ such that
\begin{equation*}
 (1+\delta)^2<1+\varepsilon_0
\end{equation*}
and set
\begin{equation*}
 \beta_1=(1+\delta)\Theta(1).
\end{equation*}
By~\eqref{eq:appb-volume-waist}--
\eqref{eq:appb-waist-choice}, one has $\beta_1<\rho$.  We may therefore
choose
\begin{equation}
 \beta_1<\beta_*<
 \min\left\{
  \rho,\sqrt{1+\varepsilon_0}\,\Theta(1)
 \right\}.
 \label{eq:appb-B-choice}
\end{equation}

For $\tau>0$ and $t_0>0$, define
\begin{equation*}
 c_0=\log\frac{\rho}{\beta_1},
 \qquad
 \kappa(t)=\tau\log\frac{t}{t_0},
 \qquad
 b(t)=\rho e^{-\kappa(t)},
 \qquad
 t_1=t_0e^{c_0/\tau}.
\end{equation*}
Thus
\begin{equation*}
 \kappa(t_0)=0,\qquad
 \kappa(t_1)=c_0,
 \qquad
 b(t_0)=\rho,\qquad
 b(t_1)=\beta_1.
\end{equation*}
Choose a smooth nondecreasing function
$\chi\colon[0,c_0]\to[0,1]$, constant near both endpoints, such that
\begin{equation*}
 \chi(\kappa)=0
 \quad\text{for}\quad
 0\leq\kappa\leq\log\frac{\rho}{\beta_*},
 \qquad
 \chi(\kappa)=1
 \quad\text{near }\kappa=c_0.
\end{equation*}
On $[t_0,t_1]\times\mathbb S^3$, define
\begin{equation}
 \widehat g_{\mathrm{neck},3}
 =
 \dd t^2+\bigl(tb(t)\bigr)^2
 \widetilde g_{\chi(\kappa(t))}.
 \label{eq:appb-neck}
\end{equation}

\begin{proposition}\label{prop:appb-neck}
For every sufficiently small $\tau>0$, there exist
$\ell,R_0,\lambda_0>0$ and a $T^2$-invariant metric
$g_{\mathrm{neck},3}$ on $[0,\ell]\times\mathbb S^3$ such that
\begin{equation*}
 \Ric_{g_{\mathrm{neck},3}}>0
\end{equation*}
and the following boundary properties hold.
\begin{enumerate}[label=\textup{(\roman*)}]
\item On $\{0\}\times\mathbb S^3$, the induced metric is
\begin{equation*}
 R_0^2g_{\mathbb S^3},
\end{equation*}
and every principal curvature with respect to the outward normal
$-\partial_s$ equals $-\lambda_0$.  Moreover,
\begin{equation}
 \lambda_0R_0=(1-\tau)\rho.
 \label{eq:appb-round-product}
\end{equation}

\item On $\{\ell\}\times\mathbb S^3$, the induced metric is
\begin{equation*}
 \Phi_1^*g_{\mathrm{ang},3},
\end{equation*}
and every principal curvature with respect to the outward normal
$\partial_s$ is strictly larger than one.
\end{enumerate}
The diffeomorphism $\Phi_1$ and the metric
$g_{\mathrm{neck},3}$ commute with the $T^2$-action.
\end{proposition}

\begin{proof}
We apply the Gaussian-coordinate identities of
Lemma~\ref{lem:gaussian-ricci-unified} to the slice metrics
\begin{equation*}
 h(t)
 =
 \bigl(tb(t)\bigr)^2
 \widetilde g_{\chi(\kappa(t))}.
\end{equation*}
Here primes on $\chi$ denote differentiation with respect to its
argument.  Since
\begin{equation*}
 \kappa'(t)=\frac{\tau}{t},
 \qquad
 \frac{b'(t)}{b(t)}=-\frac{\tau}{t},
\end{equation*}
the shape endomorphism of a $t$-slice is
\begin{equation*}
 \frac12h^{-1}h'
 =
 \frac{1-\tau}{t}\Id
 +
 \frac{\tau\chi'(\kappa(t))}{t}
 L_{\chi(\kappa(t))}^\sharp.
\end{equation*}
Using $\tr L_u^\sharp=0$ in
Lemma~\ref{lem:gaussian-ricci-unified}, we obtain
\begin{equation*}
 \Ric_{\widehat g_{\mathrm{neck},3}}
   (\partial_t,\partial_t)
 =
 \frac{
  3\tau(1-\tau)
  -\tau^2\chi'(\kappa(t))^2
   \tr\!\left(
    \bigl(L_{\chi(\kappa(t))}^\sharp\bigr)^2
   \right)}
 {t^2}.
\end{equation*}
Thus there are constants $\tau_0,c_{\mathrm{rad}}>0$, depending only
on the fixed path $\widetilde g_u$ and on $\chi$, such that
\begin{equation}
 \Ric_{\widehat g_{\mathrm{neck},3}}
   (\partial_t,\partial_t)
 \geq\frac{c_{\mathrm{rad}}\tau}{t^2}
 \qquad
 (0<\tau<\tau_0).
 \label{eq:appb-radial-positive}
\end{equation}

Write
\begin{equation*}
 u=\chi(\kappa(t)).
\end{equation*}
The tangential component is
\begin{align}
 \bigl(
  \Ric_{\widehat g_{\mathrm{neck},3}}
    |_{T\mathbb S^3}
 \bigr)^\sharp
 ={}&
 \bigl(tb(t)\bigr)^{-2}
 (\Ric_{\widetilde g_u})^\sharp
 +\frac{-2+5\tau-3\tau^2}{t^2}\Id
 \notag\\
 &-\frac{\tau}{t^2}
   \bigl[\tau\chi''+(2-3\tau)\chi'\bigr]L_u^\sharp
 -\frac{\tau^2(\chi')^2}{t^2}
   \frac{\dd}{\dd u}L_u^\sharp,
 \label{eq:appb-tangential}
\end{align}
where $\chi'$ and $\chi''$ are evaluated at $\kappa(t)$.  The mixed
component is
\begin{equation}
 \Ric_{\widehat g_{\mathrm{neck},3}}
   (\partial_t,\mathord\cdot)
 =
 \frac{\tau\chi'}{t}
 \divsymb_{\widetilde g_u}L_u.
 \label{eq:appb-mixed}
\end{equation}

We next isolate a uniform positive part in the tangential block.
Where $\chi(\kappa(t))=0$, we have $u=0$ and
$\widetilde g_0=g_{\mathbb S^3}$, so
\begin{equation*}
 \bigl(tb(t)\bigr)^{-2}
 (\Ric_{\widetilde g_0})^\sharp
 -\frac2{t^2}\Id
 =
 \frac2{t^2}
 \left(\frac1{b(t)^2}-1\right)\Id
 \geq
 \frac2{t^2}
 \left(\frac1{\rho^2}-1\right)\Id.
\end{equation*}
Wherever the path varies, the choice of $\chi$ gives
$b(t)<\beta_*$.  Equations~\eqref{eq:appb-Ric-gtilde}
and~\eqref{eq:appb-B-choice} then imply
\begin{equation*}
 \bigl(tb(t)\bigr)^{-2}
 (\Ric_{\widetilde g_u})^\sharp
 -\frac2{t^2}\Id
 \geq
 \frac2{t^2}
 \left(
  \frac{(1+\varepsilon_0)\Theta(1)^2}{\beta_*^2}-1
 \right)\Id.
\end{equation*}
Consequently, there is a constant $\mu>0$, independent of $\tau$,
such that
\begin{equation*}
 \bigl(tb(t)\bigr)^{-2}
 (\Ric_{\widetilde g_u})^\sharp
 -\frac2{t^2}\Id
 \geq\frac{\mu}{t^2}\Id
\end{equation*}
throughout the neck.

The compact path $\widetilde g_u$ supplies uniform bounds for
$L_u$, $L_u^\sharp$, $\frac{\dd}{\dd u}L_u^\sharp$, and their required
spatial derivatives.  Hence there are constants $C_1,C_2>0$,
independent of $\tau$, such that, after requiring
\begin{equation*}
 \tau<\frac{\mu}{2C_1},
\end{equation*}
equation~\eqref{eq:appb-tangential} gives
\begin{equation*}
 \Ric_{\widehat g_{\mathrm{neck},3}}
   |_{T\mathbb S^3}
 \geq
 \frac{\mu}{2t^2}
 \widehat g_{\mathrm{neck},3}|_{T\mathbb S^3},
\end{equation*}
while~\eqref{eq:appb-mixed} gives
\begin{equation*}
 \left|
  \Ric_{\widehat g_{\mathrm{neck},3}}
   (\partial_t,\mathord\cdot)
 \right|
 \leq
 \frac{C_2\tau}{t^2b(t)}
 \leq
 \frac{C_2\tau}{t^2\beta_1}.
\end{equation*}
The loss in the radial Schur complement is therefore at most
\begin{equation*}
 \frac{2C_2^2\tau^2}
      {\mu t^2\beta_1^2}.
\end{equation*}
Combining this with~\eqref{eq:appb-radial-positive}, we obtain
$\Ric_{\widehat g_{\mathrm{neck},3}}>0$ provided
\begin{equation*}
 0<\tau<
 \min\left\{
  \tau_0,\,
  \frac{\mu}{2C_1},\,
  \frac{c_{\mathrm{rad}}\mu\beta_1^2}{4C_2^2}
 \right\}.
\end{equation*}

It remains to normalize the two ends.  Set
\begin{equation*}
 C_{\mathrm{sc}}
 =
 \frac{\Theta(1)}{t_1\beta_1}.
\end{equation*}
After multiplying~\eqref{eq:appb-neck} by $C_{\mathrm{sc}}^2$ and
introducing the arclength coordinate
\begin{equation*}
 s=C_{\mathrm{sc}}(t-t_0),
 \qquad
 \ell=C_{\mathrm{sc}}(t_1-t_0),
\end{equation*}
we obtain a metric $g_{\mathrm{neck},3}$ on
$[0,\ell]\times\mathbb S^3$.

At $t=t_1$, we have
\begin{equation*}
 b(t_1)=\beta_1,
 \qquad
 \kappa(t_1)=c_0,
 \qquad
 \chi(c_0)=1.
\end{equation*}
Hence the induced metric is
\begin{equation*}
 C_{\mathrm{sc}}^2(t_1\beta_1)^2\widetilde g_1
 =
 \Theta(1)^2\widetilde g_1
 =
 \Phi_1^*g_{\mathrm{ang},3}.
\end{equation*}
Since $\chi$ is constant near $c_0$, every outward principal curvature
at this end equals
\begin{equation*}
 \frac{(1-\tau)/t_1}{C_{\mathrm{sc}}}
 =
 (1-\tau)\frac{\beta_1}{\Theta(1)}
 =
 (1-\tau)(1+\delta).
\end{equation*}
After additionally requiring
\begin{equation*}
 \tau<\frac{\delta}{1+\delta},
\end{equation*}
these principal curvatures are strictly larger than one.

At $t=t_0$, the induced metric is round with radius
\begin{equation*}
 R_0=C_{\mathrm{sc}}t_0\rho.
\end{equation*}
Since the outward normal at this end is $-\partial_s$, every outward
principal curvature equals $-\lambda_0$, where
\begin{equation*}
 \lambda_0
 =
 \frac{1-\tau}{C_{\mathrm{sc}}t_0}.
\end{equation*}
It follows that
\begin{equation*}
 \lambda_0R_0=(1-\tau)\rho,
\end{equation*}
which proves~\eqref{eq:appb-round-product}.  Finally, the path
$\widetilde g_u$, the function $\chi$, and the Moser isotopy $\Phi_u$
all commute with the $T^2$-action.  Hence so does the entire neck
construction.
\end{proof}

We now complete the construction of the Ricci-positive $\mathbb Z_p$-equivariant cobordism.
Fix $p\ge2$.  Let $\mathbb Z_p$ act on the round disk $\Sigma$ by
rotation about $o$.  Choose $x_0\ne o$ sufficiently close to $o$, put
\begin{equation*}
 x_j=\exp(2\pi jK_2/p)x_0,
 \qquad j=0,\ldots,p-1,
\end{equation*}
and then choose $r>0$ so small that
$d_\Sigma(x_i,x_j)>2r$ for $i\ne j$.  Choose $T<\pi/8$ with
\begin{equation*}
 B_r^\angle(x_j)\Subset B_T^\angle(o)
 \qquad(0\le j\le p-1).
\end{equation*}
All these choices can be made inside the coordinate region used above.
Define
\begin{equation*}
 \Omega_{p,r}=B_T^\angle(o)\setminus
 \bigsqcup_{j=0}^{p-1}\operatorname{int}B_r^\angle(x_j).
\end{equation*}
The angular balls are normal four-balls.  Consequently,
\begin{equation}
 (\Omega_{p,r},S_T^\angle(o),S_r^\angle(x_0),\ldots,
 S_r^\angle(x_{p-1}))
 \cong(P_{3,p},\partial_oP_{3,p},I_0,\ldots,I_{p-1}).
 \label{eq:appb-topology}
\end{equation}

Identify the tangent plane of $\Sigma$ at $o$ and the normal disk with
$\mathbb C$.  In these two factors define
\begin{equation}
 \sigma(z,t,e^{\mathrm i\eta})
 =\left(e^{2\pi\mathrm i/p}z,
        t,e^{2\pi\mathrm i/p}e^{\mathrm i\eta}\right).
 \label{eq:appb-local-action}
\end{equation}
This is an order-$p$ isometry of~\eqref{eq:appb-ambient}.  It preserves
the outer angular ball and cyclically permutes the inner angular balls.
At the fixed center $o$, its tangent representation is
\begin{equation}
 d\sigma_o=\operatorname{diag}
 (R_{2\pi/p},R_{2\pi/p})
 \quad\text{on }T_o\Sigma\oplus\nu_o\Sigma\cong\mathbb C^2.
 \label{eq:appb-tangent-action}
\end{equation}

\begin{proof}[Proof of Theorem~\ref{thm:Zp-equivariant-cobordism}]
Choose
\begin{equation}
 0<\rho<\min\{1,R_*\varepsilon_*/2\},
 \label{eq:appb-rho-choice}
\end{equation}
and then $w>0$ so small that Lemma~\ref{lem:concave-warping},
\eqref{eq:appb-waist-choice}, and Proposition~\ref{prop:appb-neck}
apply.  The proof follows that of
Theorem~\ref{thm:ricci-positive-equivariant-cobordism}, with
$\Omega_{p,r}$ in place of $\Omega_r$ and the necks of
Proposition~\ref{prop:appb-neck} in place of those of
Proposition~\ref{prop:unified-neck}, we indicate the changes.

The metric $g_{\mathrm{amb},3}$ on $\Omega_{p,r}$ has positive sectional
curvature by \eqref{eq:appb-ambient-K1}--\eqref{eq:appb-ambient-K2}.
Each inner boundary component has induced metric $g_{\mathrm{ang},3}$,
and, since its outward normal as a boundary of $\Omega_{p,r}$ is the
negative of the normal pointing out of the removed angular ball,
Lemma~\ref{lem:appb-angular} gives
\begin{equation}
 \II_{\Omega_{p,r}}\ge-g_{\mathrm{ang},3}.
 \label{eq:appb-inner-Omega-II}
\end{equation}
We attach $p$ copies of the neck of Proposition~\ref{prop:appb-neck},
identifying each terminal end with an inner boundary component via
$\Phi_1$ and the angular coordinates.  Since all metrics and maps
commute with the $T^2$-action, these identifications can be obtained
from one of them by the action \eqref{eq:appb-local-action}.  By
\eqref{eq:appb-inner-Omega-II} and the principal-curvature estimate of
Proposition~\ref{prop:appb-neck}, Lemma~\ref{lem:relative-gluing-unified}
applies at every interface, and its equivariant version preserves the
order-$p$ isometry.  The exposed inner ends are round of a common radius
$R_{\mathrm{in}}$.  After multiplying the metric by $c_{\mathrm{in}}^2$,
where
\begin{equation}
 c_{\mathrm{in}}=\frac{R_*}{R_{\mathrm{in}}},
 \label{eq:appb-global-scale}
\end{equation}
the incoming radius become $R_*$ and, by \eqref{eq:appb-round-product}
and \eqref{eq:appb-rho-choice}, every incoming principal curvature
equals $-(1-\tau)\rho/R_*>-\varepsilon_*$.

To round the outer boundary we use, instead of Ricci flow, the
explicit $T^2$-invariant path
\begin{equation}
 g^{\mathrm{out}}_u=c_{\mathrm{in}}^2\,
  g_{\mathrm{amb},3}(u)|_{S_T^\angle(o)},
 \qquad
 g_{\mathrm{amb},3}(u)=\cot^2r\left(
 \dd t^2+\cos^2t\,g_\Sigma+\mathcal R_u(t)^2\dd\eta^2\right),
 \qquad 0\le u\le1.
 \label{eq:appb-outer-path}
\end{equation}
By Lemma~\ref{lem:concave-warping} and
\eqref{eq:appb-ambient-K1}--\eqref{eq:appb-ambient-K2}, every
$g_{\mathrm{amb},3}(u)$ has positive sectional curvature.  The two
principal curvatures of $S_T^\angle(o)$ in the directions tangent to
$\Sigma$ are positive, and the remaining one has the sign of
$\mathcal R_u'>0$, hence $S_T^\angle(o)$ is strictly convex and the Gauss
equation gives $\sec_{g^{\mathrm{out}}_u}>0$.  Here
$g^{\mathrm{out}}_1$ is the actual outer boundary metric, whereas
\begin{equation}
 g^{\mathrm{out}}_0=(c_{\mathrm{in}}\cot r\sin T)^2g_{\mathbb S^3}
 \label{eq:appb-outer-round}
\end{equation}
in the fixed join coordinates.  Reversing and reparametrizing this path,
we attach an equivariant Ricci-positive path cylinder by
Theorem~\ref{thm:path-cylinder-unified} and
Lemma~\ref{lem:relative-gluing-unified}, exactly as in the proof of
Theorem~\ref{thm:ricci-positive-equivariant-cobordism}.  The new outer
boundary is round and strictly convex, and the metric is unchanged near
the incoming boundary components.

Finally, the path \eqref{eq:appb-outer-path} preserves the full
$T^2$-action in fixed join coordinates.  Under the round identification
\eqref{eq:appb-outer-round}, the generator is therefore the restriction
of the linear map \eqref{eq:appb-tangent-action}, namely
$(z_1,z_2)\mapsto(e^{2\pi\mathrm i/p}z_1,e^{2\pi\mathrm i/p}z_2)
=\exp(2\pi K_3/p)(z_1,z_2)$.  Transporting the metric and the action
through \eqref{eq:appb-topology} gives $(h_{3,p},\sigma_{3,p})$.
\end{proof}

\begin{remark}\label{rem:Zp-downstream}
The rest of the construction follows Sections~\ref{sec:unified-cable}
and~\ref{sec:holonomy-interpolation-cylinder} with only the cyclic order
changed.  In Lemma~\ref{lem:un-warp}, prescribe boundary value one on
$\partial_oP_{3,p}$ and zero on $\bigsqcup_jI_j$ in the harmonic problem,
uniqueness makes the resulting circle warping $\sigma_{3,p}$-invariant.
For the quotient
\begin{equation*}
 W_{3,p}=(\mathbb R\times P_{3,p})/
 \langle(t,x)\mapsto(t+T,\sigma_{3,p}x)\rangle,
\end{equation*}
the first return time is $T$ on the outer boundary and $pT$ on any chosen
inner component.  Thus the distinguished loop classes satisfy
\begin{equation}
 [\lambda_i]=[\lambda_o]^p,
 \label{eq:appb-loop-relation}
\end{equation}
and the boundary data in the notation of
\eqref{eq:un-boundary-metric-definition} are
\begin{equation*}
 \ell_o=\frac{T\rho_*(1+\alpha)}{2\pi},
 \quad \omega_o=\frac1p,
 \qquad
 \ell_{\mathrm{in}}=\frac{pT\rho_*}{2\pi},
 \quad \omega_{\mathrm{in}}=0,
 \qquad
 \ell_o=\frac{1+\alpha}{p}\ell_{\mathrm{in}}.
\end{equation*}
For the holonomy interpolation cylinder take, in dimension $m=3$,
\begin{equation*}
 \omega_p(s)=\frac1p\left(
  1-\frac{\displaystyle\int_0^s\ell(t)\rho(t)^{-5}\dd t}
          {\displaystyle\int_0^{s_*}\ell(t)\rho(t)^{-5}\dd t}
 \right).
\end{equation*}
Then $(\rho^5\omega_p'/\ell)'=0$, so the same mixed Ricci terms vanish.
Relative to the order-two function, $\omega_p'=(2/p)\omega_2'$, hence
the quadratic twisting errors and the boundary cross terms are no larger
than those already estimated.  The gluing, completeness, one-endedness,
and fundamental-group arguments therefore proceed as in the dyadic case,
with the square relation replaced by~\eqref{eq:appb-loop-relation}.
\end{remark}

\bibliographystyle{amsplain}
\bibliography{Milnorreferences}

\end{document}